\documentclass{amsart}
\usepackage{amssymb}
\usepackage{amsmath, amscd}
\usepackage{amsthm}
\usepackage{mathtools}
\usepackage{mathrsfs}
\usepackage{booktabs}
\usepackage{perpage}
\usepackage[all,cmtip]{xy}
\usepackage{setspace}
\usepackage{amsbsy}
\usepackage{url}
\usepackage[T1]{fontenc}
\usepackage{verbatim}
\usepackage{mathdots}
\usepackage{mathrsfs}
\usepackage{ifthen}
\usepackage{blkarray}
\usepackage{multirow}
\usepackage{enumerate}
\usepackage{enumitem}
\usepackage[foot]{amsaddr}
\usepackage[left=2cm,head=13pt,
top=1.5cm,right=2cm, bottom=1.5cm]{geometry}
\usepackage{pgf,tikz}\usetikzlibrary{arrows}

\AtEndDocument{\bigskip{\footnotesize%

\noindent (Jean-Stefan Koskivirta) \textsc{Department of Mathematics, University of Tokyo} \par 
\noindent \texttt{jeanstefan.koskivirta@gmail.com}
}}

\usepackage{lipsum}

\numberwithin{equation}{subsection}

\newtheorem{theorem}{Theorem}[subsection]
\newtheorem{lemma}[theorem]{Lemma}
\newtheorem{conjecture}[theorem]{Conjecture}

\newtheorem{corollary}[theorem]{Corollary}
\newtheorem{definition}[theorem]{Definition}
\newtheorem{question}[theorem]{Question}

\newtheorem{proposition}[theorem]{Proposition}

\newtheorem{assumption}[theorem]{Assumption}

\newtheorem*{question1}{Question 1}

\newtheorem*{thm1}{Theorem 1}
\newtheorem*{thm2}{Theorem 2}
\newtheorem*{thm3}{Theorem 3}

\newtheorem*{cor1}{Corollary 1}
\newtheorem*{cor2}{Corollary 2}

\newtheorem*{conj1}{Conjecture 1}
\newtheorem*{conj2}{Conjecture 2}

\newtheorem*{prop1}{Proposition 1}

\theoremstyle{remark}

\newtheorem{rmk}[theorem]{Remark}

\title[zip-cone]{Automorphic forms on the stack of $G$-zips}

\DeclarePairedDelimiter\floor{\lfloor}{\rfloor}

\newcommand{\GZip}{\mathop{\text{$G$-{\tt Zip}}}\nolimits}

\newcommand{\GF}{\mathop{\text{$G$-{\tt ZipFlag}}}\nolimits}

\newskip\procskipamount
\newskip\interskipamount
\newskip\refskipamount
\newcommand{\procskip}{\vskip\procskipamount}
\newcommand{\interskip}{\vskip\interskipamount}
\newcommand{\refskip}{\vskip\refskipamount}

\newcommand{\procbreak}{\par
   \ifdim\lastskip<\procskipamount\removelastskip
   \penalty-100
   \procskip\fi
   \noindent\ignorespaces}

\newcommand{\titlebreak}{\par%
\ifdim\lastskip<\interskipamount\removelastskip%
\penalty10000%
\interskip\fi%
\noindent}%

\newcommand{\interbreak}{\par%
\ifdim\lastskip<\interskipamount\removelastskip%
\penalty-100%
\interskip\fi%
\noindent\ignorespaces}%

\newcommand{\refbreak}{\par%
\ifdim\lastskip<\refskipamount\removelastskip%
\penalty-100%
\refskip\fi%
\noindent\ignorespaces}%

\newcounter{listcounter}
\newcounter{deflistcounter}
\newcounter{equivcounter}

\newskip{\itemsepamount}
\newskip{\topsepamount}
\newenvironment{assertionlist}{%
  \begin{list}
    {\upshape (\arabic{listcounter})}
    {\setlength{\leftmargin}{18pt}
     \setlength{\rightmargin}{0pt}
     \setlength{\itemindent}{0pt}
     \setlength{\labelsep}{5pt}
     \setlength{\labelwidth}{13pt}
     \setlength{\listparindent}{\parindent}
     \setlength{\parsep}{0pt}
     \setlength{\itemsep}{\itemsepamount}
     \setlength{\topsep}{\topsepamount}
     \usecounter{listcounter}}}
  {\end{list}}

\newenvironment{definitionlist}{%
  \begin{list}
    {\upshape (\alph{deflistcounter})}
    {\setlength{\leftmargin}{18pt}
     \setlength{\rightmargin}{0pt}
     \setlength{\itemindent}{0pt}
     \setlength{\labelsep}{5pt}
     \setlength{\labelwidth}{13pt}
     \setlength{\listparindent}{\parindent}
     \setlength{\parsep}{0pt}
     \setlength{\itemsep}{\itemsepamount}
     \setlength{\topsep}{\topsepamount}
     \usecounter{deflistcounter}}}
  {\end{list}}

\newenvironment{equivlist}{%
  \begin{list}
    {\upshape (\roman{equivcounter})}
    {\setlength{\leftmargin}{18pt}
     \setlength{\rightmargin}{0pt}
     \setlength{\itemindent}{0pt}
     \setlength{\labelsep}{5pt}
     \setlength{\labelwidth}{13pt}
     \setlength{\listparindent}{\parindent}
     \setlength{\parsep}{0pt}
     \setlength{\itemsep}{\itemsepamount}
     \setlength{\topsep}{\topsepamount}
     \usecounter{equivcounter}}}
  {\end{list}}

% ...

\newcommand{\Bcal}{{\mathcal B}}

\newcommand{\Fcal}{{\mathcal F}}

\newcommand{\Mcal}{{\mathcal M}}

\newcommand{\Ocal}{{\mathcal O}}
\newcommand{\Pcal}{{\mathcal P}}

\newcommand{\Rcal}{{\mathcal R}}
\newcommand{\Scal}{{\mathcal S}}

\newcommand{\Vcal}{{\mathcal V}}

\newcommand{\Xcal}{{\mathcal X}}
\newcommand{\Ycal}{{\mathcal Y}}
\newcommand{\Zcal}{{\mathcal Z}}

\newcommand{\gfr}{{\mathfrak g}}

\newcommand{\pfr}{{\mathfrak p}}

\newcommand{\Sfr}{{\mathfrak S}}

\renewcommand{\AA}{\mathbf{A}}

\newcommand{\CC}{\mathbf{C}}

\newcommand{\FF}{\mathbf{F}}
\newcommand{\GG}{\mathbf{G}}

\newcommand{\NN}{\mathbf{N}}

\newcommand{\QQ}{\mathbf{Q}}
\newcommand{\RR}{\mathbf{R}}

\newcommand{\ZZ}{\mathbf{Z}}

\DeclareMathOperator{\Pic}{Pic}

\newcommand{\Ascr}{{\mathscr A}}

\newcommand{\Escr}{{\mathscr E}}

\newcommand{\Lscr}{{\mathscr L}}

\newcommand{\Vscr}{{\mathscr V}}

\DeclareMathOperator{\Gal}{Gal}

\DeclareMathOperator{\rank}{rank}
\DeclareMathOperator{\Span}{Span}

\DeclareMathOperator{\Lie}{Lie}

\DeclareMathOperator{\Stab}{Stab}

\DeclareMathOperator{\Ker}{Ker}

\DeclareMathOperator{\rk}{rk}

\DeclareMathOperator{\Schub}{Sbt}
\DeclareMathOperator{\Sbt}{Sbt}

\DeclareMathOperator{\spec}{Spec}

\DeclareMathOperator{\Sch}{Sbt}

\DeclareMathOperator{\zip}{zip}
\DeclareMathOperator{\GS}{GS}

\DeclareMathOperator{\pol}{pol}

\newcommand{\id}{{\rm Id}}

\newcommand{\Th}{{\rm Th.}}

\newcommand{\Rmk}{{\rm Rmk.}}

\newcommand{\Cor}{{\rm Cor.}}

\newcommand{\Lem}{{\rm Lem.}}

\newcommand{\Conj}{{\rm Conj.}}

\newcommand{\Def}{{\rm Def.}}

\newcommand{\Prop}{{\rm Prop.}}

\newcommand{\loccit}{{\em loc.\ cit. }}
\newcommand{\loccitn}{{\em loc.\ cit.}}

\newcommand{\diag}{{\rm diag}}

\renewcommand{\Im}{{\rm Im}}

\DeclareMathOperator{\Std}{Std}

\renewcommand{\div}{{\rm div}}

\DeclareMathOperator{\Det}{det}
\DeclareMathOperator{\Vol}{Vol}
\DeclareMathOperator{\iden}{id}
\DeclareMathOperator{\Norm}{Norm}
\DeclareMathOperator{\Ind}{Ind}

\author{Jean-Stefan Koskivirta}

\begin{document}

\date{\today}

\pagestyle{plain}

\begin{abstract}
We define automorphic vector bundles on the stack of $G$-zips introduced by Moonen-Pink-Wedhorn-Ziegler and study their global sections. In particular, we give a combinatorial condition on the weight for the existence of nonzero mod $p$ automorphic forms on Shimura varieties of Hodge-type. We attach to the highest weight of the representation $V(\lambda)$ a mod $p$ automorphic form and we give a modular interpretation of this form in some cases.
\end{abstract}

\maketitle

\section*{Introduction}

Following our previous article \cite{Goldring-Koskivirta-Diamond-I}, we continue to investigate mod $p$ automorphic forms. One possible definition of such automorphic forms is the following. Let $p$ be a prime number and let $\Scal_K$ denote the Kisin-Vasiu integral model of a Hodge-type Shimura variety with hyperspecial level $K$ at $p$, defined over a ring $\Ocal$. Let $G$ denote the reductive group over $\QQ$ attached to $\Scal_K$ and let $T\subset G$ be a maximal torus (chosen appropriately). For each character $\lambda \in X^*(T)$, one defines an automorphic vector bundle $\Vscr(\lambda)$ on $\Scal_K$ attached to $\lambda$. If $R$ is any $\Ocal$-algebra, global sections of $\Vscr(\lambda)$ over $\Scal_K\otimes R$ are called automorphic forms of level $K$ and weight $\lambda$ over $R$. When $R$ is a field of characteristic $p$, we call them mod $p$ automorphic forms. A natural question is the following:

\begin{question1}
For which $\lambda \in X^*(T)$ are there non-zero automorphic forms over $R$ ?
\end{question1}

The two cases that we consider are $R=\CC$ and $R=\overline{\FF}_p$. The set of such $\lambda$ forms a cone inside $X^*(T)$, denoted by $C_{K}(R)$. This problem is part of a more general project to establish cohomology vanishing results of the type $H^i(S_K,\Vscr_K(\lambda))=0$ for automorphic vector bundles. This question is developed in our upcoming paper \cite{Brunebarbe-Goldring-Koskivirta-Stroh-ampleness}, with applications to lifting mod $p$ automorphic forms to characteristic zero ones.

In this paper, we consider the case $i=0$. The answer to Question 1 depends of course on the level $K$, as one can see already with modular forms. However, define the saturated cone $\langle C_K(R) \rangle$ as the set of $\lambda \in X^*(T)$ such that $m\lambda \in C_K(R)$ for some $m\geq 1$. Since the change of level maps are finite \'{e}tale, the saturated cone $\langle C_K(R) \rangle$ is independent of $K$. In characteristic zero, it is expected that this coarser question has a simple answer. Denote by $\Phi_+$ is the set of positive roots, $\Delta$ the simple roots, and let $I\subset \Delta$ be the type of the parabolic subgroup of $G$ attached to the Shimura datum (the stabilizer of the Hodge filtration). The work of Griffiths-Schmid (\cite{Griffiths-Schmid-homogeneous-complex-manifolds}) suggests that $\langle C_K(\CC) \rangle $ should be given as follows:
\begin{equation}\label{GSintro} \tag{1}
\langle C_K(\CC) \rangle=\left\{\lambda \in X^*(T) \ | \
\begin{aligned}
&\langle \lambda, \alpha^\vee \rangle \geq 0 \ \textrm{ for }\alpha\in I \\
&\langle \lambda, \alpha^\vee \rangle \leq 0 \ \textrm{ for }\alpha\in \Phi_+ \setminus I
\end{aligned}
 \right\}.
\end{equation}
For the modular curve, this is just the set of nonnegative integers, as expected. In characteristic $p$ however, very little is known. In this paper, we give a conjectural answer to the characteristic $p$ case.

The stack of $G$-zips has proved to be a useful tool to study questions related to Shimura varieties, as a number of recent papers have demonstrated (\cite{Goldring-Koskivirta-Strata-Hasse-v2}, \cite{Goldring-Koskivirta-zip-flags}, \cite{Goldring-Koskivirta-Diamond-I}). It is a finite smooth stack denoted by $\GZip^\mu$, where $\mu$ pertains to the cocharacter attached to the Shimura datum. A fundamental property is the existence of a smooth map of algebraic stacks $\zeta:S_K \to \GZip^\mu$, where $S_K:=\Scal_K\otimes \overline{\FF}_p$. The fibers of $\zeta$ define the Ekedahl-Oort stratification of $S_K$ (\cite{ZhangEOHodge}). For each $\lambda \in X^*(T)$, there exists a vector bundle $\Vscr(\lambda)$ on $\GZip^\mu$ whose pull-back by $\zeta$ is the automorphic vector bundle $\Vscr(\lambda)$ introduced before. It is thus natural to determine which mod $p$ automorphic forms can be obtained from the stack $\GZip^\mu$ by pullback. For example, the papers \cite{Koskivirta-Wedhorn-Hasse} and \cite{Goldring-Koskivirta-Strata-Hasse-v2} construct certain sections called Hasse invariants on this stack, with useful geometric and arithmetic applications. Hence we want to understand how much information is encoded in the group-theoretical object $\GZip^\mu$. One advantage of this method is that all sections that are produced in this way are automatically Hecke-equivariant.

Similarly to the definition of $C_K$, denote by $C_{\rm zip}$ the cone of characters $\lambda \in X^*(T)$ such that $\Vscr(\lambda)$ admits nonzero global sections on $\GZip^\mu$. The map $\zeta$ induces an injection $H^0(\GZip^\mu, \Vscr(\lambda))\to H^0(S_K,\Vscr(\lambda))$ and hence an inclusion $C_{\rm zip}\subset C_{K}(\overline{\FF}_p)$. We conjectured in \cite{Goldring-Koskivirta-Diamond-I}:

\begin{conj1}\label{conjShzipcone}
One has $\langle C_{\rm zip} \rangle = \langle C_{K}(\overline{\FF}_p)\rangle $.
\end{conj1}

Note that the first cone is a much more tractable, group-theoretical object, whereas the second cone is related to the geometry of the Shimura variety, and more generally to number theory via automorphic representations. We gave evidence for this conjecture by proving it for groups $G$ which are $\FF_p$-forms of $GL_2^n$, unitary groups of signature $(2,1)$ at a split prime, and $GSp(4)$. Note also that in general the equality $C_{\rm zip} = C_{K}(\overline{\FF}_p)$ does not hold (as the latter depends on $K$). The conjecture says that the coarser information about the saturated cone $\langle C_{K}(\overline{\FF}_p)\rangle$ is encoded in the stack $\GZip^\mu$. In this paper, we will not discuss Conjecture 1 (for which we refer the interested reader to \cite{Goldring-Koskivirta-Diamond-I}). Instead, one objective of this paper is to determine $C_{\rm zip}$ and $\langle C_{\rm zip} \rangle$ for general groups. This will at least give a lower bound for $C_K(\overline{\FF}_p)$. The general expected phenomenon is that the cone $C_{\rm zip}$ "converges" towards $C_{\GS}$ as $p$ goes to infinity.

The determination of $C_{\rm zip}$ is still a difficult problem. One of the main results of this paper is a description of $H^0(\GZip^\mu,\Vscr(\lambda))$ in terms of the representation theory of the $L$-representation $V(\lambda)=\Ind^P_B(\lambda)$ (see \eqref{repV}). Here we assume that $L$ is defined over $\FF_p$ and that $T$ is $\FF_p$-split. We define a subspace $V(\lambda)_{\leq 0}\subset V(\lambda)$ as the sum of certain weight spaces (see \S\ref{sec-zipreptheo}).

\begin{thm1}[\Th~\ref{thminter}]
There is an isomorphism $H^0(\GZip^\mu,\Vscr(\lambda))\simeq V(\lambda)_{\leq 0}\cap V(\lambda)^{L(\FF_p)}$.
\end{thm1}

In particular, the cone $C_{\zip}$ coincides with the set of $\lambda\in X^*(T)$ such that $V(\lambda)_{\leq 0}\cap V(\lambda)^{L(\FF_p)}\neq 0$. This question is now formulated in terms of representation theory of $V(\lambda)$, for which many tools are available. We have the following obvious corollary:

\begin{cor1}
Let $S_K$ be the special fiber of a Shimura variety of Hodge-type at a place of good reduction. For $\lambda\in X^*(T)$, one has an inclusion
\begin{equation*}
V(\lambda)_{\leq 0}\cap V(\lambda)^{L(\FF_p)}\subset H^0(S_K,\Vscr(\lambda)).
\end{equation*}
\end{cor1}

For the groups $Sp(4)$ and $Sp(6)$, we are able to determine completely $C_{\zip}$ and $\langle C_{\rm zip}\rangle$ (see below). For general groups, we show the following partial result. Denote by $W$ (resp. $W_L$) the Weyl group of $G$ (resp. $L$) and by $\ell:W\to \NN$ the length function with respect to $\Delta$. Assume that $L$ is defined over $\FF_p$. For $\alpha \in \Delta$, let $r_\alpha$ the smallest integer $r\geq 1$ such that $\alpha$ is defined over $\FF_{p^r}$.

\begin{thm2}[\Cor~\ref{lamisin}]
Let $\lambda\in X^*_{+,L}(T)$ be an $L$-dominant character. Assume that for all $\alpha \in \Delta\setminus I$, the following holds:
\begin{equation}\label{zipintro}\tag{2}
\sum_{w\in W_L(\FF_p)} \sum_{i=0}^{r_\alpha-1} p^{i+\ell(w)} \ \langle w\lambda, \sigma^i\alpha^\vee \rangle\leq 0, \quad \forall \alpha \in \Delta \setminus I.
\end{equation}
Then $\lambda \in \langle C_{\rm zip}\rangle$. In particular, there exists a non-zero mod $p$ automorphic form of weight $m\lambda$ for some $m\geq 1$.
\end{thm2}

Furthermore, one can give an explicit integer $m$ satisfying $H^0(\GZip^\mu,\Vscr(m\lambda))\neq 0$. \Th~2 illustrates the difference with the characteristic 0 situation. Indeed, the Griffiths-Schmid conditions \eqref{GSintro} imply that each of the summands in \eqref{zipintro} is $\leq 0$. Hence \eqref{zipintro} is a much weaker condition than \eqref{GSintro}. In particular, a lot of these mod $p$ automorphic forms do not lift to characteristic zero.

Also, \Th~2 describes only part of the cone $\langle C_{\rm zip}\rangle$, which we call the "highest weight cone" (we explain the terminology later, making a link to representation theory). In general, there exist global sections whose weights do not satisfy the inequality of \Th~2, which are hence even "further away" from the characteristic zero possible weights. This part of the cone remains mysterious in general, but we can give some good approximations by several natural subcones. A particular case of interest is the Siegel modular variety, attached to the group $G=Sp(2n)$ (actually one should work with $GSp(2n)$, but this makes no difference for these questions). In this case, we obtain the following corollary:

\begin{cor2}\label{thm-intro}
In the case $G=Sp(2n)$, identify $X^*(T)\simeq \ZZ^n$ in the usual way. Let $\lambda=(a_1,...,a_n)\in \ZZ^n$ and assume that the inequality $\sum_{i=1}^n p^{-i}a_i \leq 0$ holds. Then $\lambda \in \langle C_{\rm zip}\rangle$. In particular, this gives a non-zero mod $p$ Hecke-equivariant automorphic form of weight $m\lambda$ for some (explicit) $m\geq 1$.
\end{cor2}

The sections constructed to prove Th. 2 have a highly complicated vanishing locus, in contrast to the sections constructed in previous papers \cite{Goldring-Koskivirta-Strata-Hasse-v2}, \cite{Goldring-Koskivirta-zip-flags}, \cite{Goldring-Koskivirta-Diamond-I}. We compute the cone $\langle C_{\rm zip}\rangle$ in the Siegel case (see \S\ref{sec-GScone}), where $G=Sp(2n)$ for $n=2,3$:

\begin{thm3}[\Cor~\ref{corSp4}, \Cor~\ref{cor2Sp4}, \Prop~\ref{propconeSp6}] Let $G$ be the group $G=Sp(2n)$ endowed with the usual Hodge-type zip datum.
\begin{enumerate}
\item For $n=2$, one has $C_{\zip}=\NN (1,-p)+\NN(1-p,1-p)+\NN(0,-p(p-1))$ and $\langle C_{\zip} \rangle$ is the set of pairs $(a_1,a_2)\in \ZZ^2$ satisfying the inequalities
\begin{align*}
& a_1\geq a_2  \\
& p a_1+a_2\leq 0  
\end{align*}
\item For $n=3$, the cone $\langle C_{\zip} \rangle$ is the set of $(a_1,a_2,a_3)\in \ZZ^3$ satisfying the inequalities
\begin{align*}
& a_1\geq a_2 \geq a_3  \\
& p^2 a_1+pa_3+a_2\leq 0  \\
& p^2 a_2+pa_1+a_3\leq 0
\end{align*}
\end{enumerate}
\end{thm3}

Let us now explain another direction of this paper. Let $S_K(G,X)$ be a Hodge-type Shimura variety with reflex field $F$ which has good reduction at a prime $p$. Write $\Scal_K$ for the Kisin-Vasiu integral canonical model over $\Ocal_{F,\pfr}$, where $\pfr$ is a prime of $F$ dividing $p$. For an $\Ocal_{F,\pfr}$-algebra $\Ascr$, one can define a ring of automorphic forms over $\Ascr$ by taking the following direct sum: 
\begin{equation*}
R_{K}(\Ascr):=\bigoplus_{\lambda \in X^*(T)} H^0(\Scal_K\otimes \Ascr, \Vscr(\lambda)).
\end{equation*}
This space inherits a structure of $\Ascr$-algebra via the natural maps $\Vscr(\lambda)\otimes \Vscr(\nu)\to \Vscr(\lambda+\nu)$ for all $\lambda,\nu\in X^*(T)$. This algebra captures the whole family of automorphic forms over $\Ascr$ for all weights, and is therefore a highly interesting object. The formation of $R_K(\Ascr)$ is functorial in $\Ascr$. Similarly, we may define an analog for the stack of $G$-zips:
\begin{equation*}
R_{\rm zip}:=\bigoplus_{\lambda \in X^*(T)} H^0(\GZip^\mu, \Vscr(\lambda)).
\end{equation*}
Again, $R_{\rm zip}$ is an algebra over the residue field $\kappa$ of $\pfr$, but we will implicitly base change to $\overline{\FF}_p$. The map $\zeta$ yields an injection $\zeta^*:R_{\rm zip} \to R_{K}(\overline{\FF}_p)$, which commutes with change of level $K$. In particular, the elements of $R_{\rm zip}$ commute with Hecke operators.

The graded algebras $R_{\rm zip}$ and $R_K:=R_K(\overline{\FF}_p)$ carry much finer information that the cones $C_{\rm zip}$ and $C_K:=C_K(\overline{\FF}_p)$. These cones can be recovered as the gradings of the graded algebras $R_{\rm zip}$ and $R_K$, respectively. We conjecture the following:
\begin{conj2}\label{Rzipfg}
The $\overline{\FF}_p$-algebra $R_{\rm zip}$ is finitely generated.
\end{conj2}
We will see that this conjecture is related to Hilbert's 14th problem. In particular, the conjecture implies that the cone $C_{\rm zip}$ is finitely generated as a monoid. We were able to prove some cases of the conjecture and determine $R_{\zip}$ explicitly in these cases. More specifically, we show:
\begin{prop1}[\Th~\ref{thmRzipSp4}]
Assume $G=Sp(4)$, endowed with the usual Hodge-type zip datum. The $k$-algebra $R_{\zip}$ is isomorphic to a polynomial ring in three indeterminates.
\end{prop1}
We give explicit mod $p$ automorphic forms that generate $R_{\rm zip}$, and we give a modular interpretation of these sections. An interesting question is to determine the structure of $R_{K}$ as an $R_{\rm zip}$-algebra. For example, one can ask whether it is smooth, finitely generated, etc. 

\section*{Acknowledgements}
We would like to thank Wushi Goldring for very useful discussions. This work is a continuation of a joint project with him. We are also grateful to David Helm for helpful advice. Finally, we would like to thank the referee for his remarks and suggestions on a first version of the paper.

\section{Global sections of automorphic vector bundles on the stack of $G$-zips}
\subsection{The stack of $G$-zips}
\subsubsection{Zip datum} \label{review}
We fix an algebraic closure $k$ of $\FF_p$. For an $\FF_p$-scheme $X$, we denote by $X^{(p)}$ its Frobenius twist and by $\varphi:X\to X^{(p)}$ its relative Frobenius. Let $G$ be a connected reductive group over $\FF_p$ and $\varphi:G\to G$ the $p$-th power Frobenius homomorphism. Let $\Zcal:=(G,P,L,Q,M,\varphi)$ be a zip datum, i.e $P,Q$ are parabolic subgroups of $G_k$, $L\subset P$ and $M\subset Q$ are Levi subgroups such that $\varphi(L)=M$. The zip group is the subgroup of $P\times Q$ defined by
\begin{equation}
E:=\{(x,y)\in P\times Q, \ \varphi(\overline{x})=\overline{y}\}
\end{equation}
where $\overline{x}\in L$ and $\overline{y}\in M$ denote the Levi components of the elements $x$ and $y$ respectively. Let $G\times G$ act on $G$ by $(a,b)\cdot g:=agb^{-1}$, let $E$ act on $G$ by restricting this action. The stack of $G$-zips of type $\Zcal$ (\cite{Pink-Wedhorn-Ziegler-zip-data},\cite{PinkWedhornZiegler-F-Zips-additional-structure}) is isomorphic to the quotient stack
\begin{equation}
\GZip^\Zcal \simeq \left[E\backslash G\right].
\end{equation}
In this paper, a frame is a triple $(B,T,z)$ where $(B,T)$ is a Borel pair of $G$ defined over $\FF_p$ such that $B\subset P$ and $z$ is an element of the Weyl group $W=W(G,T)$ satisfying the conditions:
\begin{equation}\label{eqBorel}
{}^z \! B \subset Q \quad \textrm{and} \quad
\varphi(B\cap L) = B\cap M= {}^z \! B\cap M. 
\end{equation}

\subsubsection{Zip stratification} \label{zipstrata}
Fix a frame $(B,T,z)$. Let $\Phi\subset X^*(T)$ be the set of $T$-roots in $G$. Let $\Phi^+$ be the system of positive roots given by putting $\alpha \in \Phi^+$ when the $\alpha$-root group $U_{-\alpha}$ is contained in $B$. Write $\Delta\subset \Phi_+$ for the set of simple roots. For $\alpha \in \Phi$, let $s_\alpha \in W$ be the corresponding reflection. Then $(W,\{s_\alpha, \alpha \in \Delta\})$ is a Coxeter group and we denote by $\ell :W\to \NN$ the length function. We denote by $w_0$ the longest element of $W$.

For $K\subset \Delta$, let $W_K$ denote the subgroup of $W$ generated by $\{s_\alpha, \ \alpha \in K\}$, and let $w_{0,K}$ be its longest element. Let ${}^KW$ (resp. $W^K$) denote the subset of elements $w\in W$ which have minimal length in the coset $W_K w$ (resp. $wW_K$). Then ${}^K W$ (resp. $W^K$) is a set of representatives of $W_K\backslash W$ (resp. $W/W_K$). Let $I\subset \Delta$ (resp. $J\subset \Delta$) be the type of $P$ (resp. $Q$).  We denote by $X_{+,I}^*(T)$ the set of characters $\lambda\in X^*(T)$ such that $\langle \lambda,\alpha^\vee \rangle \geq 0$ for $\alpha \in I$, and call them $L$-dominant or $I$-dominant characters.

For $w\in W$, fix a representative $\dot{w}\in N_G(T)$, such that $(w_1w_2)^\cdot = \dot{w}_1\dot{w}_2$ whenever $\ell(w_1 w_2)=\ell(w_1)+\ell(w_2)$ (this is possible by choosing a Chevalley system, \cite[ XXIII, \S6]{SGA3}). For $w\in W$, define $G_w$ as the $E$-orbit of $\dot{w}\dot{z}^{-1}$. The $E$-orbits in $G$ form a stratification of $G$ by locally closed subsets. By \Th~7.5 and \Th~11.2 in \cite{Pink-Wedhorn-Ziegler-zip-data}, the map $w\mapsto G_w$ restricts to two bijections:
\begin{align} \label{orbparam}
{}^I W &\to \{E \textrm{-orbits in }G\}\\  
\label{dualorbparam} W^J &\to \{E \textrm{-orbits in }G\}
\end{align}
Furthermore, for $w\in {}^I W \cup W^J$, one has $\dim(G_w)= \ell(w)+\dim(P)$. In particular, there is a unique open $E$-orbit $U_\Zcal\subset G$ corresponding to the longest elements $w_{0,I}w_0\in {}^I W$ via \eqref{orbparam} and $w_0w_{0,J}\in W^J$ via \eqref{dualorbparam}.

\subsubsection{Cocharacters} \label{subsec-cochar}
A cocharacter datum is a pair $(G,\mu)$ where $G$ is a reductive group over $\FF_p$ and $\mu:\GG_{m,k}\to G_k$ is a cocharacter. Such a pair $(G,\mu)$ gives rise to a zip datum as follows. First, $\mu$ yields a pair of opposite parabolics $P_\pm(\mu)$ such that $P_+(\mu)\cap P_-(\mu)=L(\mu)$ is the centralizer of $\mu$. The parabolic $P_+(\mu)$ consists of elements $g\in G$ such that the limit $\lim_{t\to 0}\mu(t)g\mu(t)^{-1}$ exists, i.e such that the map $\GG_{m,k} \to G_{k}$,  $t\mapsto\mu(t)g\mu(t)^{-1}$ extends to a morphism of varieties $\AA_{k}^1\to G_{k}$. The unipotent radical of $P_+(\mu)$ is the set of such elements $g$ for which this limit is $1\in G$.

Set $P:=P_-(\mu)$, $Q:=( P_+(\mu))^{(p)}$, $L:=L(\mu)$ and $M:= (L(\mu))^{(p)}$. The tuple $\Zcal_\mu:=(G,P,L,Q,M,\varphi)$ is the zip datum attached to the cocharacter $\mu$. In the following we will consider mostly $G$-zips arising in this way. We write simply $\GZip^\mu$ for $\GZip^{\Zcal_\mu}$.



\subsection{Automorphic vector bundles on $\GZip^\Zcal$} \label{sec-vector-bundles-gzipz}
Fix a zip datum $\Zcal=(G,P,L,Q,M,\varphi)$. An algebraic representation $(\rho,V)$ of $E$ (where $V$ is a finite-dimensional $k$-vector space) gives rise naturally to a vector bundle $\Vscr(\rho)$ on the stack $\GZip^\Zcal\simeq [E\backslash G]$ by the associated sheaf construction (\cite[\S5.8]{jantzen-representations}). The space of global sections of $\Vscr(\rho)$ is
\begin{equation}\label{secequa}
H^0(\GZip^\Zcal,\Vscr(\rho))=\left\{f:G\to V, \ f(\epsilon \cdot g)=\epsilon f(g)  \right\}, \quad \forall\epsilon\in E, \ \forall g\in G.
\end{equation}

If $(V,\rho)$ is an algebraic representation of $P$, we view it as an $E$-representation via the first projection $E\to P$. We denote again the attached vector bundle on $\GZip^\Zcal$ by $\Vscr(\rho)$.

\begin{lemma}
Let $(V,\rho)$ be an $E$-representation. The $k$-vector space $H^0(\GZip^\Zcal,\Vscr(\rho))$ is finite-dimensional, of dimension $\leq \dim(V)$.
\end{lemma}

\begin{proof}
Let $U_\Zcal \subset G$ denote the unique Zariski open $E$-orbit. A function $f:G\to V$ satisfying the property of equation \eqref{secequa} is uniquely determined by its restriction to $U_\Zcal$ (by density and reducedness), and hence also by its value at a fixed element $g\in U_\Zcal$ (by $E$-equivariance). Hence the dimension of $H^0(\GZip^\Zcal,\Vscr(\rho))$ is bounded above by the dimension of $V$.
\end{proof}

For each character $\lambda\in X^*(T)$, denote by $\Lscr(\lambda)$ the line bundle attached to $\lambda$ on the variety $L/ B_L$, again by the associated sheaf construction (\cite[\S5.8]{jantzen-representations}). We obtain an $L$-representation 
\begin{equation}\label{repV}
V(\lambda):=H^0(L/ B_L,\Lscr(\lambda)).
\end{equation}
It is the induced representation $\Ind_{B_L}^{L}(\lambda)$. In general, $V(\lambda)$ is not irreducible, but it has a unique irreducible $L$-subrepresentation. We obtain a vector bundle on $\GZip^\Zcal$ which we denote by $\Vscr(\lambda)$. We call $\Vscr(\lambda)$ an \emph{automorphic vector bundle} on $\GZip^\Zcal$. When $\lambda \in X^*(L)$, the vector bundle $\Vscr(\lambda)$ is a line bundle. Note that if $\lambda\in X^*(T)$ is not $L$-dominant, then $\Vscr(\lambda)=0$. In this paper, we ask the following question:
\begin{question}\label{zipquestion}
For which $\lambda\in X^*(T)$ does the vector bundle $\Vscr(\lambda)$ admit nonzero global sections on $\GZip^\Zcal$ ?
\end{question}
In view of Question \ref{zipquestion}, it is relevant to define the following subset 
\begin{equation}
C_{\rm zip}:=\left\{ \lambda\in X^*(T)  \ | \  H^0(\GZip^\Zcal, \Vscr(\lambda))\neq 0 \right\}.
\end{equation}
We will prove later that $C_{\zip}$ is a cone in $X^*(T)$ (i.e an additive submonoid of $X^*(T)$). Since $\Vscr(\lambda)=0$ whenever $\lambda$ is not $L$-dominant, we have an obvious inclusion $C_{\zip}\subset X^*_{+,I}(T)$.

\subsection{The stack of $G$-zip flags} \label{sec-zipflag}
The stack of $G$-zip flags (attached to $(\Zcal,B)$), denoted by $\GF^\Zcal$, was defined in \cite{Goldring-Koskivirta-zip-flags}. It parametrizes $G$-zips endowed with a full flag refining the Hodge filtration (\loccitn, \Def~3.1.1). By \Th~3.1.2 of \loccitn, it is isomorphic to the quotient stack $\left[E'\backslash G\right]$ where $E'$ is the subgroup of $E$ defined by $E':=E\cap (B\times G)$, where the subgroup $E'$ acts on $G$ by restricting the action of $E$ on $G$ defined in \S\ref{review}. The inclusion $E'\subset E$ induces a natural projection
\begin{equation}
\pi:\GF^\Zcal \to \GZip^\Zcal
\end{equation}
whose fibers are flag varieties, isomorphic to $E'\backslash E \simeq B \backslash P$.

Using again the associated sheaf construction (\cite[\S5.8]{jantzen-representations}), we can attach to each $\lambda \in X^*(T)$ a line bundle $\Lscr(\lambda)$ on the quotient stack $[E'\backslash G]\simeq\GF^\Zcal$. One has the formula
\begin{equation}\label{linebndmult}
\Lscr(\lambda) \otimes \Lscr(\nu)=\Lscr(\lambda+\nu).
\end{equation}
for all $\lambda,\nu\in X^*(T)$. Furthermore, we have
\begin{equation}\label{pushforward}
\pi_*(\Lscr(\lambda))=\Vscr(\lambda).
\end{equation}
In particular, this formula implies:
\begin{equation}
H^0(\GZip^\Zcal,\Vscr(\lambda))=H^0(\GF^\Zcal,\Lscr(\lambda)).
\end{equation}

\begin{rmk}\label{equivdivrmk}
It follows that elements of $H^0(\GZip^\Zcal, \Vscr(\lambda))$ identify with regular functions $f:G\to \AA^1$ satisfying the relation $f(\varepsilon\cdot g)=\lambda(\varepsilon)f(g)$ for all $\varepsilon\in E'$ and $g\in G$. Conversely, let $f:G\to \AA^1$ be a regular function. By a theorem of Rosenlicht, the following are equivalent
\begin{enumerate}
\item There exists $\lambda \in X^*(T)$ such that $f\in H^0(\GZip^\Zcal, \Vscr(\lambda))$.
\item The Weil divisor $\div(f)$ is stable by $E'$.
\end{enumerate}
\end{rmk}

\begin{rmk}\label{secprodrmk}
In particular, \Rmk~\ref{equivdivrmk} has the following consequence: Assume $f:G\to \AA^1$ is an element of $H^0(\GZip^\Zcal,\Vscr(\lambda))$. Assume that $f=f_1 f_2$ where $f_1,f_2:G\to \AA^1$ are regular functions. Then there exists characters $\lambda_1,\lambda_2\in X^*(T)$ such that $f_i\in H^0(\GZip^\Zcal,\Vscr(\lambda_i))$ for $i=1,2$ and $\lambda=\lambda_1+\lambda_2$. 
\end{rmk}

\subsection{Cone terminology}\label{coneter}
Let $M$ be a free abelian group of finite-type. In this paper, a \emph{cone} in $M$ is a subset $C\subset M$ containing $0$ and stable by addition (i.e an additive submonoid). We say that $C$ is generated by $m_1,...,m_n\in C$ if any element of $C$ can be written as $\sum_{i=1}^n \lambda_i m_i$ for $\lambda_i\in \NN$. In this case, we write
\begin{equation}\label{strictgen}
C=\NN m_1+...+\NN m_n.
\end{equation}
For a cone $C\subset M$, we define $\langle C \rangle $ by 
\begin{equation}\label{satcone}
\langle C \rangle :=\{x\in M, \exists n\geq 1, nx\in C\}.
\end{equation}
It is easy to see that $\langle C \rangle$ is again a cone. We say that $C$ is saturated in $M$ if $\langle C \rangle=C$. For elements $m_1,...,m_n\in C$, we denote by $\langle m_1,...,m_n\rangle$ the smallest saturated cone in $M$ containing $m_1,...,m_n$. If $C=\langle m_1,...,m_n\rangle$, we say that $C$ is s-generated by $m_1,...,m_n$ or that $\{m_1,...,m_m\}$ is an s-generating set for $C$. If $C\subset M$ is a cone, we denote by $C_{\RR^+}\subset M\otimes \RR$ the set
\begin{equation}\label{defCR}
C_{{\RR^+}}:=\left\{\sum_{i=1}^r \lambda_i m_i, \ \forall i=1,...,r, \ \lambda_i\in \RR_{\geq 0}, \ m_i\in C \right\}.
\end{equation}

The saturated cone depends of course on the ambient free abelian group $M$. In this paper, $M$ will always be the group of characters $X^*(T)$ of a fixed torus $T$.

\begin{lemma}
The set $C_{\zip}\subset X^*(T)$ is a cone.
\end{lemma}
\begin{proof}
Let $f$ and $g$ be nonzero global section of $\Lscr(\lambda)$ and $\Lscr(\nu)$ respectively over $\GF^\Zcal$. Then by \eqref{linebndmult} $fg$ is a nonzero global section of $\Lscr(\lambda+\nu)$, hence the result.
\end{proof}

\begin{definition}
The saturated zip cone is defined as $\langle C_{\zip} \rangle$.
\end{definition}

Since $C_{\zip}\subset X^*_{+,I}(T)$ and the latter cone is clearly saturated in $X^*(T)$, we have $\langle C_{\zip} \rangle \subset X^*_{+,I}(T)$.

\subsection{Motivation}\label{sec-motiv}
Let $X$ be a scheme endowed with a dominant morphism of stacks $\zeta:X\to \GZip^\Zcal$. For $\lambda\in X^*(T)$, define a vector bundle $\Vscr_{X}(\lambda):=\zeta^*\Vscr(\lambda)$ and a subset
\begin{equation}
C_X:=\left\{\lambda \in X^*(T), \ H^0(X,\Vscr_{X}(\lambda))\neq 0 \right\}.
\end{equation}
For example, $X$ can be the special fiber of a Hodge-type Shimura variety at a place of good reduction, and $\zeta$ the map defined in \cite[\Th~2.4.1]{ZhangEOHodge} (in this case the map $\zeta$ is even surjective). Since $\zeta$ is dominant, pull-back gives an inclusion $C_{\rm zip}\subset C_X$.

When $X$ satisfies some good properties and for certain reductive groups $G$, it was proved in \cite{Goldring-Koskivirta-Diamond-I} that one has an equality $\langle C_{\zip} \rangle= \langle C_X \rangle$. More specifically, we formulated the conjecture:

\begin{conjecture}\label{conjzipX}
Assume $X$ is an irreducible, proper $k$-variety endowed with a smooth and surjective map $\zeta:X\to \GZip^\Zcal$. Assume that $(G,\mu)$ is of Hodge-type (\cite[\Def~1.3.1]{Goldring-Koskivirta-Strata-Hasse-v2}). Then one has $\langle C_{\zip} \rangle= \langle C_X \rangle$.
\end{conjecture}

It is possible to modify this conjecture to allow non-irreducible, non-proper schemes, see \loccit However, it is in general not true that $C_{\zip} = C_X$. This suggests that only the saturated cone $\langle C_X \rangle$ is encoded in the stack $\GZip^\Zcal$, but the finer information about $C_X$ cannot be recovered in this way.

Let us expand slightly on the example of Shimura varieties Shimura varieties. Let $Sh_K(G,X)$ be a Shimura variety of Hodge-type attached to a reductive $\QQ$-group $G$. Let $A\to Sh_K(G,X)$ be the universal abelian variety. Then we may consider the flag space $Fl_K(G,X)$ of $Sh_K(G,X)$, which parametrizes points $x\in Sh_K(G,X)$ endowed with a full flag in $H^1_{dR}(A_x)$ refining the  Hodge filtration. In the Siegel case, these spaces were introduced in \cite{EkedahlGeerEO}, and in the general case, they were used in our previous articles \cite{Goldring-Koskivirta-Strata-Hasse-v2}, \cite{Goldring-Koskivirta-Diamond-I}.

If $p$ is a place of good reduction, Kisin and Vasiu have constructed a canonical, smooth integral model $\Scal_K$ of $Sh_K(G,X)$ over the ring $\Ocal_{E_v}$ where $v|p$ is any place of the reflex field $E$. Then, the flag space $Fl_K(G,X)$ also admits a similar model $\Fcal_K$ (see \cite{Goldring-Koskivirta-Strata-Hasse-v2}). By abuse of notation, denote also by $G$ the special fiber of a reductive $\ZZ_p$-model of $G_{\QQ_p}$ and write $S_K:=\Scal_K\otimes \kappa(v)$ and $F_K:=\Fcal_K\otimes \kappa(v)$. There is a smooth surjective map $\zeta:S_K\to \GZip^\mu$ defined in \cite[\Th~2.4.1]{ZhangEOHodge}, where $\mu$ is the cocharacter attached to the Shimura datum. Let $C_K$ denote the cone $C_X$ for $X=S_K$. Although the cone $C_K$ depends on the level $K$ (this is already clear for modular forms), we will now see that the saturated cone $\langle C_K \rangle$ does not. Hence, dependence on the level is not a valid obstruction for Conjecture \ref{conjzipX} to hold.

\begin{lemma}\label{indeplem}
Let $f:X\to Y$ be a finite surjective morphism of integral schemes with $Y$ normal. Let $\Lscr$ be a line bundle over $Y$ and assume that there exists a nonzero section $h\in H^0(X,f^*\Lscr)$. Then there exists $n\geq 1$ such that $H^0(Y,\Lscr^n)\neq 0$.
\end{lemma}

\begin{proof}
By the assumption on $f:X\to Y$, it admits a norm map $\Norm:f_*\Ocal_X\to \Ocal_Y$ of degree $n=[k(X):k(Y)]$ (see \cite[\Lem~30.17.7]{stacks-project}) and the induced map $\Norm:\Pic(X)\to \Pic(Y)$. The section $h$ gives rise to a morphism of $\Ocal_X$-linear map $\Ocal_X\to f^*\Lscr$. We can the use \loccitn, \Lem~30.17.3 to obtain an $\Ocal_Y$-linear map 
\begin{equation}
\Ocal_Y\simeq\Norm(\Ocal_X)\to \Norm(f^*\Lscr)\simeq \Lscr^n
\end{equation}
This yields a global section $h'\in H^0(Y,\Lscr^n)$. By the second part of \loccitn, \Lem~30.17.3, the section $h'$ vanishes exactly on $f(Z)$, where $Z\subset X$ is the vanishing locus of $Y$. In particular $h'\neq 0$, which proves the result.
\end{proof}

\begin{corollary}
The saturated cone $\langle C_K \rangle$ is independent of $K$.
\end{corollary}
\begin{proof}
Let $K'$, $K$ be two compact subgroups, we may assume $K'\subset K$. We have a finite \'{e}tale surjective map $\pi:F_{K'}\to F_K$. We clearly have $C_K\subset C_{K'}$ by pullback. Conversely, if $\lambda \in C_{K'}$, then $\Lscr(\lambda)$ admits a nonzero global section over $S_{K'}$. It is nonzero on at least one connected component $F_{K'}^\circ \subset F_{K'}$. Choose a connected component $F_K^\circ \subset F_K$ mapping to $F_{K'}^\circ$. By \Lem~\ref{indeplem}, there is a nonzero global section of $\Lscr(n\lambda)$ over $F_K^\circ$. We may extend it by zero on the other connected components of $F_K$, and this shows $n\lambda \in C_K$.
\end{proof}

\subsection{The case of line bundles}

Recall that $\Vscr(\lambda)$ is a line bundle on $\GZip^\Zcal$ if and only if $\lambda\in X^*(L)$. In this section, we restrict our attention to automorphic line bundles, i.e we consider the cone $C_{\zip}\cap X^*(L)$.

By \cite[\Th~3.1]{Koskivirta-compact-hodge}, there exists $n_0\geq 1$ such that the line bundle $\Vscr(\lambda)^{n_0}=\Vscr(n_0\lambda)$ admits a unique (up to scalar) nonzero section $h_\lambda$ over the open stratum $[E\backslash U_\Zcal]\subset \GZip^\Zcal$. The integer $n_0$ is determined explicitly in \cite[\Def~3.2.4]{Koskivirta-Wedhorn-Hasse}.

Recall (\cite[\Def~4.1.3]{Goldring-Koskivirta-zip-flags}) that $\lambda \in X^*(L)$ is called $\Zcal$-ample if $\lambda \in X^*(L)\cap X_-^*(T)$. If $\lambda \in X^*(L)$ is $\Zcal$-ample, the section $f_{\lambda}$ extends to a global section of $\GZip^\Zcal$ (\cite[\Th~5.1.4]{Koskivirta-Wedhorn-Hasse}). In other words, this shows:
\begin{proposition}\label{Zamplecont}
One has $ X^*(L)\cap X_-^*(T)\subset \langle C_{\rm zip} \rangle$.
\end{proposition}

\subsection{The Schubert cone} \label{schubcone}
We will define a subcone $C_{\Schub}\subset C_{\zip}$ which has the advantage of being easily computable. It can sometimes happen that $\langle C_{\Schub} \rangle =\langle C_{\zip} \rangle$. The Schubert stack is defined as the quotient stack
\begin{equation}\label{schubertstack}
\Schub:=[B\backslash G /B].
\end{equation}
This stack is naturally stratified by the $B\times B$-orbits in $G$, called the Schubert cells. They are parametrized by the set $W$ via the bijection $w\mapsto S_w:=B\dot{w}B$. Then $\Schub_w:=[B\backslash S_w /B]$ is a locally closed substack of $\Schub$. A fundamental property is that the Zariski closure $\overline{S}_w$ is normal (\cite[\Th~3]{Ramanan-Ramanathan-projective-normality}).

Again using \cite[\S5.8]{jantzen-representations}, for any pair of characters $(\lambda,\nu)\in X^*(T)\times X^*(T)$, there is an associated line bundle $\Lscr(\lambda,\nu)$ on the stack $\Schub$. One has an equivalence (\cite[\Th~2.2.1(a)]{Goldring-Koskivirta-Strata-Hasse-v2}):
\begin{equation}
H^0(\Schub_w, \Lscr(\lambda,\nu))\neq 0 \ \Longleftrightarrow \ \nu=-w^{-1}\lambda.
\end{equation}

If this condition is satisfied, the space $H^0(\Schub_w, \Lscr(\lambda,\nu))$ is one-dimensional, and the divisor of any nonzero element is given by Chevalley's formula (\loccit \Th~2.2.1(c)). In particular for $w=w_0$, the line bundle $\Lscr(\lambda,\nu)$ admits a nonzero section on the open Schubert stratum $\Schub_{w_0}$ if and only if $\nu=-w_0\lambda.$ The divisor of any nonzero $h_\lambda\in H^0(\Schub_{w_0}, \Lscr(\lambda,-w_0\lambda))$ is given by
\begin{equation}\label{divflam}
\div(h_\lambda)=-\sum_{\alpha\in \Delta}\langle \lambda, w_0 \alpha^\vee \rangle \overline{S}_{w_0s_\alpha}.
\end{equation}
In particular, $h_\lambda$ extends to $\Schub$ if and only if $\lambda\in X^*_+(T)$. It is explained in \cite[\S2.3]{Goldring-Koskivirta-Strata-Hasse-v2} that there is a natural map
\begin{equation}\label{psimap}
\psi:\GF^\Zcal \to \Schub.
\end{equation}
We obtain a natural stratification on $\GF^\Zcal$ defined as preimages of the Schubert strata in $\Schub$: for $w\in W$, set $\Xcal_w:=\psi^{-1}(\Schub_w)$. The stratum $\Xcal_w$ is isomorphic to $[E'\backslash BwBz^{-1}]$ ($E'$ acts via the inclusion $E'\subset B\times {}^zB$). It is locally closed in $\GF\simeq [E'\backslash G]$. Furthermore, for all $\lambda,\nu \in X^*(T)$, one has the formula
\begin{equation}\label{pullbackformula}
\psi^*\Lscr(\lambda,\nu)=\Lscr(\lambda+(z\nu) \circ \varphi)
\end{equation}
where $\varphi:T\to T$ is the Frobenius homomorphism (recall that $T$ is defined over $\FF_p$). For $\lambda\in X^*(T)$, denote by $h_\lambda$ an arbitrary nonzero element of $H^0(\Schub_{w_0},\Lscr(\lambda,-w_0\lambda))$. We obtain a rational section $\psi^*(h_\lambda) \in H^0(\Xcal_{w_0},\Lscr(h(\lambda)))$ where $h:X^*(T)\to X^*(T)$ is defined by $h(\lambda)=\lambda-(zw_0 \lambda)\circ \varphi$.

\begin{definition}\label{defSbtcone}
The Schubert cone is defined by $C_{\Schub}:= h(X^*_+(T))$. The saturated Schubert cone is $\langle C_{\Schub}\rangle$.
\end{definition}

By the above discussion, we have inclusions $C_{\Schub}\subset C_{\rm zip}$ and  $\langle C_{\Schub} \rangle \subset \langle C_{\rm zip} \rangle$. In certain easy cases, it can happen that $\langle C_{\Schub} \rangle = \langle C_{\rm zip} \rangle$, but they are usually distinct. For example, equality holds for Hilbert-Blumenthal zip data (\cite{Goldring-Koskivirta-Diamond-I}) and for symplectic zip data attached to the group $Sp(4)$ (\S\ref{section symplectic}). \Prop~\ref{Zamplecont} above can be made more precise when $P$ is defined over $\FF_p$:

\begin{lemma}\label{lemmacontainedFp}
Assume that $P$ is defined over $\FF_p$. Then $X^*(L)\cap X^*_-(T) \subset  \langle C_{\Schub}\rangle $.
\end{lemma}

\begin{proof}
Since $P$ is defined over $\FF_p$, the map $h:X^*(T)\to X^*(T)$ is given by $\lambda \mapsto \lambda -p w_{0,I} {}^\sigma \lambda$, and restricts to a map $X^*(L)\to X^*(L)$. Let $\chi\in X^*(L)\cap X^*_-(T)$ be a character and $\lambda \in X^*(L)_\QQ$ such that $\chi=h(\lambda)$. Fix $r\geq 1$ such that $\sigma^r\chi =\chi$. Then it is easy to see that
\begin{equation}
\lambda =-\frac{1}{p^{2r}-1}\sum_{i=0}^{2r-1} p^i (w_{0,L})^i\sigma^i\chi.
\end{equation}
Hence for all $\alpha \in \Delta \setminus I$, one has $(p^{2r}-1)\langle \lambda, \alpha^\vee \rangle =-\sum_{i=0}^{2r-1}p^i\langle \chi,(w_{0,L})^i\sigma^{-i}\alpha^\vee \rangle$. Since $\alpha \in \Delta \setminus I$, the root $(w_{0,L})^i\sigma^{-i}\alpha$ is again positive for all $i$. Hence we obtain $\lambda \in X^*_+(T)$, and this shows $\chi \in \langle C_{\Schub}\rangle $.
\end{proof}

\begin{lemma}\label{lemmasplitSchub}
Assume that $T$ is $\FF_p$-split. For $\lambda \in X^*(T)$, we have an equivalence
\begin{equation}\label{equivSbt}
\lambda\in  \langle C_{\Schub} \rangle  \ \Longleftrightarrow \ p w_{0,I}\lambda+\lambda \in X_-^*(T).
\end{equation}
\end{lemma}

\begin{proof}
In this case, the Galois action on $X^*(T)$ is trivial, so the map $h$ is simply given by $h(\lambda)=\lambda-p w_{0,L}\lambda$. Hence if $\lambda=h(\chi)$, one sees easily that $(p^2-1)\chi=-(pw_{0,L}\lambda+\lambda)$, and the result follows.
\end{proof}

\subsection{The Griffith-Schmidt cone}\label{sec-GScone}

We define the following subcone of $X^*(T)$ :

\begin{definition}\label{GSdef}
Let $C_{\GS}$ denote the set of characters $\lambda\in X^*(T)$ satisfying the conditions
\begin{align}
\langle \lambda, \alpha^\vee \rangle &\geq 0 \ \textrm{ for }\alpha\in I \label{ineq1} \\
\langle \lambda, \alpha^\vee \rangle &\leq 0 \ \textrm{ for }\alpha\in \Phi_+ \setminus I. \label{ineq2}
\end{align}
\end{definition}

It is easy to see that $C_{\GS}$ is a saturated subcone of $X^*(T)$, which we call the Griffith-Schmidt cone. Let us give some motivation for introducing this cone. The conditions defining $C_{\GS}$ were first understood by Griffith-Schmidt in their work \cite{Griffiths-Schmid-homogeneous-complex-manifolds}. They consider certain manifolds (called Griffith-Schmidt manifolds) which generalize Shimura varieties to non-minuscule cocharacters. In particular, their result applies in our situation to the flag space $Fl_K(G,X)$ of the Shimura variety. Since they only work in the compact case, assume $Sh_K(G,X)$ is a compact Shimura variety of Hodge type over a number field $E$. Denote by $C^\circ_{\GS}$ the "interior" of $C_{\GS}$, where the inequalities \eqref{ineq1} and \eqref{ineq2} are strict. The following result is shown in \loccit
\begin{theorem}
For all $\lambda\in C^\circ_{\GS}$, the line bundle $\Lscr(\lambda)$ is ample on $Fl_K(G,X)$.
\end{theorem}
In particular, for any $\lambda \in C^\circ_{\GS}$, there exists $N\geq 1$ such that $\Lscr(N\lambda)$ admits a global section on $Fl_K(G,X)$. Now, let $p$ be a place of good reduction, and retain the notation of \S\ref{sec-motiv}. In particular, recall that we have an $\Ocal_{E_v}$-scheme $\Fcal_K$ whose base change to $E_v$ is $Fl_K(G,X)\otimes_E E_v$.

\begin{proposition}
Assume that $H^0(Fl_K(G,X)\otimes \CC,\Lscr(\lambda))\neq 0$. Then one has also $H^0(\Fcal_{K}\otimes \overline{\FF}_p,\Lscr(\lambda))\neq 0$.
\end{proposition}

\begin{proof}
By flat base change along the map $\spec(E_v)\to \spec(\Ocal_{E_v})$, we have
\begin{equation}
H^0(Fl_K(G,X)_{E_v},\Lscr(\lambda)) = H^0(\Fcal_K,\Lscr(\lambda))\otimes_{\Ocal_{E_v}}E_v.
\end{equation}
Using again flat base change for $\spec(\CC)\to \spec(E_v)$, we deduce that $H^0(Fl_K(G,X)_{E_v},\Lscr(\lambda))$ and $H^0(\Fcal_K,\Lscr(\lambda))$ are nonzero. Denote by $f\in H^0(\Fcal_K,\Lscr(\lambda))$ an arbitrary nonzero element. Let $\kappa(v)$ be the residue field of $\Ocal_{E_v}$. If $f$ maps to a nonzero element via the natural map $H^0(\Fcal_K,\Lscr(\lambda))\to H^0(\Fcal_K\otimes \kappa(v),\Lscr(\lambda))$, then we deduce the result. If $f$ maps to zero, then it is divisible by an uniformizer $\varpi\in \Ocal_{E_v}$. There exists an integer $n\geq 1$ such that $\varpi^n$ divides $f$, but $\varpi^{n+1}$ does not, write $f=\varpi^n f'$. Then $f'$ reduces to a nonzero section over $\Fcal_K\otimes \kappa(v)$, which terminates the proof.
\end{proof}

Recall our general Conjecture \ref{conjzipX} explained earlier. Denote by $C_K$ the set of $\lambda \in X^*(T)$ such that $\Vscr(\lambda)$ admits nonzero global section over $S_K=\Scal_K\otimes \overline{\FF}_p$. Then we conjectured that $\langle C_K \rangle = \langle C_{\zip} \rangle$. We have just seen that $C^\circ_{\GS}\subset \langle C_K \rangle $. Hence, if our conjecture is correct, we must have also $C^\circ_{\GS}\subset \langle C_{\zip} \rangle $. It seems reasonable that the cone $\langle C_{\zip} \rangle_{\RR^+} $ is closed in $X^*(T)\otimes \RR$ for the real topology. Hence one should also expect:

\begin{conjecture}\label{conjGS}
One has $C_{\GS} \subset \langle C_{\zip} \rangle $.
\end{conjecture}

Note that Conjecture \ref{conjGS} involves entirely group-theoretical objects. We will even see a formulation of $\langle C_{\zip} \rangle $ in representation theory terms. Hence it is quite striking that we have come to the conclusion that the inclusion $C_{\GS} \subset \langle C_{\zip} \rangle $ must hold by means of the Shimura variety and the result of Griffith-Schmidt, which uses analytical methods and tools from the theory of automorphic representations.

We will check Conjecture \ref{conjGS} in \S\ref{sec-hwc} in the case when the Levi subgroup $L$ is defined over $\FF_p$. This gives some evidence that the more ambitious Conjecture \ref{conjzipX} should hold (at least for Shimura varieties).

\section{The $\mu$-ordinary cone} \label{subsec-muord}

\subsection{Sections on the $\mu$-ordinary locus} \label{sec-ordsec}
Let $(G,\mu)$ be a cocharacter datum with associated zip datum $\Zcal$. In this case, we prefer to denote the unique open $E$-orbit by $U_\mu\subset G$ (instead of the previous notation $U_\Zcal$), and we call it the $\mu$-ordinary zip stratum. The open substack 
\begin{equation}\label{gzipmu}
\GZip^{\mu\textrm{-ord}}:=[E\backslash U_\mu]
\end{equation}
is called the $\mu$-ordinary locus of $\GZip^\mu$. To avoid confusion between $\GZip^\mu$ and $\GZip^{\mu\textrm{-ord}}$, we will simply write $\GZip$ for the stack of $G$-zips (The cocharacter $\mu$ will be fixed once and for all).
The stabilizer of an arbitrary element $x\in U_\mu$ is a finite group scheme $S:=\Stab_E(x)\subset E$ (not necessarily \'{e}tale). The map $E\to U_\mu$, $\epsilon \mapsto \epsilon \cdot x$ yields an isomorphism $E/S\simeq U_\mu$. One can show that the element $1$ lies in $U_\mu$. Its stabilizer takes the form
\begin{equation}\label{stabone}
S:=\Stab_E(1)=\{(a,b)\in E, \ a=b\}.
\end{equation}
Identify $S$ with a subgroup of $P$ via the first projection $S\to P$. The stabilizer $S$ is explicitly determined in \cite{Koskivirta-Wedhorn-Hasse}. It particular, it is proved that $S\subset Q\cap L$ and can be written as $S=L_0(\FF_p)\ltimes S^\circ$, where $L_0\subset L$ is the largest Levi subgroup containing $T$ which is defined over $\FF_p$. The group $S^\circ$ is a finite unipotent (non-smooth) subgroup. Define the $\mu$-ordinary locus in $\GF$ by
\begin{equation}
\GF^{\mu\textrm{-ord}}:=\pi^{-1}(\GZip^{\mu\textrm{-ord}}).
\end{equation}

\begin{lemma}\label{isommuord}
There are isomorphisms $\GF^{\mu\textrm{-ord}} \simeq [E'\backslash U_\mu] \simeq  [B_L\backslash L/S]$.
\end{lemma}

\begin{proof}
The isomorphism $\GF^{\mu\textrm{-ord}}\simeq [E'\backslash U_\mu]$ is clear by definition. Hence, we obtain $\GF^{\mu\textrm{-ord}}\simeq [E'\backslash (E / S)] \simeq [(E'\backslash E) /S]$ and the first projection $E\to P$ induces an isomorphism $E'\backslash E\simeq B\backslash P\simeq B_L\backslash L$, hence the result.
\end{proof}

Denote by $\pi_\mu:\GF^{\mu\textrm{-ord}}\to \GZip^{\mu\textrm{-ord}}$ the natural projection. The relation $(\pi_{\mu})_{*}(\Lscr(\lambda))=\Vscr(\lambda)$  continues to hold since $\pi_\mu$ is a base-change of $\pi$ by an open embedding. In particular, we obtain a similar identification
\begin{equation}
H^0(\GZip^{\mu\textrm{-ord}},\Vscr(\lambda))\simeq H^0(\GF^{\mu\textrm{-ord}},\Lscr(\lambda)).
\end{equation}

\begin{rmk}\label{rem-muord}
Using the isomorphisms of Lemma \ref{isommuord}, these spaces also identify with the following objects:
\begin{enumerate}
\item\label{item-mu1} Regular functions $f:U_\mu\to \AA^1$ such that $f(\varepsilon\cdot g)=\lambda(\varepsilon)f(g)$ for all $g\in U_\mu$ and $\varepsilon\in E'$.
\item\label{item-mu2} Regular functions $f':L\to \AA^1$ such that $f'(bhs)=\lambda(b)f'(h)$ for all $b\in B_L$, $h\in L$ and $s\in S$.
\end{enumerate}
If $f:U_\mu \to \AA^1$ is a function as in \eqref{item-mu1}, then $f':L\to \AA^1$, $f'(a):=f(a\varphi(a)^{-1})$ is the corresponding function as in \eqref{item-mu2}. Conversely, one recovers $f$ from $f'$ as follows: For $g\in U_\mu$, write $g=ab^{-1}$ with $(a,b)\in E$. Then $f(g)=f'(\overline{a})$ (this is independent of the choice of $(a,b)\in E$ such that $g=ab^{-1}$ because of the $S$-invariance of $f'$).
\end{rmk}

Define the \emph{$\mu$-ordinary cone} $C_\mu\subset X^*(T)$ as the set of characters $\lambda \in X^*(T)$ such that $H^0(\GZip^{\mu\textrm{-ord}},\Vscr(\lambda))\neq 0$. The inclusion $\GZip^{\mu\textrm{-ord}}\subset \GZip$ induces an inclusion $C_{\zip}\subset C_\mu$. By similar arguments as before, the set $C_\mu$ is a cone in $X^*(T)$ and is contained in $X^*_{+,I}(T)$. Define also the \emph{saturated $\mu$-ordinary cone} by $\langle C_\mu \rangle$.

\begin{lemma} \label{lemmuordsat}
One has $\langle C_\mu \rangle =X^*_{+,I}(T)$.
\end{lemma}

To prove this lemma, we introduce auxiliary functions. Denote by $p^m$ the order of the unipotent group scheme $S^\circ$, the connected part of the finite group scheme $S=\Stab_E(1)$. Let $\lambda\in X^*(T)$ be any $L$-dominant and let $V(\lambda)=H^0(B_L\backslash L, \Lscr(\lambda))$ be the attached $L$-representation \eqref{repV}. Let $f\in V(\lambda)$ be a nonzero element. For $x\in L$, set:
\begin{equation} \label{ftildadef}
\tilde{f}(x):=\left(\prod_{s\in L_0(\FF_p)} f(xs)\right)^{p^m}.
\end{equation}

\begin{lemma}\label{steinbergsection}\ Let $D:=p^m|L_0(\FF_p)|$. The function $\tilde{f}$ satisfies the relation $\tilde{f}(bxs)=\lambda(b)^{-D}\tilde{f}(x)$, for all $x\in L$, $b\in B_L$ and $s\in S$.
\end{lemma}

\begin{proof}
The invariance under $L_0(\FF_p)$ is clear by the averaging and the invariance under $S^\circ$ is accomplished by the $p^m$-power in formula \eqref{ftildadef}.
\end{proof}

Using Remark \ref{rem-muord}~\eqref{item-mu2}, the function $\tilde{f}$ identifies with a (nonzero) element $\tilde{f}\in H^0(\GF^{\mu\textrm{-ord}},\Lscr(D\lambda))=H^0(\GZip^{\mu\textrm{-ord}},\Vscr(D\lambda))$. In particular, this shows that $D\lambda \in C_\mu$ and implies $\lambda \in \langle C_\mu \rangle$. Since $\lambda$ was taken to be an arbitrary $L$-dominant character of $T$, this proves Lemma \ref{lemmuordsat} above. We have constructed a (non-linear) map
\begin{equation}\label{constrftilde}
V(\lambda) \to H^0(\GF^{\mu\textrm{-ord}},\Lscr(D\lambda)).
\end{equation}
Let us now give another interpretation of $\tilde{f}$ which shows that they are very natural elements to consider. As we noted, the choice of $1\in U_\mu$ induces a surjective map $E\to U_\mu$, $\epsilon \mapsto \epsilon \cdot 1$. This map induces therefore a ring extension $k[U_\mu]\subset k[E]$ as well as a field extension $k(G) \subset k(E)$. Furthermore, the map $E\to U_\mu$ identifies with the projection $E\to E/S$, where $S=\Stab_E(1)$ (possibly non-smooth finite group scheme). Denote by $p^m$ the order of $S^\circ$.

\begin{lemma}
The extension $k[U_\mu]\subset k[E]$ is integral. The extension $k(G) \subset k(E)$ is finite and has degree the order of the finite group scheme $S$. Its separable degree is $|L_0(\FF_p)|$ and its inseparable degree $p^m$. The extension is normal and is a Galois extension if and only if $P$ is defined over $\FF_p$.
\end{lemma}

\begin{proof}
The statements about finiteness and the degree of the extension are simple field theory. The extension is Galois if and only if the stabilizer $S$ is smooth, which is equivalent to $L=L_0$, so $P$ defined over $\FF_p$.

\end{proof}

Let $f\in H^0(B_L\backslash L, \Lscr(\lambda))$ be nonzero. View $f$ as a regular function on $E$ via the natural maps $E\to P \to L$. The auxiliary function $\tilde{f}\in H^0(\GZip^{\mu\textrm{-ord}},\Vscr(D\lambda))$ is then obtained by applying the norm map
\begin{equation}
N:k[E]\to k[U_\mu].
\end{equation}

\subsection{Hodge-type zip data}
In the rest of this section, we assume that $(G,\mu)$ is of Hodge-type. Recall that $(G,\mu)$ is of Hodge-type (\cite[\Def~1.3.1]{Goldring-Koskivirta-Strata-Hasse-v2}) if there is an embedding $\iota:(G,\mu)\to (GSp(V,\phi),\mu')$ where $(V,\phi)$ is a symplectic space and $\mu'$ is minuscule. This case is important in the applications to Shimura varieties and automorphic forms.

In this case, a special role is played by the Hodge character $\eta_\omega\in X^*(L)$. Strictly speaking, $\eta_\omega$ depends on the choice of the embedding $\iota$. However, we fix $\iota$  throughout so we drop it from the notation. By \cite[\Lem~5.1.2]{Goldring-Koskivirta-zip-flags}, the character $\eta_\omega$ satisfies $\langle \eta_\omega ,\alpha^\vee \rangle <0$ for all $\alpha \in \Delta \setminus I$. The line bundle $\omega=\Lscr(\eta_\omega)$ on $\GZip$ is called the Hodge line bundle. In the Hodge-type case, an important property is the existence of a $\mu$-ordinary Hasse invariant, i.e a global section $h_\mu\in H^0(\GZip,\omega^d)$ (some $d\geq 0$) such that the non-vanishing locus of $h_\mu$ is exactly $\GZip^{\mu\textrm{-ord}}$ (\cite{Koskivirta-Wedhorn-Hasse}). In particular, $\eta_\omega \in \langle C_{\zip} \rangle$. In \cite{Goldring-Koskivirta-Strata-Hasse-v2,Goldring-Koskivirta-zip-flags}, it is proved that a power of $\omega$ has actually Hasse invariants on all strata.

We endow the cone $X^{*}_{+,I}(T)_{\RR^+}$ with the subspace topology given by the inclusion $X^{*}_{+,I}(T)_{\RR^+}\subset X^*(T)_{\RR^+}$.

\begin{proposition}\label{propmuord}
The following assertions hold
\begin{assertionlist}
\item \label{item-charzip} One has $X^*_{+,I}(T)=\ZZ\eta_\omega +  \langle C_{\rm zip} \rangle$.
\item \label{item-neighprop} The cone $ \langle C_{\zip} \rangle_{\RR^+}$ is a neighborhood of $\eta_\omega$ in $X^*_{+,I}(T)_{\RR^+}$.
\end{assertionlist}
\end{proposition}

\begin{proof}

We first show \eqref{item-charzip}. Let $\lambda\in X^*_{+,I}(T)$ be a character. By Lemma \ref{lemmuordsat} we may assume (after replacing $\lambda$ by a multiple) that there exists a nonzero section $f\in H^0(\GF^{\mu\textrm{-ord}},\Lscr(\lambda))$. Write $h'_\mu:=\pi^*(h_\mu)$ for the pullback of $h_\mu$ along the map $\pi:\GF \to \GZip$. The non-vanishing locus of $h'_\mu$ is exactly $\GF^{\mu\textrm{-ord}}$. Hence we can find $N\geq 1$ such that $f h'^{N}_{\mu}$ has no pole along the complement of $\GF^{\mu\textrm{-ord}}$, and hence extends to a global section of $\Vscr(\lambda) \otimes \omega^{dN}$. Therefore $\lambda+N\eta_\omega \in C_{\rm zip}$, which proves \eqref{item-charzip}. The second assertion is an immediate consequence.

\end{proof}

\subsection{A criterion}
By the previous proposition, $ \langle C_{\zip} \rangle_{\RR^+}$ is always a neighborhood of $\eta_\omega$ in $X^{*}_{+,I}(T)_{\RR^+}$. We give a criterion when the same holds for the subcone $ \langle C_{\Schub} \rangle_{\RR^+}\subset \langle C_{\zip} \rangle_{\RR^+}$. Note that in general, it is not even true that $\eta_\omega \in \langle C_{\Schub}\rangle$. However, we know it when $P$ is defined over $\FF_p$ by Lemma \ref{lemmacontainedFp}.

\begin{lemma}\label{lemzipSchub}
The following assertions are equivalent
\begin{assertionlist}
\item \label{item-neigh} The cone $ \langle C_{\Schub} \rangle_{\RR^+}$ is a neighborhood of $\eta_\omega$ in $X^*_{+,I}(T)_{\RR^+}$.
\item\label{item-zipSbt} One has $\eta_\omega\in  \langle C_{\Schub} \rangle$ and $ \langle C_{\Schub} \rangle  + \ZZ\eta_\omega = X^*_{+,I}(T)$.
\item \label{item-frob} $P$ is defined over $\FF_p$ and $\sigma$ acts on $I$ by $\alpha \mapsto -w_{0,L}\alpha$.
\item \label{item-GScone} One has the inclusion $C_{\GS}\subset \langle C_{\Schub} \rangle$.
\end{assertionlist}
\end{lemma}

\begin{proof}
As in \Prop~\ref{propmuord}, it is clear that \eqref{item-neigh} and \eqref{item-zipSbt} are equivalent. Assume that \eqref{item-zipSbt} holds. In particular, $X^*(L)\subset X^*_{+,I}(T)$, so for all $\lambda \in X^*(L)$, we can write $\lambda = \gamma + m\eta_\omega$ where $\gamma \in \langle C_{\Schub} \rangle$ and $m\in \ZZ$. Hence there is $d\geq 1$ such that $d\lambda =h(\chi) +dm\eta_\omega$, with $\chi \in X^*_+(T)$. In particular, $h(\chi)\in X^*(L)$, so for all $\alpha \in I$, we have 
\begin{equation}
0=\langle h(\chi),\alpha^\vee \rangle =\langle \chi, \alpha^\vee \rangle-\langle \chi,\sigma(w_{0,L}\alpha^\vee) \rangle.
\end{equation}
Since $\alpha\in \Delta_L$, the root $\sigma(w_{0,L}\alpha^\vee)$ is negative, hence we deduce $\langle \chi, \alpha^\vee \rangle=\langle \chi,\sigma(w_{0,L}\alpha^\vee) \rangle=0$. As $\alpha$ varies in $I$, the root $-\sigma(w_{0,L}\alpha^\vee)$ spans $\sigma(I)$. Hence $\chi$ is orthogonal to $I\cup \sigma(I)$. Denote by $K\subset X^*_+(T)$ the subcone orthogonal to $I\cup \sigma(I)$. It generates a $\QQ$-vector space of dimension $\rk(X^*(T))-|I\cup \sigma(I)|$.

Consider the map $K \to X^*(L)$, $\chi\mapsto h(\chi)$. Since $\eta_\omega\in \langle C_{\Schub}\rangle$, we may write $r\eta_\omega=h(\chi_\omega)$ for $r\geq 1$ and $\lambda_\omega\in K$ (by the above discussion). Hence for any $\lambda \in X^*(L)$, we have $d\lambda = h(\chi +dmr\chi_\omega)$. It follows that $h:\Span_\QQ(K)\to X^*(L)_\QQ$ is an isomorphism. In particular, one has $|I\cup \sigma(I)| = |I|$, thus $\sigma(I)=I$ and so $P$ is defined over $\FF_p$.

For $\alpha \in I$, let $\lambda_\alpha\in X^*(T)$ be a character such that $\langle \lambda_\alpha,\beta^\vee\rangle =0$ for all $\beta \in \Delta\setminus \{\alpha\}$ and $\langle \lambda_\alpha,\alpha^\vee\rangle >0$. Since $\lambda \in X_{+,I}^*(T)$, we can find $d\geq 1$, $m\geq 1$ and $\chi_\alpha \in X^*_+(T)$ such that $d\lambda_\alpha =h(\chi_\alpha) +dm\eta_\omega$. As above, we deduce $\langle \chi_\alpha, \beta^\vee \rangle=\langle \chi_\alpha,\sigma(w_{0,L}\beta^\vee) \rangle=0$ for all $\beta \in \Delta \setminus \{\alpha\}$. Furthermore, we have $\langle \chi_\alpha, \alpha^\vee \rangle-p\langle \chi_\alpha,\sigma(w_{0,L}\alpha^\vee) \rangle>0$. It follows that $\langle \chi_\alpha, \alpha^\vee \rangle >0$. The map $\beta \mapsto -\sigma(w_{0,L}\beta)$ is a bijection $I\to \sigma(I)=I$, so we must have $-\sigma(w_{0,L}\alpha)=\alpha$, hence $\sigma(\alpha)=-w_{0,L}\alpha$ for all $\alpha \in I$, which shows that \eqref{item-frob} holds. 

Conversely, assume that $P$ is defined over $\FF_p$ and that the Galois action on $I$ is given in this way. Let $\lambda \in C_{\GS}$. As in the proof of \Lem~\ref{lemmacontainedFp}, we have $h(\chi)=\lambda$ where $\chi \in X^*(T)_\QQ$ is defined by $(p^{2r}-1)\chi =-\sum_{i=0}^{2r-1} p^i (w_{0,L})^i\sigma^i\lambda$ and where $r\geq 1$ is an integer such that $\sigma^r\lambda =\lambda$. We need to show that $(p^{2r}-1)\chi\in X^*_+(T)$. For $\alpha \in \Delta$, the sign of $\langle \chi, \alpha^\vee\rangle$ is the same as the sign of $
-\sum_{i=0}^{2r-1} p^i \langle(w_{0,L})^i\sigma^i\lambda,\alpha^\vee \rangle=-\sum_{i=0}^{2r-1} p^i \langle \lambda,\sigma^i (w_{0,L})^i\alpha^\vee \rangle$. First, assume that $\alpha \in \Delta \setminus I$. Then for all $0\leq i \leq 2r-1$, one has $\sigma^i (w_{0,L})^i\alpha\in \Phi_+\setminus I$. By assumption, we have $\langle \lambda,\sigma^i (w_{0,L})^i\alpha^\vee \rangle \leq 0$, and hence $\langle \chi, \alpha^\vee\rangle \geq 0$. Next, if $\alpha \in I$, we have
\begin{equation}
-\sum_{i=0}^{2r-1} p^i \langle \lambda,\sigma^i (w_{0,L})^i\alpha^\vee \rangle=\left(\sum_{i=0}^{2r-1} p^i (-1)^{i+1}\right) \langle \lambda,\alpha^\vee \rangle \geq 0.
\end{equation}
This shows $\chi\in X^*_+(T)_\QQ$ as desired.

Finally, if \eqref{item-GScone} holds, then $\langle C_{\Schub} \rangle_{\RR^+}$ contains $C_{\GS,\RR^{+}}$, which is a neighborhood of $\eta_\omega$ in $X^*_{+,I}(T)_{\RR^+}$, so \eqref{item-neigh} holds.

\end{proof}

\begin{rmk}
Let $(V,q)$ be a quadratic space over $\FF_p$ of dimension $n=2r$, $r\geq 1$. Recall that there are two isomorphism classes of special orthogonal groups $SO(V,q)$, one of which is $\FF_p$-split and the other one is not. Assume that $G=SO(V,q)$ is non-split over $\FF_p$. In the case when $r$ is odd, the action of the Frobenius on $\Delta$ is given by $\alpha \mapsto -w_0\alpha$. Now, choose a quadratic space over $\FF_p$ of dimension $n=4d$, $d\geq 1$ such that $G=SO(V,q)$ is non-split over $\FF_p$. The group $G$ admits a Levi subgroup $L\subset G$ isomorphic to $SO(W,q)$, where $W\subset V$ is a subspace of dimension $n-2=4d-2$. By the previous discussion, the Levi $L$ satisfies Condition \eqref{item-frob} of \Lem~\ref{lemzipSchub}. This example arises in the theory of Shimura varieties of orthogonal Hodge-type.
\end{rmk}

\section{Highest weight cone}\label{sec-hwc}

In this section, we continue to assume that $(G,\mu)$ is a cocharacter datum of Hodge-type. We define another (saturated) cone $C_{\rm hw}\subset \langle C_{\zip} \rangle$. The sections that we construct are naturally attached to the highest weight vectors in the representations $V(\lambda)$ for $\lambda \in X_{+,I}^*(T)$. Therefore, it will be called the "highest weight cone".

\subsection{Valuations}

For each nonzero element $f\in H^0(B_L\backslash L, \Lscr(\lambda))$, we have defined in equation \eqref{ftildadef} a section $\tilde{f}$ in $H^0(\GF^{\mu\textrm{-ord}}, \Lscr(D\lambda))$ where $D=p^m|L_0(\FF_p)|$ and $p^m$ is the order of $S^\circ$ and $S=L_0(\FF_p) \ltimes S^\circ$. Previously (see proof of Lemma \ref{lemmuordsat}), we multiplied $\tilde{f}$ with a power of the Hasse invariant $h'_\mu=\pi^*(h_\mu)$ to remove the possible poles of $\tilde{f}$. This did not take into account that the divisor of $\tilde{f}$ may have different multiplicities along each irreducible component of the complement of $\GF^{\mu\textrm{-ord}}$, therefore it gave only a coarse result (\Prop~\ref{propmuord}).

In this section however, we want to multiply $\tilde{f}$ with "partial Hasse invariants", to remove the poles "one by one". To obtain the sharpest result, we want to multiply $\tilde{f}$ with the smallest possible power of these partial Hasse invariants. Hence we will study the valuation of $\tilde{f}$ along the codimension one strata. The outcome will be a global section on $\GF$ of a line bundle which does not vanish along any of the components of the complement of $\GF^{\mu\textrm{-ord}}$. We will see that when $f$ is the highest weight of the $L$-representation $H^0(B_L\backslash L, \Lscr(\lambda))$, we obtain an interesting family of zip-automorphic forms.

We need to assume that $P$ is defined over $\FF_p$. In particular, $L_0=L$ and $S$ is an \'{e}tale group scheme isomorphic to the constant group scheme $L(\FF_p)$. We also have $P\cap Q=L$. Recall that we fixed a frame $(B,T,z)$ (\S\ref{review}) where we take $z=w_0 w_{0,J}=w_{0,I}w_0$. The complement of $\GF^{\mu\textrm{-ord}}$ in $\GF$ is equi-dimensional of codimension one and is a union of closures of flag strata.
Specifically, the irreducible components are exactly the flag strata closures $\overline{\Xcal}_w$ with $w=s_\alpha w_0$ for $\alpha \in \Delta \setminus I$. Recall that $U_\mu \subset G$ denotes the unique open $E$-orbit and $\GF^{\mu\textrm{-ord}}=[E'\backslash U_\mu]$. Denote by $B^-$ the unique Borel such that $B\cap B^-=T$.

\begin{lemma}\label{complemlem}
We have the following properties
\begin{enumerate}
\item \label{item-Fp1} $U_\mu= E\cdot 1 = PQ$ (the set of elements of the form $xy$, $x\in P$, $y\in Q$).
\item \label{item-Fp2} One has $BB^-\subset U_\mu$.
\item \label{item-Fp3} For $\alpha \in \Delta$, set $ S_\alpha=Bs_\alpha w_0 B z^{-1}=Bs_\alpha w_{0,I}{}^z B$. Then one has
\begin{equation*}
U_\mu = G \setminus \bigcup_{\alpha \in \Delta \setminus I } \overline{S}_\alpha
\end{equation*}
\item \label{item-Fp4} Moreover, $\overline{S}_\alpha=\overline{E\cdot s_\alpha}=\overline{Ps_\alpha Q}$.
\end{enumerate} 
\end{lemma}

\begin{proof}
The first statement is \cite[\Cor~2.12]{Wedhorn-bruhat}. For \eqref{item-Fp2}, recall that ${}^z B\subset Q$ and $z=w_{0,I}w_0$. Hence $w_{0,I}B^-w_{0,I}\subset Q$. Since $w_{0,I}\in L=P\cap Q$ (as $P$ is defined over $\FF_p$), we obtain $B^-\subset Q$, so $BB^-\subset PQ=U_\mu$. To show \eqref{item-Fp3}, note that since $P\times Q$ is connected, every irreducible component of $G\setminus U_\mu$ is also stable by $P\times Q$. The $B\times {}^zB$-orbits of codimension one are exactly the $S_\alpha$, and it is easy to show that $S_\alpha \subset PQ \Longleftrightarrow \alpha \in I$. Finally, for \eqref{item-Fp4}, we already know that $\overline{S}_\alpha$ is stable by $P\times Q$. Since $S_\alpha=Bs_\alpha w_{0,I}{}^z B$ and $B\times {}^zB\subset P\times Q$, we must have $\overline{S}_\alpha=\overline{Ps_\alpha w_{0,I} Q}=\overline{Ps_\alpha Q}$. Clearly $\overline{E\cdot s_\alpha}\subset \overline{Ps_\alpha Q}$, hence it suffices to show that $E\cdot s_\alpha$ has codimension one in $G$. First of all $E\cdot s_\alpha = E\cdot (w_{0,I} s_\alpha w_{0,I})$ since $(w_{0,I},w_{0,I})\in E$. Hence $E\cdot s_\alpha=G_{w}$ with $w=w_{0,I}s_\alpha w_0$ (notation as in \S\ref{zipstrata}). This element has colength 1 in ${}^I W$ because $x\mapsto w_{0,I}x w_0$ is an order-reversing involution of ${}^I W$. We deduce that $E\cdot s_\alpha$ has codimension 1 and thus $\overline{E\cdot s_\alpha}=\overline{Ps_\alpha Q}$.
\end{proof}

For each $\alpha \in \Delta \setminus I$, denote by $v_\alpha$ the valuation on $k(G)$ given by the irreducible divisor $\overline{B s_\alpha w_0 Bz^{-1}}\subset G$. Any section of the line bundle $\Lscr(\lambda)$ over $\GF^{\mu\textrm{-ord}}$ can be viewed as regular function on $U_\mu$, so we obtain maps
\begin{equation}\label{valpha}
v_\alpha : H^0(\GF^{\mu\textrm{-ord}}, \Vscr(\lambda)) \to \ZZ.
\end{equation}
The space $H^0(\GF, \Lscr(\lambda))$ can be interpreted as the subspace consisting of elements $f$ such that $v_\alpha(f)\geq 0$ for all $\alpha \in \Delta \setminus I$. Since we are only interested in the sign of these valuations, we introduce the following notation: For $a,b\in \ZZ$, write $a\sim b$ if $a,b$ have the same sign $\in \{+,-,0\}$.

\subsection{Adapted morphisms}
Let $X$ be an irreducible normal $k$-variety and let $U\subset X$ be an open subset such that $S=X\setminus U$ is irreducible of codimension one. Let $f\in H^0(U,\Ocal_X)$ be a regular section on $U$. Denote by $Z_U(f)\subset U$ the vanishing locus of $f$ in $U$ and let $\overline{Z_U(f)}$ be its Zariski closure in $X$. In this section, we endow all locally closed subsets of schemes with the reduced structure. Let $Y$ be an irreducible $k$-variety with a $k$-morphism $\psi:Y\to X$.
\begin{definition}
We say that $\psi$ is adapted to $f$ (with respect to $U$) if
\begin{enumerate}
\item $\psi(Y)\cap U\neq \emptyset$
\item $\psi(Y)\cap S$ is not contained in $\overline{Z_U(f)}$.
\end{enumerate}

\end{definition}

Note that $Y$ if is adapted to $f$, then in particular $\psi(Y)$ is not contained in $\overline{Z_U(f)}$ and $\psi(Y)\cap S\neq \emptyset$. The reason for this definition is the following easy lemma that we will use:

\begin{lemma}
Let $\psi:Y\to X$ adapted to $f$. Then $f$ extends to a section $f\in H^0(X,\Ocal_X)$ if and only if the section $\psi^*(f)\in H^0(\psi^{-1}(U),\Ocal_Y)$ extends to $Y$.
\end{lemma}

\begin{proof}
One direction is clear. Now assume that $\psi^*(f)$ extends to $Y$. Since $X$ is normal, it suffices to show that the Weil divisor of $f$ has non-negative multiplicity along $S$. Assume that $f$ has a pole along $S$ and consider the function $g=1/f$. It extends to a regular function on the open subset $X\setminus \overline{Z_U(f)}$, and vanishes on $S\setminus \overline{Z_U(f)}$. It pulls back to a section over $Y\setminus \psi^{-1}(\overline{Z_U(f)})$ (note that this is $\neq \emptyset$) which vanishes on $\psi^{-1}(S\setminus \overline{Z_U(f)})$. Since $\psi^{*}(f)$ extends to a regular function on $Y$, one must have $\psi^{-1}(S\setminus \overline{Z_U(f)})=\emptyset$, hence $\psi(Y)\cap S\subset \overline{Z_U(f)}$, which contradicts the assumption.
\end{proof}

Let again $\psi:Y\to X$ be adapted to $f\in H^0(U,\Ocal_X)$ with respect to $U$ and assume further that $Y$ is normal and $\psi^{-1}(S)$ is irreducible of codimension one in $Y$. As we remarked above, $\psi^{*}(f)$ is a nonzero rational function on $Y$. Denote by $v_S$ the valuation on $k(X)$ given by $S$ and by $v_{\psi^{-1}(S)}$ the one on $k(Y)$ given by $\psi^{-1}(S)$. Then the previous lemma can be restated as follows:

\begin{corollary}\label{sameval}
Under these assumptions, one has
\begin{equation}
v_S(f)\sim v_{\psi^{-1}(S)}(\psi^{*}(f)).
\end{equation}
\end{corollary}

Now assume that $X$ is an irreducible normal $k$-variety endowed with an algebraic action of a algebraic group $H$ over $k$. Assume again that $U\subset X$ is open, $S=X\setminus U$ has codimension one, and assume further that $U$ is $H$-stable. We say that $f\in  H^0(U,\Ocal_X)$ is an $H$-eigenfunction if there is a character $\lambda:H\to \GG_m$ such that $f(h\cdot x)=\lambda(h)f(x)$ for all $h\in H, x\in U$.

\begin{lemma}\label{denseadapted}
Let $\psi:Y\to X$ be a map such that $\psi(Y)\cap U\neq \emptyset$ and $H\cdot (\psi(Y)\cap S)$ is Zariski dense in $S$. Then $\psi$ is adapted to any $H$-eigenfunction on $U$.
\end{lemma}

\begin{proof}
If $f$ is a $H$-eigenfunction, then clearly $Z_U(f)$ is $H$-stable. Hence if $\psi(Y)\cap S\subset \overline{Z_U(f)}$, we deduce that $H\cdot (\psi(Y)\cap S)$ is contained in $\overline{Z_U(f)}$, and thus $S\subset \overline{Z_U(f)}$ by density. This is impossible by a dimension argument.
\end{proof}

From now on, we consider the action of $E'$ on the variety $G$ as in \S\ref{sec-zipflag} and the open subset $U_\mu\subset G$. We will now define a morphism which satisfies the conditions of Lemma \ref{denseadapted}. For $\alpha\in \Delta\setminus I$, let $T_\alpha \subset T$ be the connected component of $\Ker(\alpha)$ and let $G_\alpha\subset G$ be the centralizer of $T_\alpha$. It is a semisimple subgroup of rank $1$ and $\Lie(G_\alpha)=\Lie(T)\oplus \gfr_\alpha \oplus \gfr_{-\alpha}$. Let $D_\alpha:=T\cap D(G_\alpha)$, where $D(H)$ denotes the derived group of an algebraic group $H$. Then $D_\alpha$ is a maximal torus in $G_\alpha$ and the pair $(D(G_\alpha),D_\alpha)$ is isomorphic to $(SL_2,D_2)$ or $(PGL_2,D_2)$ where $D_2$ denotes in each case the diagonal torus. It follows that there is a unique homomorphism
\begin{equation}
R_\alpha : SL_2 \to G
\end{equation}
whose image is $D(G_\alpha)$ and maps $D_2$ onto $D_\alpha$, and maps the lower Borel subgroup $B_2\subset SL_2$ into $B$. Denote by $B^-$ the opposite Borel of $B$ with respect to $T$ and $B^-_L:=B^-\cap L$. Define a variety $Y=(R_u(B^-)\cap L) \times \AA^1$ and a morphism:
\begin{equation}
\psi_\alpha:Y \to G , \quad (b,t)\mapsto b R_\alpha \left(A(t)\right)\varphi(b)^{-1} \quad \textrm{where} \ A(t)=\left(
\begin{matrix}
t & 1 \\ -1 & 0 \end{matrix}
\right).
\end{equation}
Note that $A(0)$ coincides with $s_\alpha$ (up to an element of $T$, but this is irrelevant for us). Since $S_\alpha$ is a codimension one stratum, the set $G^\alpha:=U_\mu \cup S_\alpha$ is open in $G$. The group $E'$ acts on $G^\alpha$ by restriction and on the open subset $U_\mu\subset G^\alpha$. 

\begin{proposition}\label{psiadapted}
The following properties hold
\begin{enumerate}
\item \label{item-imagepsi} The image of $\psi^\alpha$ is contained in $U_\mu \cup \overline{S}_\alpha$ (but maybe not in $G^\alpha$).
\item \label{item-psit} For any $b\in R_u(B^-)\cap L$, one has $\psi_\alpha(b,t)\in U_\mu \Longleftrightarrow t\neq 0$.
\item \label{item-adapted} Set $Y_\alpha:=\psi_\alpha^{-1}(G^\alpha)$.  The map $\psi_\alpha :Y_\alpha\to G^{\alpha}$ is adapted to any $E'$-eigenfunction on $U_\mu$.
\end{enumerate}
\end{proposition}

\begin{proof}
First, we claim that for $t\neq 0$, one has $R_\alpha(A(t))\in U_\mu$. Indeed, for $t\neq 0$, the matrix $A(t)$ can be written as a product $bb'$ where $b\in B_2$ and $b'\in B_2^{-}$. Since $\psi_\alpha$ maps $B_2$ into $B$, we have $R_\alpha(A(t))\in BB^-\subset U_\mu$ (\Lem~\ref{complemlem}~\eqref{item-Fp2}). By \Lem~\ref{complemlem}~\eqref{item-Fp1}, $U_\mu$ is a $P\times Q$-orbit, so $\psi_\alpha(b,t)\in U_\mu$ for all $b\in R_u(B^-)\cap L$ and all $t\neq 0$. If $t=0$, the element $R_{\alpha}\left(A(t)\right)$ coincides with $s_\alpha$ up to an element of $T$. Therefore $\psi_\alpha(b,0)\in Ps_\alpha Q\subset \overline{Ps_\alpha Q}=\overline{S}_\alpha$ (by Lem~\ref{complemlem}~\eqref{item-Fp4}). This proves \eqref{item-imagepsi} and \eqref{item-psit}.

For the last assertion, we use Lemma \ref{denseadapted}. We need to show that $E'\cdot (\psi_\alpha(Y_\alpha)\cap S_\alpha)$ is dense in $S_\alpha$. In other words, the union of $E'$-orbits in $S_\alpha$ intersecting $\psi_\alpha(Y_\alpha)$ is dense in $S_\alpha$. Recall that $\overline{S}_\alpha$ contains a unique dense open $E$-orbit $C=E\cdot s_\alpha$ (\Lem~\ref{complemlem}~\eqref{item-Fp4}). Furthermore, it follows immediately from \cite[\Th~5.14]{Pink-Wedhorn-Ziegler-zip-data} that $E\cdot s_\alpha = E\cdot (t s_\alpha)$ for all $t\in T$. In particular, we have $E\cdot s_\alpha = E\cdot (R_\alpha(0))$. Hence, any element $c\in C$ can be written $c=ux R_\alpha(0) \varphi(x)^{-1}v$ for $x\in L$, $u\in R_u(P)$ and $v\in R_u(Q)$. For a dense subset $C'\subset C$, we can impose the condition that $x\in L$ is of the form $x=bb'$ where $b\in B_L$ and $b'\in R_u(B^-)\cap L$ (because the set of such products is dense in $L$). The element $c\in C'$ is then in the same $E'$-orbit as $b'R_\alpha(0) \varphi(b')^{-1}\in \psi_\alpha(Y)$. It follows that any element in $C'\cap S_\alpha$ is in the same $E'$-orbit as an element of $\psi_\alpha(Y_\alpha)$. Since $S_\alpha$ is open in $\overline{S}_\alpha$, the subset $C'\cap S_\alpha$ is also dense in $S_\alpha$ and this proves \eqref{item-adapted}.
\end{proof}

\subsection{A formula for $v_\alpha(\tilde{f})$}
Write $D=|L(\FF_p)|$. We constructed in \S\ref{sec-ordsec} a (non-linear) map $f\mapsto \tilde{f}$ :
\begin{equation}
V(\lambda) \longrightarrow H^0(\GF^{\mu\textrm{-ord}}, \Lscr(D\lambda)).
\end{equation}
Furthermore, recall that we have valuations $(v_\alpha)_{\alpha \in \Delta\setminus I}$ on the right-hand side. In this section, we give a formula for (the sign of) $v_\alpha(\tilde{f})$.

Recall that the maximal torus $T$ is defined over $\FF_p$ (but is not required to be $\FF_p$-split). Denote the group of cocharacters of $T$ by $X_*(T)$ and if $\delta\in X_*(T)$, denote by ${}^\sigma \delta$ the action of the Frobenius on $\delta$. Hence, one has the relation $\varphi\circ \delta=p ({}^\sigma \delta)$, where $\varphi:T\to T$ is the Frobenius homomorphism. Define a map $\gamma: X_*(T)\to X_*(T)$, $\delta\mapsto \delta-p({}^\sigma \delta)$.

\begin{lemma}
The map $\gamma$ induces an automorphism of $X_*(T)\otimes \QQ$. Furthermore, if $\nu\in X_*(T)$ is a cocharacter defined over $\FF_{p^r}$ for some $r\geq 1$, then one has $\gamma(\delta)=\nu$ for
\begin{equation}
\delta=-\frac{1}{p^r-1}\sum_{i=0}^{r-1} p^i({}^{\sigma^i} \nu).
\end{equation}
\end{lemma}

\begin{proof}
The map $\gamma$ is the identity on $X_*(T)\otimes \FF_p$, so the first part of the lemma is clear. Now, let $\delta \in X_*(T)\otimes \QQ$ such that $\gamma(\delta)=\nu$. Then one has $\nu+p({}^\sigma \nu)=\delta-p^2({}^{\sigma^2} \delta)$, and by induction $\sum_{i=0}^{m-1} p^i({}^{\sigma^i} \nu) =\delta-p^{m}({}^{\sigma^m} \delta)$ for all $m\geq 1$. For $m=r$, we obtain the result.
\end{proof}

For $\alpha \in \Delta\setminus I$, denote by $r_\alpha$ the smallest integer $r$ such that $\alpha$ is defined over $\FF_{p^r}$ (then $\alpha^\vee$ is also defined over this field) and set $m_\alpha=p^{r_\alpha}-1$. Define a cocharacter $\delta_\alpha\in X^*(T)$ by
\begin{equation}\label{deltaalpha}
\delta_\alpha=-\sum_{i=0}^{r_\alpha-1}p^i({}^{\sigma^i}\alpha^\vee).
\end{equation}
Be the previous lemma, the cocharacter $\delta_\alpha$ satisfies $m_\alpha\alpha^\vee=\delta_\alpha-\varphi\circ \delta_\alpha$, equivalently
\begin{equation}\label{Ealphavee}
\alpha^\vee(t)^{m_\alpha} = \delta_\alpha(t)\varphi(\delta_\alpha(t))^{-1}
\end{equation}
for all $t\in \GG_m$. 

If $R$ is an integral domain, we denote by $v_t$ the $t$-adic valuation on $R[t]$, extended to $R[t,\frac{1}{t}]$. In what follows, we take $R$ to be the $k$-algebra of regular functions on the affine $k$-variety $R_u(B^-)\cap L$. Furthermore, we view elements $f\in V(\lambda)$ as functions on $L$, satisfying the usual condition with respect to the action by $B_L$.

\begin{lemma}\label{formulalemma}
Let $f\in V(\lambda)$ be any non-zero element. For $\alpha \in \Delta \setminus I$, one has the following formula:
\begin{equation}
v_\alpha (\tilde{f})\sim \sum_{s\in L(\FF_p)}v_t(f(b\delta_\alpha (t)s)).
\end{equation}
where the function $(b,t)\mapsto f(b\delta_\alpha(t)s)$ on $Y$ is viewed as an element of $R[t,\frac{1}{t}]$, where $R$ is the $k$-algebra of regular functions on $R_u(B^-)\cap L$.
\end{lemma}

\begin{proof}
Consider the map $\psi_\alpha : Y_\alpha \to G^\alpha$, it is adapted to $\tilde{f}$ by \Prop~\ref{psiadapted}. Define an irreducible subvariety $H\subset Y_\alpha$ by the condition $t=0$. By \Cor~\ref{sameval}, we have:
\begin{align}
v_\alpha(\tilde{f})\sim v_H(\psi_\alpha^*(\tilde{f}))
\end{align}
The function $\psi_\alpha^*(\tilde{f})$ is given by
\begin{equation}
\psi_\alpha^*(\tilde{f})(b,t)=\tilde{f}(\psi_\alpha(b,t))=\tilde{f}(bR_\alpha(A(t))\varphi(b)^{-1})
\end{equation}
To compute the value of $\tilde{f}(g)$ for $g\in U_\mu$, we need to write $g$ as a product $g=xy^{-1}$ for $(x,y)\in E$. Then, by construction (\Rmk~\ref{rem-muord}~\eqref{item-mu2}) $\tilde{f}(g)$ is defined as $\prod_{s\in L(\FF_p)} f(\overline{x}s)$, where $\overline{x}\in L$ is the Levi component of $x$. Observe that one has
\begin{equation}
A(t)=\left(\begin{matrix}
t&1\\-1&0
\end{matrix} \right)=\left(\begin{matrix}
1&0\\-t^{-1}&1
\end{matrix} \right) \left(\begin{matrix}
t&0\\0&t^{-1}
\end{matrix} \right)\left(\begin{matrix}
1&t^{-1}\\0&1
\end{matrix} \right)
\end{equation}
Since $\alpha \in \Delta \setminus I$, one has $R_\alpha(m)\in R_u(P)$ (resp. $R_\alpha(m)\in R_u(Q)$) for any lower-triangular (resp. upper-triangular) unipotent matrix $m\in SL_2$. Since the computation of $\tilde{f}(g)$ depends only on the Levi component of $x$ as explained above, we can simplify the expression:
\begin{equation}
\psi_\alpha^*(\tilde{f})(b,t)=\tilde{f}\left(bR_\alpha\left(\left(\begin{matrix}
t&0\\0&t^{-1}
\end{matrix} \right) \right)\varphi(b)^{-1}\right)=\tilde{f}(b\alpha^\vee(t)\varphi(b)^{-1})
\end{equation}
where we used that the restriction of $R_\alpha$ to the diagonal torus of $SL_2$ coincides with $\alpha^\vee$. We want to compute the sign of the $t$-valuation of the above expression. This sign will not change if we replace $t$ by a positive power of $t$. Hence, we have $v_t(\psi_\alpha^*(\tilde{f})(b,t))\sim v_t(\psi_\alpha^*(\tilde{f})(b,t^{m_\alpha}))$ for the integer $m_\alpha\geq 1$ defined previously. Now, by equation \eqref{Ealphavee} we can write $\alpha^\vee(t^{m_\alpha})=\alpha^\vee(t)^{m_\alpha} = \delta_\alpha(t)\varphi(\delta_\alpha(t))^{-1}$. We obtain:
\begin{equation}
\psi_\alpha^*(\tilde{f})(b,t^{m_\alpha})=\tilde{f}(b\alpha^\vee(t)^{m_\alpha}\varphi(b)^{-1})=\tilde{f}(b\delta_\alpha (t) \varphi(b\delta_\alpha (t))^{-1})=\prod_{s\in L(\FF_p)}f(b\delta_\alpha(t)s)
\end{equation}
and we deduce finally $v_\alpha (\tilde{f})\sim \sum_{s\in L(\FF_p)}v_t(f(b\delta_\alpha (t)s))$ as claimed.

\end{proof}

\subsection{Some preliminary lemmas} \label{sec-somelem}
Recall that $R$ denotes the $k$-algebra of regular functions on the affine $k$-variety $R_u(B^-)\cap L$. In this section, we consider the term $v_t(f(b\delta_\alpha(t)s))$ of the formula of Lemma \ref{formulalemma} individually. More generally, we consider an arbitrary cocharacter $\delta:\GG_m\to T$ and an arbitrary element $x\in L$ and we study the element
\begin{equation}
f_{\delta,x}(b,t)=f(b\delta(t)x) \in R[t,\frac{1}{t}].
\end{equation}
We view $f_{\delta,x}$ as a rational function in $t$ with coefficients in $R$ by taking the element $b$ in the above formula as a generic element of $R_u(B^-)\cap L$. We denote again by $v_t(-)$ the $t$-adic valuation on $R[t,\frac{1}{t}]$. Recall the following lemma:
\begin{lemma}\label{extlemcochar}
If $\delta$ is $L$-anti-dominant, the morphism $B_L^- \times \GG_m \to  B_L^-, (z,t)\mapsto \delta(t) z \delta(t)^{-1}$ extends (uniquely) to a morphism $B_L^- \times \AA^1 \to B_L^-$. In particular, for each $z\in B_L^-$, the map $\GG_m\to B_L^-$, $t\mapsto \delta(t) z \delta(t)^{-1}$ extends to a map $\AA^1\to B_L^-$.
\end{lemma}
This is a standard fact when "anti-dominant" is replaced by "dominant" and $B_L^-$ is replaced by $B_L$. Hence, this statement is simply a reformulation for a different convention of positivity. This has the following consequence:

\begin{proposition}\label{propvalt}
Let $\delta\in X_*(T)$ be an $L$-anti-dominant cocharacter and let $z\in B^-_L$. Then one has the equality
\begin{equation}
v_t(f_{\delta,x})=v_t(f_{\delta,zx}).
\end{equation}
\end{proposition}

\begin{proof}
By symmetry, using the element $z^{-1}$, it suffices to show $v_t(f_{\delta,zx})\geq v_t(f_{\delta,x})$. Denote by $\Gamma: \AA^1\to B^-_L$ the extension of the map $t\mapsto \delta(t) z \delta(t)^{-1}$ afforded by Lemma \ref{extlemcochar}. We have
\begin{equation}
f_{\delta,zx}(b,t)=f(b\delta(t)zx)=f(b\delta(t)z\delta(t)^{-1}\delta(t)x)=f_{\delta,x}(b\Gamma(t),t)
\end{equation}
From this, the result follows easily. Indeed, if $n:=v_t(f_{\delta,x})$, then the function $\theta(b,t)=t^{-n}f_{\delta,x}(b,t)$ extends to a regular map $\theta:(R_u(B^-)\cap L) \times \AA^1\to \AA^1$. By the above formula, $t^{-n}f_{\delta,zx}(b,t)=\theta(b\Gamma(t),t)$. The map $(b,t)\mapsto \theta(b\Gamma(t),t)$ is regular on $(R_u(B^-)\cap L) \times \AA^1$, hence $t^{-n}f_{\delta,zx}(b,t)$ extends too. This shows that $v_t(f_{\delta,zx})\geq n$, hence the result.
\end{proof}

Note that for all $\alpha\in \Delta \setminus I$, the cocharacter $\delta_\alpha$ given by formula \eqref{deltaalpha} is $L$-anti-dominant. Indeed, it is clear that $\alpha^\vee$ is $L$-anti-dominant, and the Galois action preserves this notion because $L,B,T$ are defined over $\FF_p$. Hence $\delta_\alpha$ is a sum of the $L$-anti-dominant cocharacters ${}^{\sigma^i}\alpha^\vee$, so is $L$-anti-dominant. Hence, we obtain the following formula

\begin{corollary}\label{corvalt}
Let $f\in V(\lambda)$ be any non-zero element. For $\alpha \in \Delta \setminus I$, one has the following formula:
\begin{equation}
v_\alpha (\tilde{f})\sim \sum_{s\in B^-_L(\FF_p)\backslash L(\FF_p)}v_t(f(b\delta_\alpha (t)s)).
\end{equation}
where the function $(b,t)\mapsto f(b\delta_\alpha(t)s)$ on $Y$ is viewed as an element of $R[t,\frac{1}{t}]$.
\end{corollary}

We will also need some results concerning the Bruhat decomposition. We have the decomposition of $L$ into $B^-_L\times B_L$ orbits
\begin{equation}
L(k)=\bigsqcup_{w\in W_L} C_w , \quad C_w:=B^-_L(k) w B_L(k).
\end{equation}
The Frobenius homomorphism $L\to L$ (recall that $L$ is defined over $\FF_p$) sends the stratum $C_w$ to $C_{\sigma w}$, where $\sigma:W_L\to W_L$ is the induced action on $W_L$. It follows that $L(\FF_p)$ decomposes as follows:
\begin{equation}
L(\FF_p)=\bigsqcup_{w\in W_L(\FF_p)} C_w(\FF_p).
\end{equation}
\begin{lemma}\label{lemmaLFp}
For each $w\in W_L(\FF_p)$, one has $C_w(\FF_p)=B_L^-(\FF_p) w B_L(\FF_p)$.
\end{lemma}
\begin{proof}
One inclusion is clear. Conversely, let $x\in C_w(\FF_p)$, and write $x=bwb'$ with $b\in B^-_L(k)$ and $b'\in B(k)$. We deduce $\sigma(b)w\sigma(b')=bwb'$ for any $\sigma\in \Gal(\overline{\FF}_p/\FF_p)$, hence $b^{-1}\sigma(b)\in B_L^-\cap wB_Lw^{-1}$. It follows that the map $\sigma\mapsto b^{-1}\sigma(b)$ defines a $1$-cocycle in $H^1(\FF_p,B^-_L\cap wB_Lw^{-1})$. This cohomology vanishes by Lang's theorem, so we can write $b^{-1}\sigma(b)= a^{-1}\sigma(a)$ with $a\in B_L^-\cap wB_Lw^{-1}$, hence $ab^{-1}\in B_L^-(\FF_p)$. Finally, $x=bwb'=(ba^{-1})w (w^{-1}awb')$ and $w^{-1}awb'\in B_L$. This shows that we can write $x$ as before $x=bwb'$ with the additional condition $b\in B^-_L(\FF_p)$. But then clearly $b'\in B_L(\FF_p)$ too, so the result follows. 
\end{proof}

\begin{lemma}\label{Cwcard}
One has $\#(C_w(\FF_p))=\#(T(\FF_p)) p^{n}$ where $n=\dim(R_u(B_L))+\ell(w_{0,I}w)$.
\end{lemma}

\begin{proof}
We have a surjective map $B_L^-\times B_L \to C_w$, $(b,b')\mapsto bwb'^{-1}$ given by the action of $B_L^-\times B_L$ on $C_w$. It identifies $C_w$ with the quotient of $B_L^-\times B_L $ by the stabilizer $S_w$ of $w$ in $B_L^-\times B_L $. Furthermore, $S_w$ is isomorphic to $B_L^-\cap wB_Lw^{-1}$. Again by Lang's theorem we have $\#(C(\FF_p))=\#(B_L(\FF_p))^2 / \#(S(\FF_p))$. We can write $S_w=B_L^-\cap wB_Lw^{-1}$ as a semi-direct product $S=T\rtimes U_w$ where $U_w$ is smooth  connected unipotent. It is well-known that a smooth connected unipotent group is isomorphic as a variety to $\AA^n$ where $n$ is its dimension. Hence we obtain
\begin{equation}
\#(C_w(\FF_p)) = \frac{\#(T(\FF_p))^2 \times p^{2\dim(R_u(B_L))}}{\#(T(\FF_p))p^{\dim(U_w)}}= \#(T(\FF_p)) p^{\dim(C_w)-\dim(T))}.
\end{equation}
It is well-known that $\dim(C_w)=\dim(B_L)+\ell(w_{0,I}w)$, so the result follows.
\end{proof}

\subsection{The highest weight cone}
Recall that we denote by $V(\lambda)$ the $L$-representation $H^0(B_L\backslash L,\Lscr(\lambda))$. It is known that $V(\lambda)$ is not irreducible in general (for small values of $p$) but contains always a unique irreducible $L$-subrepresentation (\cite[II, Cor.~2.3]{jantzen-representations}). We can decompose this $L$-representation uniquely as a direct sum of $T$-eigenspaces
\begin{equation}
V(\lambda)=\bigoplus_{\chi\in X^*(T)} V(\lambda)_\chi
\end{equation}
where $V(\lambda)_\chi$ denotes the space of $f\in V(\lambda)$ such that $t\cdot f=\chi(t)f$ for all $t\in T$. For our purposes, we will be interested in the unique one-dimensional subspace of $V(\lambda)$ which is stable under the action of $B_L$ (the highest weight subspace). It is the subspace $V(\lambda)_{-w_{0,I}\lambda}\subset V(\lambda)$. We will denote by $f_\lambda$ an arbitrary non-zero element of this subspace.

Fix an $L$-antidominant cocharacter $\delta :\GG_m\to T$. Let $x\in L$ an element. We are interested in the behavior of the integer $v_t(f_\lambda(b \delta(t) x))$ when $x$ varies (as usual, $f_\lambda(b \delta(t) x)$ is viewed as an element of $R[t,\frac{1}{t}]$). By \Cor~\ref{propvalt}, this integer depends only on the image of $x$ in $B_L\backslash L$. But since $f_\lambda$ is a $B_L$-eigenfunction, this valuation $v_t(f_\lambda(b \delta(t) x))$ is also invariant if we change $x$ to $xz$ with $z\in B_L$. We deduce:

\begin{lemma}
The integer $v_t(f_\lambda(b \delta(t) x))$ is constant on each stratum $C_w$, i.e it is independent of the choice of the element $x\in C_w$.
\end{lemma}

Using the Bruhat decomposition, inorder to compute $v_t(f_\lambda(b \delta(t) x))$ for all $x$, it remains to compute $v_t(f_\lambda(b \delta(t) w))$ for $w\in W_L$. For this we are not using the strong $B_L$-equivariance property of $f_\lambda$, but only the $T$-equivariance. Hence, we can simply start with $f\in V(\lambda)_\chi$, any $T$-eigenfunction, for some character $\chi$. Then we have the following result:

\begin{proposition} \label{propvalPS}
Let $\delta\in X_*(T)$ be $L$-anti-dominant and $w\in W_L$. One has the formula
\begin{equation}
v_t(f(b \delta(t) w))=-\langle w\chi,\delta\rangle.
\end{equation}
\end{proposition}

\begin{proof}
Denote by $F(t)\in R[t,\frac{1}{t}]$ the function $F(t)=f(b \delta(t) w)$. For each $z\in k^\times$, we have
\begin{equation}
F(zt)=f(b \delta(zt) w)=f(b \delta(t) \delta(z) w)=f(b \delta(t) w (w^{-1} \delta(z) w))=\chi^{-1}(w^{-1} \delta(z)w)F(t).
\end{equation}
In particular, $F(t)$ must be a monomial $F(t)=rt^n$ for some $r\in R$ and $n\in \ZZ$. Furthermore, we have $z^n=\chi^{-1}(w^{-1} \delta(z)w)$, which implies $n=-\langle \chi,w^{-1}\delta\rangle=-\langle w\chi,\delta\rangle$. The result follows.
\end{proof}

We can finally prove the main theorem of this section. It gives a formula for the valuation of $\tilde{f}_\lambda$.

\begin{theorem}
Let $\lambda\in X^*_{+,I}(T)$ be an $L$-dominant character. Then for each $\alpha\in \Delta \setminus I$, we have the formula
\begin{equation}
v_\alpha (\tilde{f}_\lambda)\sim -\sum_{w\in W_L(\FF_p)} \sum_{i=0}^{r_\alpha-1} p^{i+\ell(w)} \ \langle w\lambda, \sigma^i\alpha^\vee \rangle.
\end{equation}
\end{theorem}

\begin{proof}
Start with Lemma \ref{formulalemma}. Combined with Lemma \ref{lemmaLFp}, we obtain the equation
\begin{equation}
v_\alpha (\tilde{f}_\lambda)\sim \sum_{s\in L(\FF_p)}v_t(f_\lambda(b\delta_\alpha (t)s))=\sum_{w\in W_L(\FF_p)}\#(B_L^-(\FF_p)wB_L(\FF_p)) \ v_t(f_\lambda(b\delta_\alpha (t)w)).
\end{equation}
Now use Proposition \ref{propvalPS} and Lemma \ref{Cwcard}, and then substitute $\delta_\alpha$ by the formula given by equation \eqref{deltaalpha}:
\begin{equation}
v_\alpha (\tilde{f}_\lambda)\sim - \sum_{w\in W_L(\FF_p)} \sum_{i=0}^{r_\alpha-1} p^{i+\ell(w_{0,I}w)} \ \langle ww_{0,I}\lambda, \sigma^i\alpha^\vee \rangle.
\end{equation}
Making the change of variable $w\to ww_{0,I}$ (and noting that $\ell(w_{0,I}ww_{0,I})=\ell(w)$), we obtain the desired formula.
\end{proof}

\begin{corollary}\label{lamisin}
Assume $\lambda \in X^*_{+,I}(T)$ satisfies the inequalities
\begin{equation}\label{ineqhw}
\sum_{w\in W_L(\FF_p)} \sum_{i=0}^{r_\alpha-1} p^{i+\ell(w)} \ \langle w\lambda, \sigma^i\alpha^\vee \rangle\leq 0, \quad \forall \alpha \in \Delta \setminus I.
\end{equation}
Then $\lambda \in \langle C_{\zip}\rangle$.
\end{corollary}

\begin{proof}
Indeed, if these integers are non-positive, then the divisor of $\tilde{f}_\lambda$ is effective on $G$, hence this function extends to a global section. Its weight is a positive multiple of $\lambda$, which proves the result.
\end{proof}

\begin{definition}
We define the highest weight cone $C_{\rm hw}\subset X^*_{+,I}(T)$ as the set of characters $\lambda$ satisfying the inequalities \eqref{ineqhw}.
\end{definition}
It is clear that $C_{\rm hw}$ is saturated and we have just proved that $C_{\rm hw}\subset \langle C_{\zip}\rangle$. Furthermore, we claim that $C_{\rm GS}\subset C_{\rm hw}$. To show this, take $\lambda \in C_{\GS}$ and $\alpha \in \Delta \setminus I$. For such a character, all the summands in equation \eqref{ineqhw} are non-positive. Indeed, for each $i=0,...,r_\alpha-1$ and for each $w\in W_L(\FF_p)$, one has $\langle w\lambda,{}^{\sigma^i}\alpha^\vee\rangle = \langle \lambda,(w^{-1}({}^{\sigma^i}\alpha))^\vee\rangle $. The root ${}^{\sigma^i}\alpha$ is clearly in $\Delta \setminus I$ because $B,L$ are defined over $\FF_p$. Hence $w^{-1}({}^{\sigma^i}\alpha))$ is a positive root, possibly non-simple, which is not a root of $L$. By definition of $C_{\rm GS}$, the claim follows. We have proved the following corollary, which shows that \Conj~\ref{conjGS} holds under the assumption that $P$ is defined over $\FF_p$.

\begin{corollary}\label{inclusions}
One has the inclusions $C_{\rm GS}\subset C_{\rm hw}\subset \langle C_{\zip}\rangle$.
\end{corollary}

To summarize, our construction attaches to any character $\lambda$ in the cone $C_{\rm hw}$ an element $\tilde{f}_\lambda\in H^0(\GZip, \Vscr(D\lambda))$. We will see examples (\S\ref{sec-sp6}) of the association
\begin{equation}\label{assoclambda}
\lambda \rightsquigarrow \tilde{f}_\lambda.
\end{equation}

\subsection{The $\FF_p$-split case}
Assume now that $G$ is $\FF_p$-split. Hence the action of Galois on $X^*(T)$ is trivial. In particular, we have $r_\alpha =1$ for all $\alpha \in \Delta\setminus I$. We deduce that the cone $C_{\rm hw}$ is given by the set of $\lambda\in X_{+,I}^*(T)$ such that
\begin{equation}\label{ineqFp}
\sum_{w\in W_L} p^{\ell(w)}  \langle w\lambda, \alpha^\vee \rangle\leq 0, \quad \forall \alpha \in \Delta \setminus I.
\end{equation}
We can simplify further this formula. For $\alpha\in \Delta \setminus I$, denote by $L_\alpha \subset L$ the centralizer in $L$ of the cocharacter $\alpha^\vee$. The type $I_\alpha$ of $L_\alpha$ is the orthogonal in $I$ of $\alpha^\vee$. For any $w\in W_{L_\alpha}$, we have $w\alpha=\alpha$ and $w\alpha^\vee=\alpha^\vee$. By \cite[(2.12)]{Bjorner-Brenti-book}, any element $w\in W_L$ can be written uniquely in the form
\begin{equation}
w=w_\alpha w^\alpha, \quad w_\alpha \in W_{L_\alpha}, w^\alpha \in {}^{I_\alpha} W_{L}
\end{equation}
(recall \S\ref{zipstrata} for the definition of the sets ${}^K W$). Furthermore, one has $\ell(w)=\ell(w_\alpha)+\ell(w^\alpha)$. Hence equation \eqref{ineqFp} can be written as
\begin{equation}\label{ineqsimple}
\sum_{w_\alpha}\sum_{w^\alpha} p^{\ell(w_\alpha)+\ell(w^\alpha)}  \langle w_\alpha w^\alpha \lambda, \alpha^\vee \rangle\leq 0
\end{equation}
where $w_\alpha$ runs in the set $W_{L_\alpha}$ and $w^\alpha$ in ${}^{I_\alpha}W_{L}$. But $\langle w_\alpha w^\alpha \lambda, \alpha^\vee\rangle =\langle  w^\alpha \lambda,w^{-1}_\alpha \alpha^\vee\rangle =\langle  w^\alpha \lambda,\alpha^\vee\rangle$. We deduce that this inequality boils down to
\begin{equation}\label{ineqfinal}
\sum_{w\in {}^{I_\alpha} W_{L}} p^{\ell(w)}  \langle w \lambda, \alpha^\vee \rangle\leq 0
\end{equation}
Comparing to the initial equation, we have simply reduced the indexing set. We will use this simplified formula in the case of a symplectic group $G=Sp(2n)$.

\subsection{Determination of the space $H^0(\GZip,\Vscr(\lambda))$} \label{sec-zipreptheo}
In this section, we continue to assume that $P$ (and hence also $L$) is defined over $\FF_p$. We assume further that $T$ is split over $\FF_p$ to simplify the arguments. Our goal is to determine the space $H^0(\GZip,\Vscr(\lambda))$ uniquely in terms of representation theory of reductive groups. This makes it possible to give a formulation of the cone $C_{\zip}$ which makes no mention of the geometric objects $\GZip$, $\GF$.

We first determine the sections over the ordinary locus. This can be done very generally. Let $(\rho,V)$ be an algebraic representation $\rho:P\to GL(V)$. Recall the construction explained in \S\ref{sec-vector-bundles-gzipz} which attaches a vector bundle $\Vscr(\rho)$ on $\GZip$.
\begin{proposition}\label{ordsecprop}
There is an isomorphism
\begin{equation}
    H^0(\GZip^{\mu\textrm{-ord}},\Vscr(\rho))\simeq V^{L(\FF_p)}.
\end{equation}
\end{proposition}
\begin{proof}
Recall (\S\ref{sec-ordsec}) that $\GZip^{\mu \textrm{-ord}}\simeq [E\backslash U_\mu]$ where $U_\mu$ denotes the unique open $E$-orbit in $G$. We have $1\in U_\mu$ and $\Stab_E(1)\simeq L(\FF_p)$ (this group is viewed as an etale subgroup of $E$ by the map $a\mapsto (a,a)$). Hence $\GZip^{\mu \textrm{-ord}}\simeq [E\backslash \left(E/ L(\FF_p) \right)]\simeq [1/L(\FF_p)]$. The result follows immediately.
\end{proof}

For $\lambda\in X_{+,L}^*(T)$, recall that we have the $L$-representation $V(\lambda)$ defined in \eqref{repV}. We may decompose it in a sum of $T$-eigenspaces:
\begin{equation}\label{Teigen}
V(\lambda)=\bigoplus_{\chi \in X^*(T)} V(\lambda)_\chi.
\end{equation}
Define a subspace $V(\lambda)_{\leq 0}\subset V(\lambda)$ as follows. It is the direct sum of the $T$-eigenspaces $V(\lambda)_\chi$ for the characters $\chi\in X^*(T)$ which satisfy the condition
\begin{equation}
     \langle \chi, \alpha^\vee \rangle \leq 0 \quad \textrm{ for all }\alpha\in \Delta\setminus I.
\end{equation}
Note that $V(\lambda)_{\leq 0}$ is stable under the action of $T$, but it is not a sub-$L$-representation of $V(\lambda)$.

\begin{theorem}\label{thminter}
There is a commutative diagram where the vertical maps are the natural inclusions, and the horizontal maps are isomorphisms:
\begin{equation}
\xymatrix@1@M=7pt{
H^0(\GZip^{\mu \textrm{-ord}},\Vscr(\lambda)) \ar[r]^-{\simeq} & V(\lambda)^{L(\FF_p)}  \\
H^0(\GZip,\Vscr(\lambda)) \ar[r]^-{\simeq}  \ar@{^{(}->}[u] & V(\lambda)_{\leq 0}\cap V(\lambda)^{L(\FF_p)}  \ar@{^{(}->}[u]}
\end{equation}
\end{theorem}

\begin{corollary}
The zip cone $C_{\zip}$ is the subset of $X^*_{+,I}(T)$ given by
\begin{equation}
    C_{\zip}=\{\lambda\in X^*_{+,I}(T) \ | \ V(\lambda)_{\leq 0}\cap V(\lambda)^{L(\FF_p)}\neq 0\}
\end{equation}
\end{corollary}
Note that this characterization of the zip cone makes no reference to the stack of $G$-zips and is only formulated in terms of the representation $V(\lambda)$. We may view this $L$-representation as an $L(\FF_p)$-representation or a $T$-representation. It is apparent that the zip cone is related to the interaction between these two points of view.

We now prove \Th~\ref{thminter}. The top line of the diagram is \Prop~\ref{ordsecprop} applied to the representation $V(\lambda)$. To prove the lemma, we need to show that via this isomorphism, the image of the subspace $H^0(\GZip,\Vscr(\lambda))$ is $V(\lambda)_{\leq 0}\cap V(\lambda)^{L(\FF_p)}$. Recall (see \eqref{constrftilde}) that we have a nonlinear map $f\mapsto \tilde{f}$
\begin{equation}
V(\lambda) \longrightarrow V(D\lambda)^{L(\FF_p)}\simeq H^0(\GZip^{\mu\textrm{-ord}},\Vscr(D\lambda))=H^0(\GF^{\mu\textrm{-ord}},\Lscr(D\lambda))
\end{equation}
where $D=|L(\FF_p)|$. Then we may compose this map with the valuations $v_\alpha : H^0(\GF^{\mu\textrm{-ord}},\Lscr(D\lambda)) \to \ZZ$ defined in \eqref{valpha}, for any $\alpha\in \Delta\setminus I$. The sign of $v_\alpha(\tilde{f})$ is given by $\sum_{s\in L(\FF_p)}v_t(f(b\delta_\alpha (t)s))$ for any nonzero $f\in V(\lambda)$ (\Lem~\ref{formulalemma}). Since we assume $T$ to be $\FF_p$-split, $\delta_\alpha$ is just a negative multiple of $\alpha^\vee$, so the sign is that of $-\sum_{s\in L(\FF_p)}v_t(f(b\alpha^\vee(t)s))$.
Note that if we start with $f\in V(\lambda)^{L(\FF_p)}$, then by definition $\tilde{f}=f^D$ and all the terms in the sum are equal, so $v_\alpha(\tilde{f})\sim -v_t(f(b\alpha^\vee(t)))$.

Now, let $f\in H^0(\GZip^{\mu\textrm{-ord}},\Vscr(\lambda))$ be a nonzero section, which we can view as an element of $V(\lambda)^{L(\FF_p)}$. Since $G$ is smooth, $f$ extends to $\GZip$ if and only if $f^D=\tilde{f}$ does. This is the case if and only if its divisor is effective, which means exactly $v_\alpha(\tilde{f})\geq 0$ for all $\alpha\in \Delta\setminus I$. We deduce that $f\in H^0(\GZip,\Vscr(\lambda))$ if and only if $-v_t(f(b\alpha^\vee(t)))\geq 0$ for all $\alpha\in \Delta\setminus I$. To continue the argument, we need the following linear algebra lemma:

\begin{lemma}\label{lemmauniqueF}
Let $V$ be a $k$-vector space which decomposes as $V=\bigoplus_{\chi} V_\chi$. For each $\chi$ such that $V_\chi\neq 0$, let $n(\chi)\in \ZZ$ be an integer. Then there is a unique function $F:V\setminus \{0\}\to \ZZ$ which satisfies:
\begin{enumerate}
\item\label{itemF1} $F(x)=n(\chi)$ for all $x\in V_\chi\setminus \{0\}$.
\item\label{itemF2} $F(ax)=F(x)$ for all $a\neq 0$ and all $x\in V\setminus \{0\}$.
\item\label{itemF3} $F(x+y)\geq \min\{F(x),F(y)\}$ for all $x,y\in V\setminus \{0\}$, $x+y\neq 0$.
\end{enumerate}
\end{lemma}

\begin{proof}
Assume $F$ is a function satisfying these properties. Define $V_{\geq n}:=\{x\in V\setminus\{0\} \ | \ F(x)\geq n\}\cup \{0\}$. By \eqref{itemF2} and \eqref{itemF3}, it is clear that $V_{\geq n}$ is a subspace of $V$ and we have a filtration
\begin{equation}\label{filtV}
...\subset V_{\geq n+1} \subset V_{\geq n} \subset V_{\geq n-1} \subset ...
\end{equation}
If $V_\chi \cap V_{\geq n}\neq 0$, then one has $V_\chi\subset V_{\geq n}$, so we deduce $V_{\geq n}=\bigoplus_{n(\chi)\geq n} V_\chi$. Hence this filtration is independent of $F$. The unique extension $F$ is given by 
\begin{equation}\label{filF}
F(x)=\max\{n\in \ZZ \ | \ x\in V_{\geq n}\}.
\end{equation}
\end{proof}

Fix $\alpha\in \Delta\setminus I$. We claim that the function $V(\lambda)\to \ZZ$,  $f\mapsto v_t(f(b\alpha^\vee(t)))$ satisfies the properties \eqref{itemF1}, \eqref{itemF2}, \eqref{itemF3} above for the decomposition \eqref{Teigen} and the function $n(\chi)=-\langle \chi, \alpha^\vee \rangle$. The first one follows from \Prop~\ref{propvalPS} with $w=1$ and $\delta=\alpha^\vee$. The second one is immediate, and the third one comes from the valuation property of the function $v_t$. From \eqref{filF}, we deduce that for all $f\in V(\lambda)\setminus \{0\}$, we have
\begin{equation}
    v_t(f(b\alpha^\vee(t))) = \max\{n\in \ZZ \ | \ x\in V_{\geq n}^{\alpha}\}
\end{equation}
where $V^{\alpha}_{\geq n}=\bigoplus_{-\langle \chi, \alpha^\vee \rangle\geq n} V(\lambda)_\chi$. We deduce for any $\alpha\in \Delta \setminus I$:
\begin{equation}
     -v_t(f(b\alpha^\vee(t)))\geq 0 \Longleftrightarrow f\in V^\alpha_{\geq 0}.
\end{equation}
Finally, for all $f\in V(\lambda)^{L(\FF_p)}$, we have shown that
\begin{equation}
   f\in H^0(\GZip,\Vscr(\lambda)) \Longleftrightarrow f\in \bigcap _{\alpha \in \Delta\setminus I}V^\alpha_{\geq 0} = V(\lambda)_{\leq 0}
\end{equation}
and this terminates the proof of \Th~\ref{thminter}.

In the next proposition, we will see that the space of global sections simplifies when the weight lies in the Griffiths-Schmid cone.
\begin{proposition}
Let $\lambda \in C_{\GS}$. Then the following hold true
\begin{enumerate}
    \item\label{itemGS1} One has $V(\lambda)_{\leq 0}=V(\lambda)$.
    \item\label{itemGS2} Any section of $\Vscr(\lambda)$ over the ordinary locus $\GZip^{\mu\textrm{-ord}}$ extends to $\GZip$.
    \item\label{itemGS3} One has an isomorphism $  H^0(\GZip,\Vscr(\lambda))\simeq V(\lambda)^{L(\FF_p)}$.
\end{enumerate}
\end{proposition}

\begin{proof}
It is clear that assertions \eqref{itemGS2} and \eqref{itemGS3} follow immediately from \eqref{itemGS1} combined with \Th~\ref{thminter}. It remains to show \eqref{itemGS1}. For this, we use \cite[\Prop~2.2 b)]{jantzen-representations} which states that any weight $\chi$ appearing in $V(\lambda)$ satisfies $w_{0,L}\lambda\leq \chi \leq \lambda$. Hence we can write $\chi=w_{0,L}\lambda+\sum_{i=1}^d n_i \alpha_i$ where $\alpha_1,...,\alpha_d\in \Delta_L=I$ and $n_1,...,n_d\in \ZZ_{\geq 0}$. It is well-known that if $\alpha,\beta$ are two distinct simple roots, then $\langle \alpha,\beta^\vee \rangle \leq 0$. Hence if $\alpha\in \Delta\setminus I$, then one has $\langle \alpha_i , \alpha^\vee \rangle \leq 0$ for all $i=1,...,d$. Furthermore, one has $\langle w_{0,L}\lambda , \alpha^\vee \rangle=\langle \lambda , w_{0,L}\alpha^\vee \rangle$ and $w_{0,L}\alpha$ is again a positive root. By definition of $C_{\GS}$, this expression is non-positive. It follows that $\langle\chi,\alpha^\vee \rangle\leq 0$. Since this is true for all weights $\chi$ in $V(\lambda)$ and all $\alpha\in \Delta\setminus I$, we have by definition $V(\lambda)_{\leq 0}=V(\lambda)$.
\end{proof}

\section{The symplectic case}\label{section symplectic}
This section is devoted to the case where $G$ is the symplectic group $Sp_{2n}$ and the zip datum is the usual Hodge-type one. For $n=2,3$, we will give moduli interpretations for the sections that we construct.

\subsection{Preliminaries}
\subsubsection{The group $G$}\label{thegroup}
In this section, we focus on the case when $G$ is the reductive $\FF_p$-group $Sp(V,\Psi)$, where $(V,\Psi)$ is a non-degenerate symplectic space over $\FF_p$ of dimension $2n$, for some integer $n\geq 1$. After choosing an appropriate basis $\Bcal$ for $V$, we assume that $\Psi$ is given in this basis by the matrix
\begin{equation}
\Psi:=\left(\begin{matrix}
& -J \\
J&\end{matrix}\right)
\quad \textrm{ where } \quad
J:=\left(\begin{matrix}
&&1 \\
&\iddots& \\
1&&
\end{matrix}\right)
\end{equation} 
The group $G$ is then defined by:
\begin{equation}\label{group}
G(R):=\{g\in GL_{2n}(R), {}^t g \Psi g =\Psi \}
\end{equation}
for all $\FF_p$-algebras $R$. An $\FF_p$-split maximal torus $T$ is given by the diagonal matrices in $G$, specifically:
\begin{equation}
T(R):=\{ \diag_{2n}(x_1,...,x_n,x^{-1}_n,...,x^{-1}_1), \ \ x_1,..,x_n\in R^\times \}
\end{equation}
where $\diag_d(a_1,...,a_d)$ denote the diagonal matrix of size $d$ with diagonal coefficients $a_1$,...,$a_d$. Define the Borel $\FF_p$-subgroup $B$ of $G$ as the set of the lower-triangular matrices in $G$. For a tuple $(a_1,...,a_n)\in \ZZ^n$, we define a character of $T$ by mapping $\diag_{2n}(x_1,...,x_n,x^{-1}_n,...,x^{-1}_1)$ to $x^{a_1}_1 ... x^{a_n}_n$. We obtain an identification $X^*(T) = \ZZ^n$. Denoting by $(e_1,...,e_n)$ the standard basis of $\ZZ^n$, the $T$-roots of $G$, the $B$-positive roots are given respectively by:
\begin{align}
\Phi&:=\{e_i \pm e_j , 1\leq i \neq j \leq n\} \cup \{ \pm 2e_i, 1\leq i \leq n \}\\
\Phi_+&:=\{e_i \pm e_j , 1\leq i< j \leq n\} \cup \{ 2e_i, 1\leq i \leq n\}
\end{align}
and the $B$-simple roots are $\Delta:=\{\alpha_1,..., \alpha_{n-1},\beta\}$ where:
\begin{align*}
\alpha_i&:=e_{i+1}-e_i \textrm{ for } i=1,...,n-1 \\ \beta&:=2e_n.
\end{align*}
The Weyl group $W:=W(G,T):=N_G(T)/T$ can be identified with the group of permutations $\sigma \in \Sfr_{2n}$ satisfying the relation $\sigma(i)+\sigma(2n+1-i)=2n+1$ for all $1\leq i \leq 2n$.

\subsubsection{Zip datum}
Let $(u_i)_{i=1}^{2n}$ the canonical basis of $k^{2n}$. Define a parabolic subgroup $P\subset G$ containing $B$ as the stabilizer of $\Span_k(u_{n+1},...,u_{2n})$. Similarly, denote by $Q$ the opposite parabolic subgroup of $P$ with respect to $T$, it is the stabilizer of $\Span_k(u_{1},...,u_{n})$. The intersection $L:=P\cap Q$ is a common Levi subgroup and there is an isomorphism $GL_{n,\FF_p}\to L$, $M\mapsto S(M)$, where:
\begin{equation}
S(M):=\left(\begin{matrix}\label{Sdef}
M& \\ & -J {}^t\! M^{-1} J
\end{matrix} \right), \quad M\in GL_n.
\end{equation}
The tuple $\Zcal:=(G,P,L,Q,L,\varphi)$ defines a zip datum. Since $P$ is defined over $\FF_p$, the open $E$-orbit $U_\mu \subset G$ coincides with the open $P\times Q$-orbit (\cite[\Cor~2.12]{Wedhorn-bruhat}). Hence
\begin{equation}
U_\mu=\left\{\left(\begin{matrix}
A&*\\
*&*
\end{matrix} \right)\in G, \ A\in GL_n, \ * \in M_n \right\}.
\end{equation}

The (ordinary) Hasse invariant $H_\mu\in H^0(\GZip^\Zcal,\omega^{p-1})$ is the section given by
\begin{equation}
H_\mu:\left(\begin{matrix}
A&*\\
*&*
\end{matrix} \right) \mapsto \det(A).
\end{equation}

\subsection{The cones $C_{\Sbt}$, $C_{\GS}$, $C_{\rm hw}$.}

Our ultimate goal is to determine the cone $\langle C_{\zip} \rangle$, but this is very difficult. The Schubert, Griffith-Schmidt and highest weight cones are all subcones of $\langle C_{\zip} \rangle$ and they give good approximations from below of the zip cone. First recall that all the cones that we consider are contained in $X^*_{+,I}(T)$, the cone of $L$-dominant characters. In our case, this set identifies with:
\begin{equation}
X^*_{+,I}(T)=\{(a_1,...,a_n), \ a_1\geq a_2 \geq ... \geq a_n \}.
\end{equation}
Since our zip datum is of Hodge-type, we have a Hodge character. It is given by $\eta_\omega = -(p-1,...,p-1)$. Recall the definition of the Schubert cone (\Def~\ref{defSbtcone}). The map $h:X^*(T)\to X^*(T)$ is given in this case by:
\begin{equation}
h:\ZZ^n\to \ZZ^n, \quad (a_1,...,a_n)\mapsto (a_1,...,a_n)-p(a_n,...,a_1).
\end{equation}
\begin{lemma} \label{SchubSp} \ 
\begin{enumerate}
\item The Schubert cone $\langle C_{\Schub} \rangle$ is the set of $\lambda=(a_1,...,a_n)\in \ZZ^n$ satisfying the inequalities
\begin{align*}
(p a_{i+1}+a_{n-i})-(p a_{i}+a_{n+1-i}) &\leq 0 \quad  \textrm{for all }i=1,...,n-1, \\
pa_1+a_n & \leq 0.
\end{align*}
\item The cone $\langle C_{\Schub}\rangle$ is s-generated (see \S\ref{coneter}) by $\eta_\omega=-(p-1,...,p-1)$ and the weights $S_1,...,S_{n-1}$ defined by
\begin{equation}\label{Schubweight}
S_i=(1,...,1,0,...,0)-(0,...,0,p,...,p)
\end{equation}
where both $1$ and $p$ appear $i$ times and $0$ appears $n-i$ times.
\end{enumerate}
\end{lemma}

\begin{proof}
The first part is \Lem~\ref{lemmasplitSchub}. For the second assertion, note that $X^*_+(T)$ is generated by the fundamental weights $(\mu_i)_{1\leq i\leq n}$ given by $\mu_i=(1,...,1,0,...,0)$ where $1$ appears $i$ times and $0$ appears $n-i$ times. Using the notation of \S\ref{schubcone}, $\langle C_{\Sbt}\rangle$ is generated by $(h(\mu_i))_{1\leq i \leq n}$. Hence the result.
\end{proof}
Next, we determine explicitly the Griffith-Schmidt cone. Using Definition \ref{GSdef}, we have:
\begin{equation}
C_{\GS}=\{(a_1,...,a_n), \ 0\geq a_1\geq a_2 \geq ... \geq a_n \}.
\end{equation}
In general, there is no inclusion between $C_{\GS}$ and $C_{\Sbt}$. The equivalent conditions of Lemma \ref{lemzipSchub} are satisfied if and only if $n=2$. Indeed, since $G$ is $\FF_p$-split, the Galois action on $\Delta$ is trivial, so condition \eqref{item-frob} is satisfied if and only if $-w_{0,L}$ acts trivially on $I$, which is the case only for $n=2$. If $n>2$, then in particular $C_{\GS}$ is not contained in $C_{\Sbt}$. Conversely, note that none of the $S_i$ (for $1\leq i \leq n-1$) of Lemma \ref{SchubSp} is contained in $C_{\GS}$. We now determine the highest weight cone.

\begin{proposition}
One has the following
\begin{equation}
C_{\rm hw}=\left\{ (a_1,...,a_n)\in \ZZ^n, \ \sum_{i=1}^n p^{n-i}a_i \leq 0 \right\}.
\end{equation}
\end{proposition}

\begin{proof}
We can use equation \ref{ineqfinal}, which is available in the $\FF_p$-split case. Here $\Delta\setminus I=\{\beta\}$ where $\beta=2e_n$ (notations in \S\ref{thegroup}). The centralizer $L_\beta$ of $\beta^\vee$ in $L\simeq GL_n$ is the subgroup of the form
\begin{equation}
L_\alpha = \left\{ \left( \begin{matrix}
A & 0\\ 0 & a
\end{matrix} \right), \ A\in GL_{n-1}, a\in \GG_m\right\}.
\end{equation}
We deduce that the set ${}^{I_\alpha} W_L$ identifies with the set of permutations $\sigma \in \Sfr_n$ satisfying $\sigma^{-1}(1)<\sigma^{-1}(2)<...<\sigma^{-1}(n-1)$. Such an element is entirely determined by the value $\sigma^{-1}(n)$. One can see that its length is then equal to $n-\sigma^{-1}(n)$. By indexing the sum on the value of $i=\sigma^{-1}(n)$, we obtain $\sum_{i=1}^n p^{n-i}\langle \lambda, e_i \rangle =\sum_{i=1}^n p^{n-i}a_i$ as claimed.
\end{proof}

As predicted by Corollary \ref{inclusions}, the inclusion $C_{\GS}\subset C_{\rm hw}$ is clear from the equations. For $n>2$, there is no inclusion between $C_{\rm hw}$ and $C_{\Sbt}$. Indeed, the inclusion $C_{\rm hw}\subset C_{\Sbt}$ does not hold because we saw that $C_{\GS}$ is not contained in $C_{\Sbt}$. Conversely, the weights $S_1,...,S_{n-2}$ from \Lem~\ref{SchubSp} are not in $C_{\rm hw}$. However, the weight $S_{n-1}$ lies on the boundary of $C_{\rm hw}$ because it satisfies $\sum_{i=1}^n p^{n-i}a_i=0$.

Next, we find an s-generating set for the cone $C_{\rm hw}$ (recall the terminology \S\ref{coneter}). Denote by $(u_1,...,u_n)$ the canonical $k$-basis of $k^n$. For each $1\leq i \leq n-1$, let $R_{i}\subset GL_n$ be the parabolic subgroup stabilizing the subspace $\Span(u_{1},...,u_i)$. Denote by $\binom{n}{i}_p$ the Gaussian binomial coefficient
\begin{equation}
\binom{n}{i}_p:=\left|GL_n(\FF_p)/R_i(\FF_p) \right|=\frac{(p^{i+1}-1)(p^{i+2}-1)...(p^n-1)}{(p-1)(p^{2}-1)...(p^{n-i}-1)}.
\end{equation}
It is easy to see that $\binom{n}{i}_p$ is a monic polynomial in $p$. Evaluating formally at $p=1$, one recovers the usual binomial coefficient $\binom{n}{i}$. 
For each $1\leq i \leq n-1$, denote by $\eta_i$ the weight $\eta_i=(a,...,a,b,...,b)$ where $a=\binom{n-1}{i}_p$ appears $i$ times and $b=-p^{n-i}\binom{n-1}{i-1}_p$ appears $n-i$ times.

\begin{lemma}\label{hwgen}
One has $C_{\rm hw}=\langle \eta_1,...,\eta_{n-1},\eta_\omega\rangle$.
\end{lemma}

\begin{proof}
After scaling, we see that $\eta_i$ is a positive multiple of $\eta'_i$ defined by $\eta'_i=(1,...,1,b',...,b')$ where $1$ appears $i$ times and $b'=b/a=-\frac{p^n-p^{n-i}}{p^{n-i}-1}$ appears $n-i$ times. It is easy to see that all the weights $\eta'_i$ (hence also the $\eta_i$) lie on the hyperplane $H$ of $X^*(T)\otimes \QQ=\QQ^n$ defined by $\sum_{j=1}^n p^{n-j}a_j=0$. Indeed, this follows from the relation
\begin{equation}
\sum_{j=1}^{i} p^{n-j}-\frac{p^n-p^{n-i}}{p^{n-i}-1} \sum_{j=i+1}^n p^{n-j}=p^{n-i}\left(\frac{p^i-1}{p-1}\right)-\left(\frac{p^n-p^{n-i}}{p^{n-i}-1}\right) \left(\frac{p^{n-i}-1}{p-1}\right) =0.
\end{equation}
In particular, this shows that all the weights $\eta_1,...,\eta_{n-1}$ are contained in $C_{\rm hw}$. We also know that $\eta_w\in C_{\rm hw}$, so this proves one inclusion. Conversely, we show that any element in $C_{\rm hw}$ can be written as $\sum_{i=1}^{n-1} \lambda_i \eta'_i+\lambda_\omega \eta_\omega$ with $\lambda_1,...,\lambda_n,\lambda_\omega\in \QQ_{\geq 0}$. First, we show that $(\eta'_1,...,\eta'_{n-1},\eta_\omega)$ is a $\QQ$-basis of $\QQ^n$. For this, it suffices to show that $(\eta'_1,...,\eta'_{n-1})$ generate $H$. Note that if $v=(v_1,...,v_n)$ is any vector of $H$ such that $v=\sum_{i=1}^{n-1}\lambda_i \eta'_i=0$ for $\lambda_i \in \QQ$, then we deduce immediately $v_{i}-v_{i+1}=\frac{p^n-1}{p^{n-i}-1}\lambda_i$ for all $i=1,...,n-1$. In particular, by taking $v=0$, this shows easily that $(\eta'_1,...,\eta'_{n-1})$ are linearly independent. Since $H$ has dimension $n-1$, it follows that $(\eta'_1,...,\eta'_{n-1})$ generate $H$ over $\QQ$.

Now let $v=(v_1,...,v_n)\in C_{\rm hw}$ and write $v=\sum_{i=1}^{n-1} \lambda_i \eta'_i+\lambda_\omega \eta_\omega$ with $\lambda_1,...,\lambda_{n-1},\lambda_\omega \in \QQ$. We have again $v_{i}-v_{i+1}=\frac{p^n-1}{p^{n-i}-1}\lambda_i \geq 0$ for all $i=1,...,n-1$. Since $v\in X^*_{+,I}(T)$, we have $v_1\geq v_2\geq ... \geq v_n$, so $\lambda_i \geq 0$. Finally, denote by $\psi$ the linear form $\psi(a_1,...,a_n)=\sum_{i=1}^n p^{n-i}a_i$. We have $\psi(v)=\lambda_\omega \psi(\eta_\omega)\leq 0$, hence $\lambda_\omega\geq 0$ as claimed. This terminates the proof.
\end{proof}

\subsection{The polynomial cone}

In the previous section, we obtained subcones of $\langle C_{\zip}\rangle$, which thus provide a lower bound for this cone. In this section, we want to determine an upper bound for the zip cone. Generally speaking, this tends to be more difficult. If $g\in G$, write $g$ as a block matrix of the form
\begin{equation}
g=\left(\begin{matrix}
A(g)&B(g)\\C(g)&D(g)
\end{matrix} \right), \quad A(g),B(g),C(g),D(g)\in M_n.
\end{equation}
This defines four regular functions $A,B,C,D:G\to M_n$.

\begin{lemma}\label{lemmapol} Let $\lambda \in X^*_{+,I}(T)$ and let $f:U_\mu\to \AA^1$ be an element of $H^0(\GF^{\mu\textrm{-ord}},\Lscr(\lambda))$.
\begin{enumerate}
\item \label{item-pol1} There exists a unique regular function $f_0:GL_n\to \AA^1$ such that $f(g)=f_0(A(g))$ for all $g\in U_\mu$.
\item\label{item-pol2} The section $f$ extends to $\GF$ if and only if $f_0$ extends to a function $M_n\to \AA^1$.
\end{enumerate}
\end{lemma}

\begin{proof}
We prove \eqref{item-pol1}. Define $f_0:GL_n\to \AA^1$ by $f_0(A):=f(S(A))$ (see equation \eqref{Sdef} for the definition of $S$). Let $g\in U_\mu$ and write $g=ab^{-1}$ with $(a,b)\in E$. We have $f(g)=f(ab^{-1})=f(\overline{a}(\overline{b})^{-1})$ where $\overline{a},\overline{b}\in L$ are the Levi components of $a,b$ respectively. It is easy to check that $\overline{a}(\overline{b})^{-1}=S(A(g))$, so we deduce $f(g)=f(S(A(g)))=f_0(A(g))$. The uniqueness of $f_0$ follows from the fact that the map $A:G\to M_n$ is surjective (we leave this to the reader).

We prove assertion \eqref{item-pol2}. Suppose that $f_0$ extends to $f_0:M_n\to \AA^1$. Then extend $f$ to a map $G\to \AA^1$ by setting $f(g):=f_0(A(g))$ for all $g\in G$. It is clearly regular, and remains $E'$-equivariant by a density argument. Hence $f$ extends to an element  of $ H^0(\GF,\Lscr(\lambda))$. Conversely, suppose that $f$ extends to $\GF$. Viewed as a regular function on $U_\mu$, this amounts to saying that $f$ extends to $G$. We will show that $f_0:GL_n\to \AA^1$ extends to $M_n$. Denote by $r\in \ZZ$ the multiplicity of $f_0$ along the complement of $GL_n$ in $M_n$ and assume $r<0$. Then $f'_0:=f_0 \det^{-r}$ extends to a function $M_n\to \AA^1$ which does not vanish at all points of $M_n\setminus GL_n$. Define $f':=f h_\mu^{-r}$ where $h_\mu$ is the ordinary Hasse invariant. Then $f'$ vanishes everywhere on $G\setminus U_\mu$, and one has $f'(g)=f'_0(A(g))$ for all $g\in G$. Since $A:G\to M_n$ is surjective, the function $f'_0$ vanishes everywhere on $M_n\setminus GL_n$, which is a contradiction. This terminates the proof of the lemma.
\end{proof}

Denote by $[B_L \backslash_\sigma GL_n]$ and $[B_L \backslash_\sigma M_n]$ the quotients stacks where $B_L$ acts on $M_n$ and $GL_n$ by $b\cdot g=bg\varphi(b)^{-1}$. Again, for any $\lambda \in X^*(T)$, there is a line bundle on these stacks naturally attached to $\lambda$. We continue to denote it by $\Lscr(\lambda)$. The following is a reformulation of the previous lemma:

\begin{corollary}\label{corMn}
The map $f\mapsto f_0$ given by \Lem~\ref{lemmapol} defines an isomorphism
\begin{equation}
H^0(\GF^{\mu\textrm{-ord}}, \Lscr(\lambda)) \to H^0([B_L \backslash_\sigma GL_n], \Lscr(\lambda))
\end{equation}
which maps the subspace $H^0(\GF, \Lscr(\lambda))$ to the subspace $H^0([B_L \backslash_\sigma M_n], \Lscr(\lambda))$.
\end{corollary}

In particular, any function $f\in H^0(\GF, \Lscr(\lambda))$ gives rise to a regular function $M_n\to \AA^1$, i.e a polynomial in the coefficients of the $n\times n$ matrix. Consider the action of the diagonal torus on $M_n$ by $\sigma$-conjugation $t\cdot A=tA\sigma(t)^{-1}$. For this action, the coefficients functions of $M_n$ are $T$-eigenfunctions and the weight of the coefficient $a_{i,j}$ is $pe_i-e_j$, where $(e_1,...,e_n)$ is the canonical basis of $\ZZ^n$. Similarly, any monomial $m$ is an eigenfunction for this action and write $wt(m)$ for the eigenvalue (which we also call "weight" of $m$). This defines a grading of the $k$-algebra $k[M_n]$ by the monoid $\ZZ^n$ which is coarser than the natural grading by $\NN^{n^2}$. For $\lambda \in \ZZ^n$, denote by $k[M_n]_\lambda$ the subspace of homogeneous polynomials $P\in k[M_n]$ (for the weight $wt$) of weight $\lambda$. We have a decomposition $k[M_n]=\bigoplus_{\lambda \in \ZZ^n} k[M_n]_\lambda$. It is clear that $ k[M_n]_\lambda$ is finite-dimensional.

\begin{lemma}\label{fohom}
Let $f\in H^0(\GF,\Lscr(\lambda))$ be a nonzero section and $f_0\in k[M_n]$ the polynomial attached to $f$ by \Lem~\ref{lemmapol}. Then $f_0\in k[M_n]_\lambda$.
\end{lemma}

\begin{proof}
The function $f_0$ satisfies $f_0(bx \varphi(b^{-1}))=\lambda(b)f_0(x)$ for all $b\in B_L$ and $x\in M_n$. Decompose $f$ as $f=\sum_{\lambda'}P_{\lambda'}$ where $P_{\lambda'}\in k[M_n]_{\lambda'}$. We get $\lambda(b) f_0(x) =\sum_{\lambda'} \lambda'(b) P_{\lambda'}(x)$ for all $b\in B_L$ and all $x\in M_n$. By linear independence of characters, this implies that $P_{\lambda'}=0$ for all $\lambda'\neq \lambda$. This shows the result.
\end{proof}
Define the \emph{polynomial cone} by
\begin{equation}
C_{\rm pol} =\langle e_i-p e_j, \ 1\leq i,j\leq n\rangle.
\end{equation}
Recall that $\langle ... \rangle$ denotes the s-generated cone (\S\ref{coneter}). We proved the inclusion $\langle C_{\rm zip}\rangle \subset  C_{\rm pol}$, therefore:

\begin{proposition}\label{zippol}
One has $\langle C_{\rm zip}\rangle \subset  C_{\rm pol} \cap X^*_{+,I}(T)$.
\end{proposition}

It is easy to see that $C_{\rm pol}$ is s-generated by the vectors $e_i-pe_j$ for  $a\leq i\neq j \leq n$. These vectors are the extremal rays of the cone $C_{\rm pol}$. The above proposition is enough to determine exactly $\langle C_{\zip}\rangle$ in the case $n=2$ (see later). For $n\geq 3$, we must sharpen this result. For any subset $\Sigma\subset \{1,...,n\}^2$, define a monomial cone $C_\Sigma \subset C_{\rm pol}$ by the formula
\begin{equation}
C_{\Sigma} =\langle e_i-p e_j, \ (i,j)\in \Sigma \rangle.
\end{equation}
We say that a matrix $M=(m_{i,j})_{1\leq i,j\leq n}$ has support in a subset $\Sigma\subset \{1,...,n\}^2$ if $m_{i,j}= 0$ for $(i,j)\notin \Sigma$. Define a Zariski closed subset $M_\Sigma\subset M_n$ as the set of matrices with support in $\Sigma$.

\begin{definition}\label{largedef}
A subset $\Sigma\subset \{1,...,n\}^2$ is \emph{large} if there exists a Zariski open subset $U\subset M_n$ such that for any $x\in U$, there exists $b\in B_L$ such that $bx\sigma(b)^{-1}\in M_\Sigma$.
\end{definition}
We have the following result:
\begin{proposition}
Assume $\Sigma\subset \{1,...,n\}^2$ is large. Then $\langle C_{\rm zip} \rangle\subset C_\Sigma$.
\end{proposition}
\begin{proof}
Let $\lambda\in \langle C_{\rm zip} \rangle$. After scaling we may assume that there is a non-zero section $f\in H^0(\GF, \Lscr(\lambda))$. Consider the associated regular function $f_0:M_n\to \AA^1$ as in \Lem~\ref{lemmapol}. This function satisfies $f_0(bx \varphi(b^{-1}))=\lambda(b)f_0(x)$ for all $b\in B_L$ and $x\in M_n$. Since $\Sigma$ is large, it follows that $f_0$ cannot restrict to $0$ on the subvariety $M_\Sigma$ (otherwise it would be zero everywhere). Viewing $f_0$ as a polynomial in the coefficients $(a_{i,j})_{1\leq i,j \leq n}$, the restriction $f_0|_{M_\Sigma}$ is the polynomial obtained from $f_0$ after removing all monomials containing $a_{i,j}$ with $(i,j)\notin \Sigma$. Since this polynomial is nonzero, we deduce that some monomial in $f_0$ has support in $\Sigma$, and since $f_0$ is homogeneous (\Lem~\ref{fohom}), we deduce that $\lambda\in C_\Sigma$.
\end{proof}

There are some obvious restrictions on large subsets. By a dimension argument, we must have $|\Sigma| \geq \frac{n(n-1)}{2}$ for any large subset $\Sigma$. Also, we claim that $(1,n)\in \Sigma$ for any large subset. Indeed, recall that $S_1=e_1-p e_n\in C_{\Sbt}\subset \langle C_{\zip}\rangle$ (see \Lem~\ref{SchubSp}). Since $e_1-p e_n$ is an extremal ray of the cone $C_{\rm pol}$, it is also an extremal ray of any subcone in which it lies. If $\Sigma$ is large, then $e_1-p e_n\in \langle C_{\zip}\rangle\subset C_\Sigma$, so $(1,n)\in \Sigma$. In the rest of the section, we give some subsets $\Sigma$ that are large. For us, the most interesting subset is the following:
\begin{equation}
\Sigma_1:=\{(i,j)\in \{1,...,n\}^2, \ i\geq j \}\cup \{(i,n+1-i), \ 1\leq i \leq n\}.
\end{equation}

\begin{lemma}\label{sigone}
The subset $\Sigma_1$ is large.
\end{lemma}

We will need the following lemma:
\begin{lemma}\label{dimlem}
Let $G$ be an algebraic group over $k$ acting on an irreducible $k$-variety $X$. Let $Y\subset X$ be an irreducible closed subvariety. Assume that there is a nonempty open subset $U\subset Y$ of $Y$ such that for all $y\in U$, the closed subset of $G$
\begin{equation}\label{Hy}
   H(y) := \{g\in G \ | \ g\cdot y \in Y\}
\end{equation}
has dimension $\dim(Y)+\dim(G)-\dim(X)$. Then there is a nonempty open subset $V\subset X$ such that any $x\in V$ is in the $G$-orbit of an element in $Y$.
\end{lemma}

\begin{proof}
Consider the map $\theta:G\times Y \to X$, $(g,y)\mapsto g\cdot y$. Denote by $Z$ the Zariski closure of the image of $\theta$, it is an irreducible closed subvariety of $X$. By Chevalley's theorem the image is a constructible set, i.e a finite union of locally closed subsets, so the image of $\theta$ contains an open subset $W\subset Z$ of $Z$. To prove the lemma it suffices to show that $Z=X$. It is a classic result that there exists an open subset $W_0\subset W$ such that for all $w\in W_0$, the fiber $\theta^{-1}(w)$ has exactly dimension $\dim(Y\times G)-\dim(Z)$. Now consider the open subset $\theta^{-1}(W_0)\subset G\times Y$. By irreducibility of $G\times Y$, it intersects the open subset $G\times U$. Pick any element $(g_0,y_0)\in \theta^{-1}(W_0)\cap (G\times U)$. By definition, $y_0\in U$ and $w_0:=\theta(g_0,y_0)=g_0\cdot y_0\in W_0$. It is straight-forward to see that the fiber $\theta^{-1}(w_0)$ is exactly the set of pairs $(g,(g^{-1}g_0)\cdot y_0)\in G\times X$ with the condition $(g^{-1}g_0)\cdot y_0 \in Y$, in other words $g^{-1}g_0\in H(y_0)$. Hence there is an isomorphism $H(y_0)\to\theta^{-1}(w_0)$ given by the map $h\mapsto g_0h^{-1}$. Combining the assumption on the dimension of $H(y_0)$ with the previously computed dimension of $\theta^{-1}(w_0)$, we obtain $\dim(Z)=\dim(X)$, and hence $Z=X$.
\end{proof}

\begin{proof}[Proof of \Lem~\ref{sigone}]
We use the previous lemma for $X=M_n$, $Y=M_{\Sigma_1}$ and for the group $B_L$ acting on $M_n$ by $b\cdot g=bg\varphi(b)^{-1}$ for all $b\in B_L$ and $g\in M_n$. One has $\dim(Y)+\dim(B_L)-\dim(X)=n+\floor{\frac{n}{2}}$. Let $U\subset M_{\Sigma_1}$ be the subset of matrices $M=(m_{i,j})$ satisfying $m_{i,j}\neq 0$ for all $(i,j)\in \Sigma_1$. We claim that the conditions of \Lem~\ref{dimlem} are satisfied. Define $K_+$ (resp. $K_-$) as the subset of pairs $(i,j)$ in $\{1,...,n\}^2$ such that $i>j$ and $i+j< n+1$ (resp. $i+j>n+1$). Then, we define a subset $H_+$ (resp. $H_-$) of $B_L$ as the set of matrices $m=(m_{i,j})$ in $B_L$ which satisfy $m_{i,j}=0$ for all $(i,j)\in K_+$ (resp. $K_-$). It is easy to check that $H_+$, $H_-$ are subgroups of $B_L$. Their intersection $H=H_+\cap H_-$ is the set of invertible, lower-triangular matrices whose entries which are neither on the diagonal nor the anti-diagonal are zero. This is again a subgroup of $B_L$ and its dimension is $n+\floor{\frac{n}{2}}$. To show the claim, it suffices to show that $H(y)=H$ for all $y\in U$. It is trivial to see that $H\subset H(y)$ for any $y\in U$. Conversely, let $y\in U$ and $b\in H(y)$. By definition, we have $by\varphi(b)^{-1}\in Y$. Write $b=(b_{i,j})$ and $c:=\varphi(b)^{-1}=(c_{i,j})$. By looking at the last column of $byc$, we see that $b$ satisfies $b_{2,1}=b_{3,1}=...=b_{n-1,1}=0$. Similarly, by looking at the first line, we see that $c_{n,2}=c_{n,3}=...=c_{n,n-1}=0$. Then we continue and look at the $n-1$-th column of $byc$. We see that $b_{3,2}=b_{4,2}=...=b_{n-2,2}=0$ and similarly $c_{n-1,3}=c_{n-1,4}=...=c_{n-1,n-2}=0$. Continuing this way, we arrive finally at $b\in H_+$ and $c\in H_-$. So far, we have only used the fact that $b$ and $c$ are lower-triangular. Now, we use specifically that $c=\varphi(b)^{-1}$ to deduce (since $H_-$ is a subgroup of $B_L$), that we also have $\varphi(b)\in H_-$, and hence also $b\in H_-$. This shows that $b\in H$ and terminates the proof.

\end{proof}

\begin{rmk}
The subsets $\{(i,j)\in \{1,...,n\}^2, \ i+j \leq n+1 \}$ and $\{(i,j)\in \{1,...,n\}^2, \ i+j \geq n+1 \}$ are other examples of large subsets (it follows from the density of $B_L w_{0,L} B_L$ in $L$).
\end{rmk}

We can slightly improve the previous result. Define $\Sigma'_1\subset \Sigma_1$ as the subset
\begin{equation}
\Sigma'_1:=\{(i,j)\in \{1,...,n\}^2, \ i > j \}\cup \{(i,n+1-i), \ 1\leq i \leq n\}.
\end{equation}
It is easy to see that $C_{\Sigma_1}=C_{\Sigma'_1}$ (but $\Sigma_1'$ is not large). We deduce:

\begin{corollary}\label{corzipcont}
One has $\langle C_{\zip} \rangle \subset C_{\Sigma'_1}\cap X^*_{+,I}(T)$.
\end{corollary}

We will determine the cone $\langle C_{\zip} \rangle$ in some cases in the next section.

\section{The ring of zip automorphic forms}

So far we have looked at a fixed weight $\lambda\in X^*_{+,I}(T)$. In this section, we define a ring that captures information about all the weights at once, as well as the relations between different global sections.

\subsection{Some general properties}
Let $G$ be a reductive $\FF_p$-group and let $\mu:\GG_{m,k}\to G_k$ be a cocharacter. Choose an $\FF_p$-frame $(B,T,z)$ (\S~\ref{review}). Denote by $\Xcal:=\GZip^\mu$ the attached stack of $G$-zips and $\Ycal:=\GF^\mu$ the stack of zip flags. Define the ring of zip automorphic forms as
\begin{equation}\label{rzip}
R_{\rm zip}:=\bigoplus_{\lambda \in X^*_{+,I}(T)} H^0(\Xcal,\Vscr(\lambda)).
\end{equation}
The additive group $R_{\rm zip}$ inherits a structure of $k$-algebra since $H^0(\Xcal,\Vscr(\lambda))$ identifies with $H^0(\Ycal,\Lscr(\lambda))$ and using formula \eqref{linebndmult}. It is naturally a graded algebra. Note that the cone $C_{\zip}$ gives the actual grading of this algebra. Hence, the problem of determining $R_{\zip}$ can be thought of as a refinement of the questions studied in previous sections.

Retain notation of \S\ref{sec-motiv} for Shimura varieties. As mentioned in the introduction, one can form a similar ring
\begin{equation}
R_{K}:=\bigoplus_{\lambda \in X^*_{+,I}(T)} H^0(S_K, \Vscr(\lambda))
\end{equation}
which can legitimately be called the ring of automorphic forms of level $K$, which explains our terminology for the ring $R_{\zip}$. The map $\zeta:S_K\to \Xcal$ yields an inclusion $R_{\zip}\to R_K$, which is compatible with change of level maps. Hence $R_{K}$ inherits a structure of $R_{\zip}$-algebra and this action commutes with all Hecke operators.

Recall that $H^0(\Xcal,\Vscr(\lambda))$ also identifies with the space of regular functions $f:G_k\to \AA^1$ satisfying $f(agb^{-1})=\lambda(b)f(g)$ for all $g\in G$ and all $(a,b)\in E'$. Hence, we can define a map $\iota : R_{\rm zip} \to k[G]$ by
\begin{equation}
(f_\lambda)_{\lambda} \mapsto \sum_{\lambda\in X^*_{+,I}(T)} f_\lambda
\end{equation}
It is easy to see that $\iota$ is a $k$-algebra homomorphism. By linear independence of characters, this map is injective. Thus, $R_{\rm zip}$ identifies with a sub-algebra of $k[G]$. We view $k[G]$ as a representation of $G\times G$.
Note that $G\times G$ acts on $k[G]$. In particular, $E'$ acts on $k[G]$ by restricting this action. Denote by $R_u(E')$ the unipotent radical of $E'$.

\begin{proposition}\label{Rzipinvar}
The image of $\iota$ is $k[G]^{R_u(E')}$, the subalgebra of $k[G]$ invariant under $R_u(E')$.
\end{proposition}

\begin{proof}
Clearly $\iota(R_{\rm zip})\subset k[G]^{R_u(E')}$. We prove the opposite inclusion. Note that $E'$ acts on $k[G]^{R_u(E')}$ since $R_u(E')$ is a normal subgroup of $E'$. This action factors through the quotient $E'/R_u(E')\simeq T$. Hence, $k[G]^{R_u(E')}$ decomposes as a direct sum of $E'$-eigenfunctions. In other words, any element $f\in k[G]^{R_u(E')}$ can be written as $f=\sum_{\lambda\in X^*(T)}f_\lambda$, where $f_\lambda$ is an $E'$-eigenfunction for the character $\lambda$. This proves the result.
\end{proof}
From now on, we identify implicitly $R_{\zip}$ with its image by $\iota$.

\begin{lemma}
An element $f\in k[G]$ lies in $R_{\zip}$ if and only if $\div(f)$ is $R_u(E')$-invariant.
\end{lemma}

\begin{proof}
One implication is clear. Conversely, let $f\in k[G]$ such that $\div(f)$ is $R_u(E')$-invariant. The condition on the divisor implies that the rational function on $G\times R_u(E')$ defined by $(\epsilon,g)\mapsto \frac{f(\epsilon \cdot g)}{f(g)}$ extends to a non-vanishing function on $G\times R_u(E')$. By \cite[\S1]{Knop-Kraft-Vust-G-variety}, we can write it as a product $\alpha(g)\beta(\epsilon)$, where $\alpha:G\to \GG_m$ and $\beta:R_u(E')\to \GG_m$ are non-vanishing. Evaluating at $\epsilon=1$, we find $\alpha(g)\beta(1)=1$ for all $g\in G$. Hence the above function is $\frac{\beta(\epsilon)}{\beta(1)}$. This is a non-vanishing regular function on the connected group $R_u(E')$ with value $1$ at $1$, so it is a character by \loccit Since $R_u(E')$ is unipotent, it is trivial. Hence $f(\epsilon \cdot g)=f(g)$ for all $\epsilon\in R_u(E')$ and all $g\in G$. 
\end{proof}
In particular, the units of $R_{\zip}$ are exactly the non-vanishing functions $G\to \GG_m$. Any such function can be uniquely written as $a \chi$ where $a\in k^\times$ and $\chi \in X^*(G)$.

\begin{proposition}
Assume that $\Pic(G)=0$. Then $R_{\zip}$ is a UFD.
\end{proposition}

\begin{proof}
By \cite[II,\Prop~6.2]{Hartshorne-Alg-Geom}, the ring $k[G]$ is a UFD. Denote by $\Pcal$ a set of representatives of irreducible elements in $k[G]$. Let $\Pcal' \subset \Pcal$ the set of elements $f\in \Pcal$ such that $\div(f)$ is $R_u(E')$-invariant. Let $f\in R_{\zip}$ be a function, and decompose $f=a f^{e_1}_1...f^{e_n}_n$ in $k[G]$ with $f_i\in \Pcal$ and $e_i>0$ for all $i=1,...,n$. Since $\div(f)$ is $R_u(E')$-invariant and since $R_u(E')$ is connected, each component $\div(f_i)$ is invariant as well. This shows that $R_{\zip}$ is a UFD and $\Pcal'$ is a set of representatives of the irreducible elements. 
\end{proof}

\begin{conjecture}\label{conjRzipfg}
The $k$-algebra $R_{\rm zip}$ is finitely generated.
\end{conjecture}
\Prop~\ref{Rzipinvar} shows that this conjecture is related to Hilbert's 14th problem. Nagata has given a counterexample to Hilbert's original conjecture (\cite{NagataHilbert}).

\subsection{Related algebras}
To simplify, we assume in this section that $P$ is defined over $\FF_p$. We consider the $\mu$-ordinary locus $\Xcal_\mu=[E\backslash U_\mu]\subset \Xcal$ (see \ref{gzipmu}). By replacing $\Xcal$ by $\Xcal_\mu$ in the definition of $R_{\zip}$ \eqref{rzip}, we define a graded $k$-algebra $R_\mu$. One has a natural inclusion $R_{\rm zip}\to R_\mu$ given restriction of sections. This inclusion respects the grading.

\begin{proposition}
The $k$-algebra $R_\mu$ is finitely generated.
\end{proposition}

\begin{proof}
Let $\pi:\Ycal\to \Xcal$ be the natural projection map, and write $\Ycal_\mu:=\pi^{-1}(\Xcal_\mu) = [E'\backslash U_\mu]$ for the $\mu$-ordinary locus of $\Ycal$. There is an isomorphism $\Ycal_\mu \simeq [B_L\backslash L / L(\FF_p)]$ (\Lem~\ref{isommuord} and \S\ref{sec-ordsec}). It follows that $R_\mu$ identifies with the $L(\FF_p)$-invariants of the $k$-algebra $
\bigoplus_{\lambda \in X^*(T)} H^0(B_L\backslash L, \Lscr(\lambda))$. This algebra is finitely generated by standard monomial theory (\cite[\S9.4]{LakshmibaiSMT}). Since $L(\FF_p)$ is a finite group, $R_\mu$ is finitely generated by \cite[\S3]{Humphreys-Hilbert}.
\end{proof}

Let $G_{\max}\subset G$ by the unique open $B\times {}^zB$-orbit in $G$ (recall that $z=w_{0,I}w_0$). One has $G_{\max}=BB^-w_{0,I}$. Recall that one has inclusions $G_{\max}\subset U_\mu \subset G$. We obtain an open substack $\Ycal_{\max}\subset \Ycal$ defined by $\Ycal_{\max}:=[E'\backslash G_{\max}]$. Similarly, we define a $k$-algebra $R_{\max}$ by
\begin{equation}
R_{\max}=\bigoplus_{\lambda\in X^*(T)} H^0(\Ycal_{\max},\Lscr(\lambda)).
\end{equation}
Note that contrary to $R_{\zip}$ and $R_\mu$, the space $H^0(\Ycal_{\max},\Lscr(\lambda))$ may be nonzero even if $\lambda$ is not $L$-dominant. Hence for $R_{\max}$, we want to allow the index set to be all of $X^*(T)$.

Since $w_{0,I}\in G_{\max}$, the stack $\Ycal_{\max}$ is isomorphic to $[E'\backslash B\times {}^zB/S]$ where $S=\Stab_{E'}(w_{0,I})$. It is easy to see that $S$ coincides with the set of pairs $(t,w_{0,I}t^{-1}w_{0,I})$ for $t\in T$. Furthermore, the quotient $E'\backslash B\times {}^zB$ maps isomorphically to $B_L$ by $(a,b)\mapsto \overline{b}$, where $\overline{b}\in L$ denotes the Levi component of $b$ with respect to the decomposition $Q=L R_u(Q)$. We deduce that there is an isomorphism:
\begin{equation}
\Ycal_{\max}\simeq  [B_L/T]
\end{equation}
where $T$ acts on $B_L$ on the right by the rule $b\cdot t:=\varphi(t)^{-1}b w_{0,I}t w_{0,I}$ for all $t\in T$ and all $b\in B_L$. Since this is a quotient by a torus, the ring $k[B_L]$ decomposes as a direct sum of $T$-eigenspaces $k[B_L]=\bigoplus_{\lambda \in X^*(T)}k[B_L]_\lambda$, and via this identification $k[B_L]_\lambda=H^0(\Ycal_{\max},\Lscr(\lambda))$. We obtain the following result:

\begin{proposition}\label{proprmax}
The algebra $R_{\max}$ identifies with $k[B_L]$. There are inclusions
\begin{equation}
R_{\zip}\subset R_\mu \subset k[B_L]
\end{equation}
and all three $k$-algebras have the same field of fractions.
\end{proposition}
This shows that the field of fractions of $R_{\zip}$ is $k(B_L)$. Furthermore, the scheme $\spec(R_{\zip})$ is birational to an affine space. However, we believe that it is not true in general that $R_{\zip}$ is a polynomial ring.

\subsection{The general symplectic case} \label{sec-sympex}
In this section, we assume that $G=Sp(2n)$ and the zip datum is the usual Hodge-type one (\S \ref{section symplectic}). Recall that we identify $L$ with $GL_n$ and $B_L$ is the group of lower-triangular matrices.  Let $(a_{i,j})_{1\leq i,j \leq n}$ be indeterminates and let
$R=k[(a_{i,j})_{i,j}]$. Consider the generic matrix $A=(a_{i,j})\in M_n(R)$. Let $\Delta_i$ be the determinant of the $i\times i$ matrix
\begin{equation}\label{eqDelta}
\Delta_i=\left| \begin{matrix}
a_{1,n+1-i} & a_{1,n+2-i} & \cdots & a_{1,n}\\
a_{2,n+1-i} & a_{2,n+2-i} & \cdots & a_{2,n} \\
\vdots & \vdots & & \vdots \\
a_{i,n+1-i} & a_{i,n+2-i} & \cdots & a_{i,n}
\end{matrix} \right|.
\end{equation}
Denote by $U_{\max}\subset GL_n$ the subset where all $\Delta_i$ for $1\leq i \leq n$ are nonzero. The function $\Delta_i:M_n\to \AA^1$ is $B_L\times B_L$-equivariant. In particular, it is also equivariant with respect to the action of $B_L$ by $\varphi$-conjugation on $M_n$, so it corresponds to a global section over $\Ycal$ (\Lem~\ref{lemmapol}). Actually, the stronger invariance under $B_L\times B_L$ means that $\Delta_i$ comes as a pull-back via the map $\psi:\Ycal\to \Sbt$ \eqref{psimap}. More precisely, recall that we have functions $f_\lambda$ on the stack $\Sbt$ (see \eqref{divflam}). The function $\Delta_i$ is the pull-back of $f_\lambda$ for $\lambda=(0,...,0,-1,...,-1)$, where $-1$ appears $i$ times. The weight of $\Delta_i$ is $wt(\Delta_i)=S_i$ (see \Lem~\ref{SchubSp}).

\begin{lemma}\label{identSp}
One has the following identifications
\begin{enumerate}
\item \label{Spitem1} $R_{\rm zip} =\{f:M_n \to \AA^1 \textrm{ regular, } f(ux\varphi(u)^{-1})=f(x), \ \forall x\in M_n, u\in R_u(B_L)\}$.
\item \label{Spitem2} $R_{\mu}=\{f:GL_n \to \AA^1 \textrm{ regular, } f(ux\varphi(u)^{-1})=f(x), \ \forall x\in GL_n, u\in R_u(B_L)\}$.
\item \label{Spitem3} $R_{\rm max}= \{f:U_{\max} \to \AA^1 \textrm{ regular, } f(ux\varphi(u)^{-1})=f(x), \ \forall x\in U_{\rm max}, u\in R_u(B_L)\}$.
\end{enumerate}
Furthermore, the set of homogeneous elements of weight $\lambda\in X^*(T)$ in each of the graded algebras $R_{\zip}$, $R_\mu$, $R_{\max}$ correspond to those functions $f$ which satisfy $f(bx\varphi(b)^{-1})=\lambda(b) f(x)$ for all $x$ and all $b\in B_L$.
\end{lemma}

\begin{proof}
Assertions \eqref{Spitem1} and \eqref{Spitem2} follow easily from \Cor~\ref{corMn} and the third one follows from the above discussion. The last assertion on homogeneous elements is clear.
\end{proof}

Note that we have an obvious inclusion $R_{\zip}\subset R$. Let $K$ be the field of fractions of $R$. Recall that $A$ denotes the generic matrix $A=(a_{i,j})\in M_n(R)$.

\begin{lemma}\label{exb}
There exists a unique matrix $z\in R_u(B_L)(K)$ such that if we write $z A =(m_{i,j})$, then $m_{i,j}=0$ for all $1\leq i,j\leq n$ such that $i+j>n+1$. Furthermore, $z\in R[\frac{1}{\Delta_1},...,\frac{1}{\Delta_{n-1}}]$.
\end{lemma}

\begin{proof}
Clear.
\end{proof}
Denote again by $\varphi:K\to K$ the map $x\mapsto x^p$ and write also $\varphi:M_n(K)\to M_n(K)$ the map obtained by applying $\varphi$ to the coefficients. Define a matrix $\Gamma=(\gamma_{r,s})_{1\leq r,s\leq n}\in M_n(K)$ by the following:
\begin{equation} \label{matGam}
\Gamma:=zA\varphi(z)^{-1}
\end{equation}
where $z\in R_u(B_L)(K)$ is the matrix given by \Lem~\ref{exb}. It is clear that $\gamma_{r,s}=0$ for $r+s > n+1$. For $(r,s)$ such that $r+s\leq n+1$, the element $\gamma_{r,s}$ lies in $R[\frac{1}{\Delta_1},...,\frac{1}{\Delta_{n-1}}]$. The uniqueness of $z$ in \Lem~\ref{exb} implies that $\gamma_{r,s}$ is automatically $E'$-equivariant, so it is actually in $R_{\zip}[\frac{1}{\Delta_1},...,\frac{1}{\Delta_{n-1}}]$. Its weight is $wt(\gamma_{r,s})=e_{r}-pe_{s}$, where $(e_i)_i$ is the canonical basis of $\ZZ^n$. 

\begin{corollary}
Let $f\in R$ be a polynomial. Then $f\in R_{\zip}$ if and only if $f((a_{i,j})_{i,j})=f((\gamma_{i,j})_{i,j})$ in $K$.
\end{corollary}
By \Prop~\ref{proprmax}, the algebra $R_{\max}$ identifies with $k[B_L]$. By \Lem~\ref{identSp}~\eqref{Spitem3}, it also identifies with certain elements of $R[\frac{1}{\Delta_1},...,\frac{1}{\Delta_n}]$. We have proved the following:
\begin{proposition} \ 
\begin{enumerate}
\item The elements $(\gamma_{r,s})$ for $1\leq r,s \leq n$ such that $r+s\leq n+1$ are algebraically independent, and one has $R_{\max}=k[(\gamma_{r,s}), \frac{1}{\Delta_1},...,\frac{1}{\Delta_{n}}]$.
\item The algebra $R_{\zip}$ identifies with $R\cap k[(\gamma_{r,s})]$ inside $K$.
\end{enumerate}
\end{proposition}
We may cancel the denominator in $\gamma_{r,s}$ by multiplying with a certain product of the $\Delta_i$. We obtain a polynomial $\gamma_{r,s'}$ which lies automatically in $R_{\zip}$. By following precisely the construction of $\Gamma$, one can show that for all $(r,s)$ such that $r+s\leq n+1$, the function
\begin{equation}\label{homog}
\Delta_{r-1}\left(\prod_{m=s}^{n-r} \Delta_m\right)^p \gamma_{r,s}
\end{equation}
is a polynomial (where $\Delta_0=1$ by convention), hence lies in $R_{\zip}$. Note that the above product is not optimal. In general, there is a smaller product of $\Delta_i$ which satisfies this condition. The function \eqref{homog} is homogeneous, so its weight lies in $C_{\zip}$ by definition. We deduce:
\begin{equation}
S_{r-1}+p\sum_{m=s}^{n-r}S_m + e_r-pe_s \in C_{\zip}, \quad \textrm{for all }r+s\leq n+1.
\end{equation}

\subsection{The case $Sp(4)$}
In this section we specialize the discussion to the case $n=2$. One can see immediately that the matrix $z\in M_2(K)$ in \Lem~\ref{exb} is $z=\left( \begin{matrix}
1 & 0 \\ -a_{2,2}/a_{1,2} & 1
\end{matrix} \right)$. We obtain
\begin{equation}
\Gamma=\left(\begin{matrix}
a_{1,1}+ \frac{a_{2,2}^p }{\Delta_1^{p-1} } & \Delta_1 \\ -\frac{\Delta_2}{\Delta_1}  & 0 
\end{matrix}\right).
\end{equation}
This gives us $R_{\max}=k[a_{1,1}\Delta_1^{p-1}+a_{2,2}^p, \Delta_1^{\pm 1},\Delta_2^{\pm 1}]$. To obtain $R_{\zip}$, we need to work slightly more:
\begin{theorem}\label{thmRzipSp4}
For $G=Sp(4)$, the algebra $R_{\zip}$ is $k[a_{1,1}\Delta_1^{p-1}+a_{2,2}^p, \Delta_1,\Delta_2]$. It is isomorphic to a polynomial ring in $3$ variables.
\end{theorem}

\begin{lemma}\label{lemfrac}
Let $R_0\subset R$ be integral domains with the same field of fractions $K$. Let $x\in R_0$ be a nonzero element. Then $R_0=R\cap R_0[\frac{1}{x}]$ if and only if the map $R_0/xR_0 \to R/xR$ is injective.
\end{lemma}

\begin{proof}
The map $R_0/xR_0 \to R/xR$ is injective if and only if $xR_0=R_0\cap xR$. We deduce $x^2R_0=xR_0\cap x^2 R = R_0\cap x^2R$, and by induction $x^nR_0=R_0\cap x^nR$ for all $n\geq 1$. If $y\in R_0[\frac{1}{x}]\cap R$, then there exists $n\geq 1$ such that $x^n y\in x^nR\cap R_0=x^nR_0$, hence $y\in R_0$. 
\end{proof}

\begin{proof}[Proof of Theorem \ref{thmRzipSp4}]
It is clear that $R_0:=k[a_{1,1}\Delta_1^{p-1}+a_{2,2}^p, \Delta_1,\Delta_2]\subset R_{\zip}$. We know that the elements $\Delta_1$, $\Delta_2$ and $\alpha=a_{1,1}\Delta_1^{p-1}+a_{2,2}^p$ are algebraically independent. Furthermore, $R_{\zip}$ coincides with $k[\frac{\alpha}{\Delta_1^{p-1}},\Delta_1,\frac{\Delta_2}{\Delta_1}]\cap R$, which is clearly contained in $R_0[\frac{1}{\Delta_1}]\cap R$. To prove the result, it suffices to show $R_0[\frac{1}{\Delta_1}]\cap R=R_0$. By \Lem~\ref{lemfrac}, this is equivalent to the injectivity of $R_0/\Delta_1 R_0\to R/\Delta_1 R$. Choosing indeterminates $x,y,a,b,c$, this map identifies with: $k[x,y]\to k[a,b,c], \ x\mapsto c^p, \ y\mapsto ac$, which is clearly injective.
\end{proof}
Note that $\alpha=a_{1,1}\Delta_1^{p-1}+a_{2,2}^p$ is homogeneous of weight $wt(\alpha)=(0,-p(p-1))$.

\begin{corollary}\label{corSp4}
For $G=Sp(4)$, one has $C_{\zip}=\NN (1,-p)+\NN(1-p,1-p)+\NN(0,-p(p-1))$.
\end{corollary}
\begin{proof}
By \Th~\ref{thmRzipSp4}, any element $f\in R_{\zip}$ is of the form $f=P(\alpha,\Delta_1,\Delta_2)$ where $P$ is a polynomial in $3$ variables. For $f$ homogeneous of weight $\lambda\in X_{+,I}^*(T)$, it follows that $\lambda$ is a linear combination with positive integer coefficients of the weights of $\alpha,\Delta_1,\Delta_2$. 
\end{proof}

\begin{corollary}\label{cor2Sp4}
For $G=Sp(4)$, one has $\langle C_{\zip} \rangle = \langle (1,-p), (-1,-1) \rangle$.
\end{corollary}

The Schubert cone, the highest weight cone and the Griffith-Schmidt cone are given by
\begin{align*}
C_{\Sbt}&=\NN(1,-p)+\NN(1-p,1-p).\\
C_{\rm hw}&=\{(a_1,a_2), \ pa_1+a_2\leq 0\}.\\
C_{\GS}&=\langle (-1,-1),(0,-1) \rangle.
\end{align*}
Hence, the saturated cones of $C_{\Sbt}$ and $C_{\zip}$ coincide with $C_{\rm hw}$. Note that \Prop~\ref{zippol} gives $\langle C_{\zip} \rangle\subset C_{\rm pol}\cap X_{+,I}^*(T)=\langle (-1,-1),(-1,p) \rangle$. Therefore, we could have easily deduced $\langle C_{\zip} \rangle = \langle C_{\Sbt} \rangle = C_{\rm hw}=\langle (-1,-1),(-1,p) \rangle$ without determining $R_{\zip}$ and $C_{\zip}$ first.

\subsection{The case $Sp(6)$.} \label{sec-sp6} Now we take $n=3$. We will see that computations get quickly very complicated. The cone $\langle C_{\zip}\rangle$ can easily be computed using previous results. First, the Schubert cone, the Griffith-Schmidt cone and the highest weight cones are given by:
\begin{align*}
C_{\Sbt}&=\NN(1,0-p)+\NN(1,1-p,-p)+\NN(1-p,1-p,1-p).\\
C_{\GS}&=\{(a_1,a_2,a_3), \ 0\geq a_1\geq a_2\geq a_3\}..\\
C_{\rm hw}&=\{(a_1,a_2,a_3), \ p^2a_1+pa_2+a_3\leq 0\}.
\end{align*}
By \Lem~\ref{SchubSp}, the cone $C_{\Sbt}$ is generated by the weights
\begin{equation} \label{conept1}
S_1=(1,0,-p), \ S_2=(1,p-1,-p), \ S_3=\eta_\omega
\end{equation}
The Griffith-Schmidt cone is s-generated by the weights
\begin{equation}\label{conept2}
(0,0,-1), \ (0,-1,-1), \ \eta_\omega.
\end{equation}
By \Lem~\ref{hwgen}, the cone $C_{\rm hw}$ is s-generated by the weights
\begin{equation}\label{conept3}
\eta_1=(p+1,-p^2,-p^2), \ \eta_2=(1,1,-p(p+1)), \ \eta_\omega.
\end{equation}

\begin{proposition}\label{propconeSp6} \ 
\begin{enumerate}
\item One has $\langle C_{\zip} \rangle = \langle \eta_1,\eta_2,\eta_\omega,S_1 \rangle$.
\item The cone $\langle C_{\zip} \rangle$ is the set of $(a_1,a_2,a_3)\in \ZZ^3$ satisfying the inequalities
\begin{align}
& a_1\geq a_2 \geq a_3 \label{Sp6eq1} \\
& p^2 a_1+pa_3+a_2\leq 0 \label{Sp6eq2} \\
& p^2 a_2+pa_1+a_3\leq 0 \label{Sp6eq3}
\end{align}
\end{enumerate}
\end{proposition}
\begin{proof}
By \Cor~\ref{corzipcont}, we have $\langle C_{\zip} \rangle \subset C_{\Sigma'_1}\cap X_{+,I}^*(T)$. It is easy to see that $C_{\Sigma'_1}\cap X_{+,I}^*(T)$ is s-generated by $\eta_1,\eta_2,\eta_\omega,S_1$ and is given by the equations \eqref{Sp6eq1}, \eqref{Sp6eq2}, \eqref{Sp6eq3}.
\end{proof}

Here is a representation of the cones. As it is 3-dimensional, we take a slice of the cone, perpendicularly to the axis generated by $\eta_\omega$ (which hence appears as a dot at the center of the picture). The exterior shape which looks like a triangle with corners cut off is the trace of the polynomial cone on the slice. The two dotted half-lines that intersect at $\eta_\omega$ are the traces of the two hyperplanes that define $X_{+,I}^*(T)$. This picture represents all four cones $C_{\pol}$, $C_{\zip}$, $C_{\Schub}$ and $C_{\GS}$. Please refer to \eqref{conept1}, \eqref{conept2}, \eqref{conept3} for the points marked on the picture.

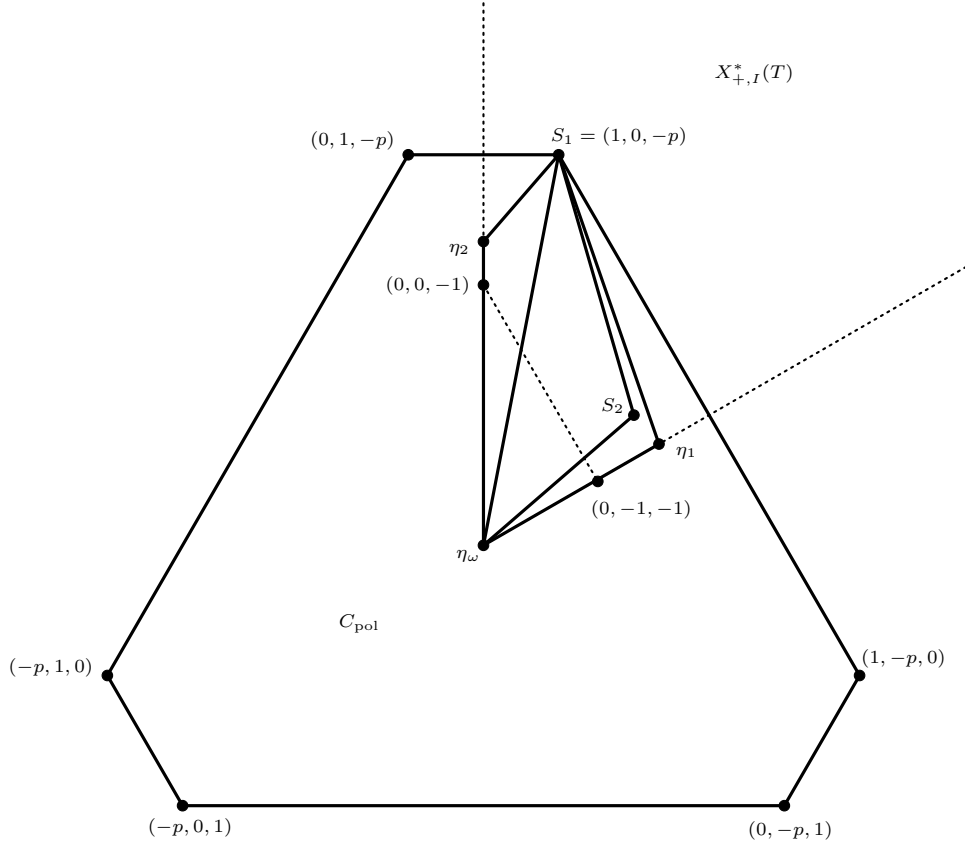
\begin{center}
\begin{tikzpicture}[line cap=round,line join=round,>=triangle 45,x=2cm,y=2cm]
\clip(-0.191393831107175,-0.3946952381920047) rectangle (6.2,5.34706282321349);
\draw [line width=1.2pt] (2.5,4.330127018922194)-- (3.5,4.330127018922193);\draw [line width=1.2pt] (0.5,0.8660254037844388)-- (1,0);
\draw [line width=1.2pt] (5.5,0.8660254037844359)-- (5,0);
\draw [line width=1.2pt] (2.5,4.330127018922194)-- (0.5,0.8660254037844388);
\draw [line width=1.2pt] (3.5,4.330127018922193)-- (5.5,0.8660254037844359);
\draw [line width=1.2pt] (5,0)-- (1,0);
\draw [line width=1.2pt] (3.5,4.330127018922193)-- (4,2.5980762113533147);
\draw [line width=0.8pt,dotted,domain=3:8.897024033046604] plot(\x,{(-0--0.8660254037844382*\x)/1.5});

\draw [line width=0.8pt,dotted] (3,1.732050807568876) -- (3,6.34706282321349);
\draw [line width=0.8pt,dotted] (3,3.464) -- (3.76,2.157);

\draw [fill=black] (5,0) circle (2pt);
\draw [fill=black] (1,0) circle (2pt);

\draw [fill=black] (3,3.464) circle (2pt);
\draw [fill=black] (3.76,2.157) circle (2pt);

\draw [line width=1.2pt] (3.5,4.330127018922193)-- (3,3.752776749732569);
\draw [line width=1.2pt] (3.5,4.330127018922193)-- (4.166666666666667,2.405626121623439);
\draw [line width=1.2pt] (3,1.732050807568876)-- (4,2.5980762113533147);
\draw [line width=1.2pt] (3,1.732050807568876)-- (4.166666666666667,2.405626121623439);
\draw [line width=1.2pt] (3,1.732050807568876)-- (3,3.752776749732569);
\draw [line width=1.2pt] (3,1.732050807568876)-- (3.5,4.330127018922193);

\begin{scriptsize}
\draw [fill=black] (1,-23.832) circle (2pt);
\draw[color=black] (1.05,-0.142) node {$(-p,0,1)$};
\draw [fill=black] (5,-18.916498697301474) circle (2pt);
\draw[color=black] (5.04,-0.161) node {$(0,-p,1)$};
\draw [fill=black] (2.5,4.330) circle (2pt);
\draw[color=black] (2.13,4.428) node {$(0,1,-p)$};
\draw [fill=black] (3.5,4.330) circle (2pt);
\draw[color=black] (3.9,4.45) node {$S_1=(1,0,-p)$};
\draw [fill=black] (0.5,0.8660254037844388) circle (2pt);
\draw[color=black] (0.124668,0.920590) node {$(-p,1,0)$};
\draw [fill=black] (5.5,0.866) circle (2pt);
\draw[color=black] (5.792,0.9802) node {$(1,-p,0)$};
\draw [fill=black] (3,1.732) circle (2pt);
\draw[color=black] (2.9,1.661863016367702) node {$\eta_\omega$};
\draw[color=black] (2.1754267313941136,1.2141333900979427) node {$C_{\rm pol}$};
\draw[color=black] (4.8,4.859198753441583) node {$X_{+,I}^{*}(T)$};

\draw [fill=black] (4,2.5980762113533147) circle (2pt);
\draw[color=black] (3.86,2.66) node {$S_2$};
\draw [fill=black] (4.166666666666667,2.405626121623439) circle (2pt);
\draw[color=black] (4.35,2.35) node {$\eta_1$};
\draw [fill=black] (3,3.752776749732569) circle (2pt);
\draw[color=black] (2.84,3.7) node {$\eta_2$};

\draw[color=black] (2.63,3.464) node {$(0,0,-1)$};
\draw[color=black] (4.05,1.97) node {$(0,-1,-1)$};

\end{scriptsize}
\end{tikzpicture}
\end{center}

Let us try to compute $R_{\zip}$. We use the following notation: For $1\leq i,j\leq$ let $\delta_{i,j}$ be the minor of $A$ obtained by removing the $i$-th row and $j$-th column. First, the matrix $\Gamma$ defined in \eqref{matGam} is given by:
\begin{equation}
\Gamma=\left(
\begin{matrix}
\frac{\epsilon}{\Delta_1^p}&\frac{f_1}{\Delta_2^p}&\Delta_1 \\[10pt]
\frac{f_2}{\Delta_1^{p+1}} & \frac{\Delta_2}{\Delta_1}&0\\[10pt]
\frac{\Delta_3}{\Delta_2}&0&0
\end{matrix}
\right)
\end{equation}
where $\epsilon, f_1, f_2 \in R_{\zip}$ are defined as follows:
\begin{align}
\epsilon & =a_{1,1}a_{1,3}^p +a_{1,2}a_{2,3}^p+a_{1,3}a_{3,3}^p \\
f_1 & =a_{1,2}\Delta_2^p+\Delta_1 \delta_{2,1}^p \\
f_2 & =\Delta_1^p \delta_{3,2}-\Delta_2 a_{2,3}^p.
\end{align}

\begin{rmk}\label{rmketaf}
The weight of $f_1$ is $\eta_1$ and the weight of $f_2$ is $\eta_2$ (defined in \eqref{conept3}). These weights lie on the hyperplane $\sum_{i} p^{n-i}a_i=0$. Actually, one can check that $f_1,f_2$ are obtained by the association \eqref{assoclambda}. The weight of $\epsilon$ is $(1,0,-p^2)$. This section is not attached to a highest weight.
\end{rmk}

We obtain a subalgebra $R_0:=k[\epsilon,f_1,f_2,\Delta_1,\Delta_2,\Delta_3]\subset R_{\zip}$, by canceling out the denominators in $\Gamma$. However, contrary to the case $n=2$, this naive subalgebra is strictly contained in $R_{\zip}$ because the map $H:R_0/\Delta_1 R_0 \to R/\Delta_1 R$ is not injective. For example, the following elements are in $R_{\zip}$, but not in $R_0$:



\begin{equation}\label{R1eq}
\theta =\frac{\Delta_2^{p+1}\epsilon + f_1 f_2}{\Delta_1^{p+1}}, \ 
\rho  = \frac{\Delta_2 f_2^{p-1}-\epsilon^p}{\Delta_1^p}, \ 
\tau =\frac{\Delta_2^{p^2}-f_1^{p-1}\epsilon}{\Delta_1^p}.
\end{equation}
Define $R_1$ as the $k$-algebra $k[\epsilon,f_1,f_2,\Delta_1,\Delta_2,\Delta_3,\theta,\rho,\tau]$. It is isomorphic to a polynomial ring in $9$ variables divided by 3 equations (given in an obvious way by \eqref{R1eq}). We believe that $R_1 = R_{\zip}$ but we did not check it.

\section{Modular interpretation} \label{sec-ex}
We continue to work with the group $G=Sp(2n)$ endowed with its usual Hodge-type zip datum. We write $\Xcal=\GZip$ and $\Ycal=\GF$.

\subsection{Dictionary}\label{dicsec}
Recall that the stack of $G$-zips has the following modular interpretation. Let $S$ be a scheme over $\FF_p$. Then a morphism of stacks $S\to \Xcal$ is the same as a tuple $(\Mcal,\langle - , - \rangle,\Omega,F,V)$ where 
\begin{enumerate}
\item $\Mcal$ is a locally free $\Ocal_S$-module of rank $2n$,
\item $\langle -,-\rangle :\Mcal\times \Mcal \to \Ocal_S$ is a perfect $\Ocal_S$-linear pairing,
\item $\Omega\subset \Mcal$ is a locally free isotropic $\Ocal_S$-module of rank $n$, such that $\Mcal/\Omega$ is locally free,
\item $F:\Mcal^{(p)}\to \Mcal$ is an $\Ocal_S$-linear morphism,
\item $V:\Mcal\to \Mcal^{(p)}$ is an $\Ocal_S$-linear morphism,
\end{enumerate}
satisfying the relations:
\begin{align*}
&\Im(V)=\Ker(F)=\Omega^{(p)}\\
&\Im(F)=\Ker(V),\\
&\langle Fx,y\rangle = \langle x,Vy\rangle^{(p)}, \textrm{ for all }x\in \Mcal^{(p)},y\in \Mcal.
\end{align*}
In the above, the sheaf $\Mcal^{(p)}$ is the pull-back of $\Mcal$ via the absolute Frobenius morphism $F_S:S\to S$. The pairing $\langle -,-\rangle^{(p)}$ is the one induced from $\langle -,-\rangle$ on $\Mcal^{(p)}$.

Similarly, there is a moduli interpretation for the stack of (full) $G$-zip flags. A morphism $S\to \Ycal$ is the same as a tuple $(\Mcal,\langle - , - \rangle,\Omega,F,V)$ as above, together with a full flag
\begin{equation}
0=\Fcal_0\subset \Fcal_1\subset ... \subset \Fcal_n=\Omega
\end{equation}
where $\Fcal_i$ is a locally free $\Ocal_S$-module of rank $i$ and the quotients $\Omega/\Fcal_{i}$ are locally free for all $i=1,...,n$. The natural projection $\pi:\Ycal\to \Xcal$ is given by forgetting the flag $\Fcal_\bullet$. This illustrates that the fibers of $\pi$ are flag varieties. By assumption, the quotient $\Lscr_i=\Fcal_i/\Fcal_{i-1}$ is a line bundle for each $i=1,...,n$. For a $n$-tuple $(a_1,...,a_n)\in \ZZ^n$, define
\begin{equation}\label{Ladef}
\Lscr(a_1,...,a_n)=\Lscr_1^{-a_1}\otimes ... \otimes \Lscr_n^{-a_n}.
\end{equation}
Via the usual identification $X^*(T)=\ZZ^n$, the line bundles $\Lscr(a_1,...,a_n)$ correspond to the line bundles $\Lscr(\lambda)$ studied throughout the previous sections of this paper (therefore the notation is unambiguous). The goal of this section is to give a modular interpretation of some of the sections defined in \S\ref{sec-sympex}.

\begin{rmk}\label{rmkSpeck}
Let $M$ be a $k$-vector space equipped with a perfect pairing $\langle -,-\rangle:M\times M\to k$, a $\sigma$-linear endomorphism $F_0:M\to M$ and a $\sigma^{-1}$-linear endomorphism $V_0:M\to M$. Assume that $\langle F_0 x,y\rangle =\sigma\langle x,V_0 y \rangle$. Then we obtain a $G$-zip over $k$ by defining $F:M\otimes_\sigma k \to M$, $F(x\otimes \lambda)=\lambda F_0(x)$ and $V:M \to M\otimes_\sigma k$, $F(x)=F_0(x)\otimes 1$, and taking $\Omega:=\Im(V_0)$. We call $(M,F_0,V_0,\langle -,-\rangle)$ a symplectic Dieudonn\'{e} space over $k$. This defines an equivalence of categories between $G$-zips over  $\spec(k)$ and symplectic Dieudonn\'{e} spaces. Similarly, we can define a filtered symplectic Dieudonn\'{e} space by adding to $M$ a filtration  $0\subset \Fcal_1 \subset ... \subset \Fcal_n=\Omega$. These objects correspond to $G$-zip flags.
\end{rmk}

Recall that an algebraic representation $(V,\rho)$ of $E$ gives a vector bundle $\Vscr(\rho)$ on $\Xcal=\GZip$ (see \S\ref{sec-vector-bundles-gzipz}). We only consider representations of $E$ that arise from a representation of $P$ via the first projection $E\to P$. We now give the representations that correspond to the vector bundles $\Omega$ and $\Mcal$. If $\rho:P\to GL(V)$ is a representation, we denote by $\rho^\vee$ its dual. It is clear that $\Vscr(\rho^\vee)\simeq\Vscr(\rho)^\vee$. Denote by $\rho_L:P\to GL_n$ given by the composition of the natural projection $P\to L$ composed by the identification $L\simeq GL_n$. One can show that we have the following formula:
\begin{equation}
\Vscr(\rho^\vee_L)\simeq \Omega.
\end{equation}
Hence we set $\rho_\Omega:=\rho_L^\vee$. Similarly, denote by $\Std:Sp(2n)\to GL_{2n}$ the natural inclusion (the standard representation). Denote by $\Std|_P$ its restriction to $P$. Then we have
\begin{equation}
\Vscr(\Std|_P^\vee)\simeq \Mcal.
\end{equation}
Let $V$ be an $\FF_p$-vector space and $\rho:P\to GL(V)$ a representation of $P$ defined over $\FF_p$. Then we define $\rho^{(p)}$ as the Frobenius-twist of $\rho$. It is clear that $\Vscr(\rho^{(p)})\simeq \Vscr(\rho)^{(p)}$.

\subsection{The Schubert sections}

Recall that we defined for $1\leq i \leq n$ a section $\Delta_i$ of weight $S_i$ on $\Ycal$ in \eqref{eqDelta}. We give its modular interpretation. Consider the map $f_i$ of vector bundles obtained by the following composition
\begin{equation}
\Fcal_i \xrightarrow{V} \Omega^{(p)} \to \Omega^{(p)}/\Fcal_{n-i}^{(p)}
\end{equation}
Then taking the determinant of $f_i$, we obtain a map $\Delta_i:=\det(f_i) : \det(\Fcal_i)\to \det(\Omega/\Fcal_{n-i})^p$. One has clearly $\det(\Fcal_i)=\Lscr_1\otimes ...\otimes \Lscr_i$ and $\det(\Omega/\Fcal_j)=\Lscr_{j+1}\otimes ... \otimes \Lscr_n$. We deduce that $\Delta_i$ identifies with a map
\begin{equation}
\Delta_i : \Lscr_1\otimes ...\otimes \Lscr_i \to (\Lscr_{n-i+1}\otimes ...\otimes \Lscr_n)^p
\end{equation}
This is the same as a global section over $S$ of the line bundle $(\Lscr_{n-i+1}\otimes ...\otimes \Lscr_n)^p\otimes (\Lscr_1\otimes ...\otimes \Lscr_i )^{-1}=\Lscr(S_i)$. We claim (proof omitted) that this section coincides with the section $\Delta_i$ defined in \eqref{eqDelta}. In the case $i=n$, we obtain the classical Hasse invariant, as the determinant of the map $V:\Omega \to \Omega^{(p)}$.

For $S=\spec(k)$, let $(M,\Fcal_\bullet,F_0,V_0,\langle-,-\rangle)$ be a filtered symplectic Dieudonn\'{e} space over $k$, and let $z:\spec(k)\to \GZip$ be the corresponding point. Then one has an equivalence
\begin{equation}
\Delta_i(z)\neq 0 \Longleftrightarrow V_0(\Fcal_i)\oplus \Fcal_{n-i}=M.
\end{equation}

\subsection{The case $n=2$}
For $n=2$, we know that the algebra $R_{\zip}$ is generated by $\Delta_1,\Delta_2,\alpha$, where $\alpha=a_{1,1}\Delta_1^{p-1}+a_{2,2}^p$ (\Th~\ref{thmRzipSp4}). We already know the modular interpretation of $\Delta_1$ and $\Delta_2$. We determine the modular interpretation of $\alpha$. We leave certain verifications to the reader.

To find the modular interpretation, we look closer at the formula for $\alpha$. Note that $\Delta_1 \alpha=a_{1,1}\Delta_1^{p}+\Delta_1 a_{2,2}^p$ can be seen as the scalar product of two vector-valued functions
\begin{equation}
\Delta_1\alpha=\left(\begin{matrix}
a_{1,1} \\ a_{1,2}
\end{matrix}\right)^t \left(\begin{matrix}
a_{1,2} \\ a_{2,2}
\end{matrix}\right)^{(p)}
\end{equation}
One can easily check that the maps $\Xi_1, \Xi_2:M_2 \to \AA^2$ defined by $\Xi_1:(a_{i,j}) \mapsto \left(\begin{matrix}
a_{1,1} \\ a_{1,2}
\end{matrix}\right)$ and $\Xi_2: (a_{i,j})\mapsto \left(\begin{matrix}
a_{1,2} \\ a_{2,2}
\end{matrix}\right)$ satisfy the following relations:
\begin{align}
&\Xi_1\left(M A \varphi(M)^{-1}\right)=x \ \rho_\Omega^{(p)}(M^{-1}) \ \Xi_1(A) & \forall M=\left( \begin{matrix}
x & 0\\ y& z
\end{matrix}
\right), \  & \forall A=\left( \begin{matrix}
a_{1,1} & a_{1,2}\\a_{2,1} & a_{2,2}
\end{matrix} \right) \\
&\Xi_2\left(M A \varphi(M)^{-1}\right)=\frac{1}{z^p} \rho_{\Omega}^{\vee}(M) \ \Xi_2(A) & \forall M=\left( \begin{matrix}
x & 0\\ y& z
\end{matrix}
\right), \ & \forall A=\left( \begin{matrix}
a_{1,1} & a_{1,2}\\a_{2,1} & a_{2,2}
\end{matrix} \right).
\end{align}

By the dictionary of vector bundles \S\ref{dicsec}, we deduce that $\Xi_1$ is a global section over $\Ycal$ of the vector bundle $\Lscr^{-1}_1 \otimes \Omega^{(p)}$. Similarly, $\Xi_2$ is a global section of the vector bundle $\Lscr_2^p \otimes \Omega^\vee$. From this observation, we can find the modular interpretation of $\Xi_1$ and $\Xi_2$. The map $V:\Lscr_1\to \Omega^{(p)}$ gives a section of $\Lscr^{-1}_1 \otimes \Omega^{(p)}$ over $\Ycal$ which corresponds to $\Xi_1$. Similarly, the map $V:\Omega\to \Omega^{(p)}$ followed by the projection $\Omega^{(p)}\to \Lscr_2^p$ gives a map $\Omega\to \Lscr_2^p$, hence a global section of $\Lscr_2^p \otimes \Omega^\vee$. This section corresponds to $\Xi_2$. Since $\Delta_1 \alpha = \Xi_1 \circ \Xi_2^{(p)}$, we can obtain $\Delta \alpha$ by composing the map $V:\Lscr_1 \to \Omega^{(p)}$ with the Frobenius-twist of $\Omega\to \Lscr_2^p$. This gives a map $\Lscr_1 \to \Lscr_2^{p^2}$. The result is divisible by $\Delta_1$. Therefore, we deduce the following result:

\begin{proposition}
The modular interpretation of $\alpha$ is the following. It is the unique map $\alpha:\Lscr_2^p\to \Lscr_2^{p^2}$ which is functorial in $S$ and makes the following diagram commute:
$$\xymatrix@M=7pt{
\Lscr_1 \ar[r]^V \ar[rd]_{V} & \Omega^{(p)} \ar[r]^{V^{(p)}} & \Omega^{(p^2)} \ar[r] & \Lscr_2^{p^2}  \\
& \Omega^{(p)} \ar[r] &\Lscr_2^{p} \ar@{-->}[ru]_{\alpha} &
}$$
\end{proposition}
Without the group-theoretical interpretation of $\alpha$, it would be very difficult to show that there exists such a factorization.


\subsection{Highest weight sections}

In the case $n=3$, we give a moduli interpretation for the sections $f_1,f_2$ attached to the highest weight of $V(\lambda)$ for the two extremal weights $\eta_1=(p+1,-p^2,-p^2)$ and $\eta_2=(1,1,-p(p+1))$, respectively (\Rmk~\ref{rmketaf}). For any other weight $\lambda \in C_{\rm hw}$, the corresponding function is obtained by taking a product of powers of $f_1,f_2$ and the Hasse invariant $\Delta_3$, so it suffices to understand these sections.

The modular interpretation of $f_1$ is the following. Consider the map defined by the composition
\begin{equation}
\Lscr_1^p \otimes \Lscr_1 \xrightarrow{F\otimes F^2} \Omega^{(p^2)}\otimes \Omega^{(p^2)} \rightarrow \bigwedge^2\Omega^{(p^2)} \rightarrow \bigwedge^2(\Omega/\Fcal_1)^{(p^2)}=(\Lscr_2\otimes \Lscr_3)^{p^2}.
\end{equation}
This induces a global section of $\Lscr(\eta_1)$, and we claim (without proof) that it is $f_1$. For any filtered Dieudonn\'{e} space $(M,\Fcal_\bullet,F_0,V_0,\langle-,-\rangle)$ with corresponding point $z:\spec(k)\to \GZip$, there is an equivalence: 
\begin{equation}
f_1(z)\neq 0 \Longleftrightarrow \Fcal_1 \oplus V_0(\Fcal_1)\oplus V_0^{2}(\Fcal_1)=M.
\end{equation}

The modular interpretation of $f_2$ is more involved. First, consider the natural map $\iden\otimes pr:\Fcal_2\otimes \Omega \to \Fcal_2\otimes \Lscr_3$, where $pr:\Omega\to \Lscr_3$ is the canonical projection. We claim that there is a unique map $\gamma:\bigwedge^2\Omega \to \Fcal_2\otimes \Lscr_3$ such that the following diagram commutes
\begin{equation}
\xymatrix@1@M=6pt{
\bigwedge^2\Omega \ar[r]^-{\gamma}  & \Fcal_2\otimes \Lscr_3 \\
\Fcal_2\otimes \Omega \ar[u]^m \ar[ru]_{\iden\otimes pr} &
}
\end{equation}
where $m$ is the natural map $x\otimes y\mapsto x\wedge y$. This follows simply from the surjectivity of $m$ and the fact that $\Ker(m)$ is generated by the elements $x\otimes x$ with $x\in \Fcal_2$, so $\Ker(m)\subset \Fcal_2\otimes \Fcal_2=\Ker(\iden\otimes pr)$. Now, the map $F:\Fcal_2 \to \Omega^{(p)}$ induces a map $\bigwedge^2 F:\bigwedge^2\Fcal_2 \to \bigwedge^2\Omega^{(p)}$. Hence we may define the composition
\begin{equation}
\Lscr_1 \otimes \Lscr_2 = \bigwedge^2 \Fcal_2 \xrightarrow{\bigwedge^2 F} \bigwedge^2 \Omega^{(p)} \xrightarrow{\gamma^{(p)}}\Fcal_{2}^{(p)}\otimes \Lscr^p_3 \xrightarrow{F\otimes \iden} \Lscr^{p^2}_3\otimes \Lscr^p_3.
\end{equation}
This gives a section of $\Lscr(\eta_2)$. We claim (without proof) that it corresponds to $f_2$.

\bibliographystyle{plain}
\bibliography{biblio_overleaf}

\end{document}